\documentclass[11pt]{article}
\usepackage{graphicx}
\usepackage{amsmath}
\usepackage{amssymb}
\usepackage{amsthm}
\usepackage{mathrsfs}
\usepackage{float}

\newtheorem{dfn}{Definition}[section]
\newtheorem{thm}[dfn]{Theorem}
\newtheorem{prop}[dfn]{Proposition}
\newtheorem{lem}[dfn]{Lemma}
\newtheorem{cor}[dfn]{Corollary}
\newtheorem{rem}[dfn]{Remark}
\newtheorem{car}[dfn]{Computer Assisted Result}
\newtheorem{ass}[dfn]{Assumption}

\newtheorem{alg}{Algorithm}

\newcommand{\Inv}{{\rm Inv}}
\newcommand{\Int}{{\rm int}}

\newcommand{\diag}{{\rm diag}}

\newcommand{\exit}{{\rm exit}}
\newcommand{\ent}{{\rm ent}}

\numberwithin{equation}{section}

\begin{document}

\title{Rigorous numerics of finite-time singularities in dynamical systems - methodology and applications}

\author{Kaname Matsue\thanks{Institute of Mathematics for Industry, Kyushu University, Fukuoka 819-0395, Japan {\tt kmatsue@imi.kyushu-u.ac.jp}} $^{,}$ \footnote{International Institute for Carbon-Neutral Energy Research (WPI-I$^2$CNER), Kyushu University, Fukuoka 819-0395, Japan}
}
\maketitle

\begin{abstract}
This paper aims at providing rigorous numerical computation procedure for finite-time singularities in dynamical systems.
Combination of time-scale desingularization as well as Lyapunov functions validation on stable manifolds of invariant sets for desingularized vector fields with standard integration procedure for ordinary differential equations give us validated trajectories of dynamical systems involving finite-time singularities.
Our focus includes finite-time extinction, traveling wave solutions with half-line or compact support, and singular canards in fast-slow systems, including rigorous validations of enclosures of extinction, finite-passage times or size of supports for compactons.
Such validated solutions lead to a plenty of composite wave solutions for degenerate parabolic equations, for example, with concrete information of profiles and evolutions.
The present procedure also provides a universal aspect of finite-time singularities with rigorous numerics, combining with rigorous numerics of blow-up solutions in preceding works. 
\end{abstract}


\section{Introduction}
Solutions of (time-evolutionary) differential equations have great natures depending on systems.
For understanding the structure of solutions, standard approaches begin with studying well-posedness of systems, for example; namely existence and uniqueness of solutions and their smooth dependence on initial conditions in the case of initial value problems.
The next step will be studies of {\em asymptotic} nature of systems like stability, or bifurcation of invariant objects like equilibria or periodic trajectories.
Understanding of asymptotics will lead to {\em global nature} of systems; namely solution structure of differential equations.
However, there are often cases that dynamic events which determine or drastically change asymptotic behavior, and which sometimes violate the (time-)global existence of solutions occur {\em in finite times}.
Examples for such events we suppose here are blow-ups (e.g., \cite{FM2002, K1992, Mat}), finite-time extinction (e.g., \cite{DV1991, LPB2004, MS2001}), compactons (e.g., \cite{RH1993}), quenching (e.g., \cite{K1992}) and canards (e.g., \cite{DGKKOW2012, KS2001, KW2010, SW2001}), which we shall call events involving {\em finite-time singularities} throughout this paper.
The fundamental questions involving finite-time singularities are, for a given system, {\em whether such events occur} and, if so, {\em when, where and how these events occur}\footnote{
In the case of blow-up solutions, for example, \lq\lq whether" question involves the existence, \lq\lq when" question involves the blow-up time, \lq\lq where" question involves the blow-up set and \lq\lq how" question involves the blow-up rate.
}.
There are plenty of researches for studying these features, many of which restrict applications of results to specific systems with slight generalizations, or study {\em assuming the existence of such events}. 
Needless to say, existence of these finite-time singularities for given systems remains non-trivial in general, and a comprehensive validation strategy is necessary to answer this fundamental question for applying preceding knowledge to general systems.
On the other hand, there are several studies discussing the relationship among finite-time singularities mentioning here (e.g., \cite{DV1991, MS2001}), and we believe that there is a universal mechanism for studying these finite-time singularities.
\par
From the viewpoint of numerical computations, detection of these singularities is a very tough question in general.
For example, the most difficulty for detecting blow-up solutions will be the treatment of the infinity, and that for finite-time extinction will be to answer which the extinction occurs in finite time or infinite time, which will be impossible to treat with direct numerical simulations in principle.
Nevertheless, numerical treatments play key roles in applications of our knowledge to studying solution structures in {\em various} and {\em concretely given} systems.
\par
\bigskip
Our main aim in this paper is to provide {\em rigorous numerical computation} methodology for validating solutions involving finite-time singularities as well as revealing a common mechanism for detecting these singularities.
In the present paper we pick up {\em finite-time extinction}, {\em compacton traveling waves} for degenerate diffusion systems and {\em canard solutions} for fast-slow systems\footnote{
As we mention later, rigorous numerics of blow-up solutions is already discussed in several preceding papers \cite{TMSTMO, MT2017}.
}.
{\em Rigorous numerics} we mention here is a series of numerical computation methodology based on interval arithmetic (or general affine arithmeric) on computer.
Interval arithmetic enables us to compute enclosures where mathematically correct objects are contained in the phase space.
These enclosures give {\em explicit} error bounds of {\em rigorous} solutions. 
In dynamical systems, there are many applications to validations of global trajectories, stability of invariant sets, parameter ranges containing parameter values where the dynamical bifurcations occur with computer assistance (e.g. \cite{CZ2015, CZ2016, Mat2, MHY2016, SZ2013, TMSTMO, W, ZM2001}).
We point our interest at concrete profile of trajectories including {\em rigorous arrival times} of trajectories at points inducing the degeneracy of dynamical systems, which relates to size of supports for compacton waves, too.
\par
As for rigorous numerics of trajectories of singular ODEs corresponding to compacton traveling waves, Szczelina and Zgliczy\'{n}ski \cite{SZ2013} discuss the treatment in terms of rigorous numerical validations of homoclinic orbits for associated {\em regular} dynamical systems.
Their approach is basically the same as our present one, namely derivation of desingularized system from the original system via the time-scale desingularization as well as the ODE integration with topological validation procedure of connecting orbits such as covering relations and cone conditions.
However, their argument does not refer to the concrete size of supports.
Our present approach tackles validating the explicit range of finite waves or concrete times when the dynamic event occur.
In order to obtain concrete profiles of traveling waves for singular systems, the time (or moving frame) variable desingularization of the form
\begin{equation}
\label{intro-time-desing}
\frac{dt}{d\tau} = T(x(t)) 
\end{equation}
along the solution $x(t)$ is required for desingularizing the system.
Once we obtain the (complete) solution trajectory $\{x(\tau)\}_{\tau \in \mathbb{R}}$ with given initial data, the relationship (\ref{intro-time-desing}) will recover the maximal existence range $(t_{\min}, t_{\max})$ of the solution in the original time scale.
However, the concrete calculation of $t_{\max} = \int_0^\infty T(x(\tau))d\tau$ is not easy in general, which is mainly due to the following reasons:
\begin{itemize}
\item We need compute integrals {\em on infinite domains}.
\item The integral needs information of {\em complete} trajectories.
\end{itemize}
Very recently, the author and his collaborators have tackled this difficulty in validating blow-up solutions \cite{TMSTMO, MT2017}.
They have described infinite-time integration of the form $\int_0^\infty T(u(\tau))d\tau$ by {\em Lyapunov functions} after appropriate desingularization of the form (\ref{intro-time-desing}) to the original system, which enable us to re-parameterize trajectories near stable equilibria by values of Lyapunov functions.
This parameterization realizes computation of rigorous upper and lower bounds of maximal existence times; namely blow-up times.
The validation methodology can be also applied to various systems with the same mechanism (finite-time singularity), including degenerated differential equations in the present considerations. 
On the other hand, only trajectories asymptotic to {\em stable} equilibria are discussed in the preceding works.
As for validations of finite-time extinctions or compacton traveling waves, the following difficulty has to be is also overcome\footnote{
We cannot directly apply the same approach as \cite{TMSTMO, MT2017} to our present problems, because Lyapunov functions of the validation forms can change the sign.
This property prevents us from constructing simple and concrete estimates of $t_{\max}$. 
Details are discussed in Section \ref{section-procedure}.
}, which is typical in the present problems:
\begin{itemize}
\item The integral needs information of complete trajectories asymptotic to {\em saddle-type} invariant sets.
\end{itemize}
We thus develop an alternative tool for validating $t_{\max}$, which is {\em Lyapunov functions on stable manifolds of saddle-type equilibria}.
Standard geometric treatment like {\em cone condition} (e.g., \cite{ZCov}) ensures the graph representation of (un)stable manifolds of saddle-type equilibria. 
Using this information, we can define a quadratic functions depending only on trajectories on stable manifolds of singularities, which enables us to estimate maximal existence times of solutions without any technical obstructions.
The finite-time singularity mechanism we mention here also appears in the other type of phenomenon, one of which is {\em canard solution} in fast-slow systems.
Canard solution is typically discussed in fast-slow systems which form
\begin{equation*}
x' = f(x,y,\epsilon),\quad y' = \epsilon g(x,y,\epsilon)
\end{equation*}
with a sufficiently small parameter $\epsilon > 0$ determining the multi-time scale.
It is induced by a solution trajectory which travels from the stable branch of critical manifolds given by $S_0 = \{f(x,y,0) = 0\}$ to the unstable branch across the fold point called {\em folded singularity}.
This singularity is {\em not} an equilibrium of the reduced vector field $\dot y = g(x,y,0)$ on $S_0$, which implies that the trajectory through the singularity passes the singularity with nonzero speed, but is characterized by an equilibrium for {\em the desingularized vector field} via time-scale transform of the form (\ref{intro-time-desing}).
In other words, singular canards can be regarded as trajectories including singularities with the same mechanism as finite-time extinction and compacton traveling waves.
Throughout discussions in this paper, we will see a common singular mechanism among these phenomena from the viewpoint of numerical validations.
\par
\bigskip
The rest of our paper is organized as follows.
In Section \ref{section-review}, we briefly review backgrounds related to our present concern; {\em finite traveling waves inducing finite-time extinction}, {\em compacton traveling waves} for degenerate parabolic equations and trajectories on critical manifolds through {\em folded singularities} for fast-slow systems.
There we also prove several properties for constructing various type of solutions under assumptions looking quite strong.
Rigorous numerics of trajectories, nevertheless, lets such assumptions meaningful in practical applications.
In Section \ref{section-procedure}, we provide a rigorous computation methodology of trajectories involving finite-time singularities.
The key theorem, validation of (locally defined) Lyapunov functions on stable manifolds of equilibria, is shown there.
Once we obtain a suitable treatment of finite-time singularities, numerical validations are realized by standard methodologies discussed in many preceding works (e.g., \cite{MT2017, SZ2013, TMSTMO, W}).
We also mention a local {\em orbital equivalence} between the original and desingularized systems there.
In particular, we discuss if validated trajectories have the same information between the original and desingularized systems, which can be discussed in terms of orbital equivalence in dynamical systems.
The key point is the property that $T(x)$ in (\ref{intro-time-desing}) has an identical sign along trajectories, which is non-trivial in general\footnote{
Fortunately, this question always has an positive answer in preceding studies \cite{Mat, MT2017, TMSTMO} involving blow-up solutions.
}.
We prove that, in a simple case, the orbital equivalence can be reduced to conditions involving eigenvectors and cones around equilibria which play key roles in describing asymptotic behavior.
In Section \ref{section-example}, validation examples of finite-time singularities are shown.
The first example is concerned with traveling waves for a degenerate parabolic equation.
We validate finite traveling wave inducing finite time extinction, compacton traveling waves and composite waves superposing finite traveling waves describing expansion of dead cores.
Furthermore, we show that the existence of different compacton traveling waves with an identical speed induce a {\em chaotic} structure in a family of weak solutions of degenerate equations.
The second example is concerned with folded singularities in fast-slow systems.
Trajectories for reduced systems called {\em singular canards} pass singularities, which are equilibria {\em only for desingularized systems}, with nonzero speed.
Since there are well-known and fundamental studies that such singular canards induce very complex and rich solution structures like {\em maximal canards, canard explosion and mixed-mode oscillations} in fast-slow systems (e.g., \cite{DGKKOW2012, KS2001, KW2010, SW2001}), our present methodology will also open the door for comprehensive studies of such complex and rich solution structures in fast-slow systems from both mathematical and numerical viewpoints.

\section{Our present focus : brief review of backgrounds}
\label{section-review}
We list several typical finite-time singularities treated in this study.
Blow-up phenomena of differential equations are also included in our focus, but we do not discuss details of blow-ups because they are already discussed in previous works \cite{Mat, MT2017, TMSTMO}.

\subsection{Finite traveling waves in degenerate diffusion equation}
One of our concerns here is the following quasilinear parabolic equation (e.g., \cite{LPB2004}):
\begin{align}
\label{deg-para}
u_t &= d(u^{m+1})_{xx} + f(u),\quad (t,x)\in Q := (0,\infty)\times \mathbb{R},\\
\label{deg-para-bc}
u(0,x) &= u_0(x),\quad x\in \mathbb{R},\ 0\leq u_0 \leq 1,
\end{align}
where $m > 0$, $d = (m+1)^{-1}, u_0\in C^0(\mathbb{R})$.
Here assume that a nonlinearity $f$ satisfies
\begin{description}
\item[(Scalar1)] $f\in C^0[0,1]\cap C^1(0,1]$, $f(0)=f(1)=0$, and $f'(1) < 0$.
Moreover, there exists $\alpha \in (0,1)$ such that $f<0$ on $(0,\alpha)$ and $f>0$ on $(\alpha, 1)$.
\item[(Scalar2)] There exist $p\in (0,1)$ and $\beta > 0$ such that $N:=m+p\geq 1$ and $u^{-p}f(u)\to -\beta$ as $u\to +0$.
\item[(Scalar3)] $\int_0^1 v^m f(v)dv < 0$. 
\end{description}

Such kind of diffusion equations arises in many applications, including population genetics, signal propagation in nerve axons, and combustion theory (\cite{LPB2004}).
One of our main targets in this paper is a {\em finite traveling wave}.
This wave is not always smooth, which is due to degeneracy of the system.
We thus treat solutions of the system in a weak sense as follows.
\begin{dfn}[e.g., \cite{LPB2004}]\rm
\label{dfn-weak}
A nonnegative function $u$ is said to be a {\em weak solution of (\ref{deg-para})-(\ref{deg-para-bc})} provided for each $r,T>0$, letting $Q_T = (0,T)\times \mathbb{R}$, $u\in C(Q_T)\cap L^\infty(Q_T)$ holds and
\begin{align}
\notag
\int_{Q_T} \left\{ u\phi_t + du^{m+1}\phi_{xx} + f(u)\phi \right\} dxdt &= \int_{-r}^r \left\{ u(T,x) \phi(T,x) - u_0(x) \phi(0,x) \right\} dx \\
\label{weak-form-diff}
	&+ d\int_0^T \left\{ u^{m+1}(t,r) \phi_x(t,r) - u^{m+1}(t,-r) \phi_x(t,-r) \right\}dt
\end{align}
also holds for all $\phi\in C^{2,1}(\overline{Q_T})$\footnote{
The function space $C^{2,1}(\overline{Q_T})$ consists of functions $u$ such that $u$ is twice continuously differentiable in $x$ and that $u$ is continuously differentiable in $t$ at any points on $\overline{Q_T}$.
}
such that $\phi \geq 0$ and that $\phi(t, \pm r) = 0$ for all $t\in [0,T]$.
\end{dfn}

The following definition of traveling waves makes sense even if waves lose regularity at a point $(t,x)\in Q$. 
\begin{dfn}\rm
A {\em finite traveling wave solution} of (\ref{deg-para})-(\ref{deg-para-bc}) is a weak solution of the form $u(t,x) = \varphi(x-ct)$ with a velocity $c\in \mathbb{R}$ satisfying $\varphi(z) \equiv 0$ for $z\geq w$ (or $z\leq w$) for some $w\in \mathbb{R}$. 
\end{dfn}
Preceding works for degenerate diffusion equation (\ref{deg-para})-(\ref{deg-para-bc}) refer to such finite traveling waves and their reflections as upper solutions.
As an application, the finite time extinction of compactly supported solution of (\ref{deg-para})-(\ref{deg-para-bc}) is considered (e.g., \cite{LPB2004}).

\par
\bigskip
Now move to the traveling wave problem associated with (\ref{deg-para})-(\ref{deg-para-bc}).
Let us write (\ref{deg-para})-(\ref{deg-para-bc}) in the divergence form
\begin{equation*}
u_t = (u^m u_x)_x + f(u),\quad (t,x)\in Q.
\end{equation*}
Setting $z = x-ct$, $u$ is formally a traveling wave solution of (\ref{deg-para}) with velocity $c$ if and only if $u$ satisfies the quasilinear ODE
\begin{equation*}
-cu' = (u^m u')' + f(u),\quad z\in \mathbb{R}
\end{equation*}
with stationary boundary conditions $\lim_{z\to \pm \infty}u(z) = u_\pm\in \mathbb{R}$, respectively, where $' = \frac{d}{dz}$.
Now we {\em desingularize} the above ODE by introducing the rescaling of the \lq time' variable $z$ according to
\begin{equation*}
\frac{ds}{dz} = \frac{1}{u^m(z)}.
\end{equation*}
Setting $U(s) = u(z)$, then $u(z)$ is a (weak) traveling wave solution of (\ref{deg-para}) if and only if $U(s)$ is a ({\em classical}) traveling wave solution of 
\begin{equation}
\label{tw-desing-scalar-1}
U_s = U_{xx} + U^m f(U),\quad (t,x)\in Q,
\end{equation}
or equivalently
\begin{equation}
\label{tw-desing-scalar-2}
-c\dot U = \ddot U + U^m f(U),\quad s\in \mathbb{R}
\end{equation}
with appropriate boundary conditions, where $\dot{} = \frac{d}{ds}$.
The ODE (\ref{tw-desing-scalar-2}) is called the {\em desingularized ODE}. 
Note that, by (Scalar1) and (Scalar2), we have $U^mf(U)\in C^1[0,1]$.

\par
\bigskip

Traveling wave solutions of (\ref{tw-desing-scalar-1}) are referred to as {\em connecting orbits} of (\ref{tw-desing-scalar-2}) connecting equilibria, which reduces the problem to dynamical systems generated by ordinary differential equations with appropriate boundary conditions.
Rewrite (\ref{tw-desing-scalar-2}) as the first-order system
\begin{equation}
\label{tw-scalar-1st}
\begin{cases}
\dot U = V, &\\
\dot V = -cV - U^mf(U). 
\end{cases}
\end{equation}
The system (\ref{tw-scalar-1st}) possesses equilibria at $(0,0)$, $(1,0)$ and possibly at more points.
Connecting orbits for (\ref{tw-scalar-1st}) can be validated by standard methodology based on rigorous numerics, which is stated in Section \ref{section-procedure}.
Our validation opens the door for constructing very rich structures of solution sets.
For example, we easily obtain the superposition of solutions with special properties with computer assistance, which is stated below.

\begin{prop}
\label{prop-superpose}
Consider (\ref{deg-para})-(\ref{deg-para-bc}) with polynomial nonlinearity $f$ such that $f(0) = 0$.
Let $u=u(t,x)$ and $v=v(t,x)$ be weak solutions of (\ref{deg-para})-(\ref{deg-para-bc}) with the initial data $u_0$ and $v_0$, respectively.
Assume that, for any $T>0$, ${\rm supp}u\cap {\rm supp}v = \emptyset$ in $Q_T$.
Then $u+v$ is also a weak solution of (\ref{deg-para})-(\ref{deg-para-bc}) with the initial data $u_0 + v_0$.
\end{prop}

\begin{proof}
First note that the value of the superposition $(u+v)(t,x)$ is defined by the sum $u(t,x) + v(t,x)$, by the definition of ring structure on $C(Q_T)$.
In particular, we have $\int_{-r}^r (u+v)(T,x)\phi(T,x)dx = \int_{-r}^r u(T,x)\phi(T,x)dx + \int_{-r}^r v(T,x)\phi(T,x)dx$, and so on.
\par
Next consider the integral $\int_{Q_T} u^m \phi dxdt$.
If ${\rm supp}u\cap {\rm supp}v = \emptyset$ in $Q_T$, then $(uv)(t,x)\equiv u(t,x)v(t,x) = 0$ holds for any $(t,x)\in Q_T$, since $(t,x)$ is away from at least one support of solutions $u,v$.
We thus obtain the following property:
\begin{equation*}
\int_{Q_T} (u+v)^m \phi dxdt = \int_{Q_T} u^m \phi dxdt + \int_{Q_T} v^m \phi dxdt.
\end{equation*}
Similarly we have, for any $x\in \mathbb{R}$, 
$\int_0^T (u+v)^{m+1}(t,x) \phi_x(t,x)dt = \int_0^T \{ (u^{m+1}(t,x)+v^{m+1}(t,x)\} \phi_x(t,x)dt$. 
With the assumption about $f$, we know $\int_{Q_T} f(u+v) \phi dxdt = \int_{Q_T} \{f(u)+f(v)\} \phi dxdt$ for any $\phi$ under consideration.
Finally we have
\begin{align*}
\notag
\int_{Q_T} &\left[ (u+v)\phi_t + d(u+v)^{m+1}\phi_{xx} + f(u+v)\phi \right] dxdt \\
	& = \int_{Q_T} \left[ (u+v)\phi_t + d(u^{m+1}+v^{m+1})\phi_{xx} + \{f(u)+f(v)\}\phi \right] dxdt \\
	&= \int_{-r}^r \left\{ u(T,x) \phi(T,x) - u_0(x) \phi(0,x) \right\} dx + d\int_0^T \left\{ u^{m+1}(t,r) \phi_x(t,r) - u^{m+1}(t,-r) \phi_x(t,-r) \right\}dt\\
	&+ \int_{-r}^r \left\{ v(T,x) \phi(T,x) - v_0(x) \phi(0,x) \right\} dx + d\int_0^T \left\{ v^{m+1}(t,r) \phi_x(t,r) - v^{m+1}(t,-r) \phi_x(t,-r) \right\}dt\\
	&= \int_{-r}^r \left\{ (u(T,x)+v(T,x)) \phi(T,x) - (u_0(x)+v_0(x)) \phi(0,x) \right\} dx\\
	&\quad + d\int_0^T \left\{ (u^{m+1}(t,r)+v^{m+1}(t,r)) \phi_x(t,r) - (u^{m+1}(t,-r)+v^{m+1}(t,-r)) \phi_x(t,-r) \right\}dt\\
	&= \int_{-r}^r \left\{ (u+v)(T,x) \phi(T,x) - (u_0+v_0)(x) \phi(0,x) \right\} dx\\
	&\quad + d\int_0^T \left\{ (u+v)^{m+1}(t,r) \phi_x(t,r) - (u+v)^{m+1}(t,-r) \phi_x(t,-r) \right\}dt,
\end{align*}
which completes the proof.
\end{proof}

The most important point of this proposition is that solutions with disjoint supports can be superposed in the sense of weak solutions. 
Ordinary traveling wave solutions in regular systems can never realize such properties.
{\em Finite} traveling waves, on the other hand, easily attain this property, which indicates that finite traveling waves induce very rich solution structures.

\subsection{Trajectories passing folded singularities in fast-slow systems}
\label{section-folded}
The second topic in the present paper involves {\em folded singularities}, which brings about singular Hopf bifurcation or canard bifurcation in suitable situations (e.g. \cite{DGKKOW2012, KW2010}).
Consider the following fast-slow system:
\begin{equation*}
\begin{cases}
dw / d\tilde t = f(w,y,\epsilon), & \\
dy / d\tilde t = \epsilon g(w,y, \epsilon), & 
\end{cases}
\end{equation*}
equivalently
\begin{equation}
\label{fast-slow-param}
\begin{cases}
\epsilon w' = f(w,y,\epsilon), & \\
y' = g(w,y, \epsilon), & 
\end{cases}\quad t = \epsilon \tilde t,
\end{equation}
where $w\in \mathbb{R}$ and $y\in \mathbb{R}^l$, $l\geq 2$.
For simplicity, we assume that $l=2$ and we recast (\ref{fast-slow-param}) as
\begin{equation}
\label{fast-slow-FSN}
\begin{cases}
\epsilon w' = f(w,y,z,\epsilon), & \\
y' = g_1(w,y,z,\epsilon), &\\
z' = g_2(w,y,z,\epsilon), & 
\end{cases}
\end{equation}
where $y\in \mathbb{R}$ and $z\in \mathbb{R}$.
The variable $w\in \mathbb{R}$ can be generalized to the case $\mathbb{R}^n$ with $n> 1$, but just one component in fast variable is the essential to the present phenomenon, and hence we only discuss the case $n=1$.
General situation is discussed in \cite{KW2010}.

\begin{rem}\rm
The assumption $l\geq 2$ is crucial when we consider canard-type bifurcations.
In the case $l=1$, typically known canard phenomenon (for planar systems) is {\em canard explosion}, which induces the sudden bifurcation between small limit cycles and relaxation oscillations in an {\em exponentially small} parameter range.
However, as soon as there is more than one slow variables, canard-type bifurcation can exist for $O(1)$ ranges of a parameter.
See e.g., \cite{KS2001_RO, KW2010} for details.
\end{rem}

Assume that the critical manifold $S = \{f=0\}$ of (\ref{fast-slow-FSN}) has an attracting sheet $S^a$ and a repelling sheet $S^r$\footnote{
The {\em attracting} sheet $S^a \subset S$ is the subset of points on $S$ where the Jacobian matrix $f_w(w,y,z,0)$ is negative, whereas the {\em repelling} sheet $S^r \subset S$ is the subset of points on $S$ where $f_w(w,y,z,0)$ is positive.
} that meet a fold curve $F$, which is generically a one-parameter family of saddle-node bifurcation points; namely the set of $p_\ast \in F$ being generic in the following sense: 
\begin{equation}
\label{FS-generic}
f(p_\ast, 0) = 0,\quad \frac{\partial f}{\partial w}(p_\ast,0) = 0,\quad \frac{\partial^2 f}{\partial w^2}(p_\ast,0) \not = 0,\quad D_{(y,z)}f(p_\ast,0)\text{ has full rank one}.
\end{equation}
Note that the slow flow (\lq\lq reduced" flow on $S$) is not defined on the fold curve before desingularization.
At most fold points, trajectories approach or depart from both $S^a$ and $S^r$.
Generically there may be isolated points called {\em folded singularities} defined below, where the trajectories of the slow flow switch from incoming to outgoing.
Folded singularities are characterized as equilibria of the desingularized slow flow.
\par
In what follows we consider the \lq\lq desingularized" flow at $\epsilon =0$, which makes sense in the whole $S$ including $F$.
To this end, differentiate $f(w,y,z,0) = 0$ on $S$ with respect to $t$:
\begin{equation*}
f_w w' + f_y y' + f_z z' = 0,
\end{equation*}
which is equivalent to
\begin{equation*}
f_w w' = -(f_y g_1 + f_z g_2).
\end{equation*}
Via the time-scale transformation $\tau = -f_w(w,y,z) t$, we obtain the desingularized flow
\begin{equation}
\label{FSN-reduced}
\begin{cases}
w' = f_y g_1 + f_z g_2, & \\
y' = - f_w g_1(w,y,z,0), &\\
z' = - f_w g_2(w,y,z,0), & 
\end{cases}
\end{equation}
which makes sense including $F$, but only on $S$. 

\begin{dfn}\rm
A fold point $p_\ast \in F$ is a {\em folded singularity} if it is an equilibrium of (\ref{FSN-reduced}); namely the solution of 
\begin{equation*}
g_1(p_\ast, 0)\frac{\partial f}{\partial y}(p_\ast, 0) + g_2(p_\ast, 0)\frac{\partial f}{\partial z}(p_\ast, 0)= 0.
\end{equation*}
\end{dfn}

In the assumption ${\rm rank}D_{(y,z)}f(p_\ast,0)=1$, under arrangements of coordinates, we may assume that $D_{(y,z)}f(p_\ast,0)=(1,0)$, in which case, the critical manifold $S$ can be represented by $(x,y) = h(w,z)= (h^x(w,z), h^y(w,z))$ for some smooth function $h$ in a small neighborhood of $p_\ast$.
Therefore, (\ref{FSN-reduced}) on $S$ is further reduced to the vector field with respect to $(w,z)$.
Direct calculations yield
\begin{align*}
y' &= h^y_w w' + h^y_z z' = h^y_w w' + h^y_z g_2(w,h(w,z),z,0) = g_1(w,h(w,z),z,0),\\
z' &= g_2(w,h(w,z),z,0).
\end{align*}
Thus we have
\begin{align*}
h^y_w w' &= g_1(w,h(w,z),z,0) - h^y_z g_2(w,h(w,z),z,0),\\
z' &= g_2(w,h(w,z),z,0).
\end{align*}
Using the time-rescaling $d\tau /dt = h^y_w$, we finally obtain
\begin{equation}
\label{reduced-final}
\begin{cases}
\dot w = g_1(w,h(w,z),z,0) - h^y_z g_2(w,h(w,z),z,0), & \\
\dot z = h^y_w g_2(w,h(w,z),z,0), &
\end{cases}\quad \dot {} = \frac{d}{d\tau}.
\end{equation}
Let $J(w,z)$ be the Jacobian matrix of the right-hand side of (\ref{reduced-final}) at $(w,z)$.

\begin{dfn}\rm
Let $\sigma_1, \sigma_2$ be eigenvalues of $J(w_\ast, z_\ast)$, where $(w_\ast, h(w_\ast, z_\ast), z_\ast) = p_\ast$ is a folded singularity.
Then $p_\ast$ is called
\begin{equation*}
\begin{cases}
\text{\em folded saddle} & \text{ if }\sigma_1 \sigma_2 < 0,\quad \sigma_i \in \mathbb{R},\\
\text{\em folded node} & \text{ if }\sigma_1 \sigma_2 > 0,\quad \sigma_i \in \mathbb{R},\\
\text{\em folded focus} & \text{ if }\sigma_1 \sigma_2 > 0,\quad {\rm Im}(\sigma_i)\not = 0.
\end{cases}
\end{equation*}
\end{dfn}

\begin{rem}
The folded singularity $p_\ast$ is an equilibrium of (\ref{FSN-reduced}), while it is {\em not always} an equilibrium of the original system (\ref{fast-slow-FSN}) with $\epsilon = 0$.
Moreover, the function $h^y_w$ may change the sign, which means that the actual flow on the reduced system evolves in the opposite direction across the zero of $h^y_w$.
In particular, the trajectory on the stable manifold of $p_{\ast}$ for (\ref{FSN-reduced}) will arrive  the folded singularity $p_\ast$ is typically a zero of $h^y_w$.
\end{rem}

In many cases, folded saddles and nodes are focused in bifurcation theory. 
These singularities allow the reduced flow (\ref{reduced-final}) to cross from $S^a$ to $S^r$.
Such trajectories are called {\em singular canards}.
More precisely, the following concepts are generally given and are well studied. 

\begin{dfn}[Singular canards, e.g., \cite{MS2001, KW2010}]\rm 
Solutions of the reduced problem, namely (\ref{fast-slow-FSN}) with $\epsilon = 0$, passing through a folded singularity from an attracting critical manifold to a repelling critical manifold are called {\em singular canards}.
Solutions of the reduced problem passing through a folded singularity from a repelling critical manifold to an attracting critical manifold are called {\em singular faux canards}.
\end{dfn}

Persistence of singular (especially, {\em not faux}) canards under small perturbation $\epsilon \ll 1$ in (\ref{fast-slow-FSN}) is of great interest since it gives rise to complex dynamics (e.g., \cite{DGKKOW2012, MS2001}).
On the other hand, a folded focus is of little interest, since there are no (maximal) canards generated by a folded focus.

\begin{figure}[htbp]
  \begin{center}
\begin{minipage}{0.48\hsize}
\centering
    \includegraphics[width=5cm]{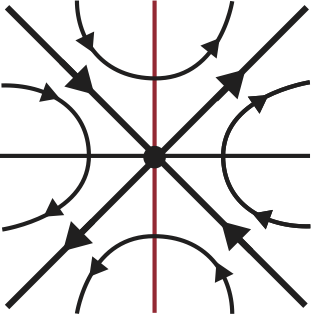}
(a)
\end{minipage}
\begin{minipage}{0.48\hsize}
\centering
\includegraphics[width=5cm]{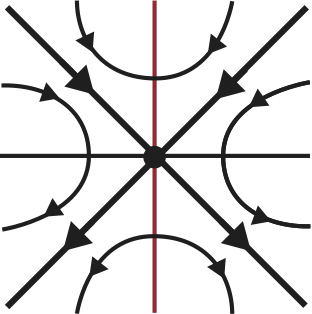}
(b)
\end{minipage}
    \caption{Folded saddle and singular (faux) canards (e.g., \cite{SW2001})}
    \label{fig_FS}
\end{center}
Desingularized and the original reduced flows on the critical manifold $S_0$ are drawn.
The red line represents a (generic) folded line $F$.
Sign of $T(x)$ determining (\ref{time-desing}) below is assumed to be identical in each component of $(U\setminus F)\cap S_0$, where $U$ is a neighborhood of a folded singularity (black point) in $S_0$.
(a) : desingularized flow around a folded saddle. 
The folded saddle is an equilibrium of the desingularized flow.
(b) : the corresponding {\em original} vector field. 
The center point is {\em not} an equilibrium.  
Typically $T(x)$ is negative in a component of $(U\setminus F)\cap S_0$ (the right-hand side in the present figure) where the flow direction is reversed. 
The bold lines correspond to special solutions on $S_0$ containing $F$, either of which is a singular canard and the other is a singular faux canard depending on the stability of branches of $S_0$ in the full singular limit system (\ref{fast-slow-param})$_{\epsilon=0}$.
\end{figure}

\section{Rigorous numerics of finite-time singularities}
\label{section-procedure}
Consider the vector field
\begin{equation}
\label{ODE}
x' \equiv \frac{dx}{dt} = f(x, \mu),\quad f:\mathbb{R}^n\times \mathbb{R}^k\to \mathbb{R}^n,
\end{equation}
where $\mu$ is a parameter.
Here we do {\em not} assume the continuity and smoothness of $f$ with respect to $x$ provided no special mentions are present.
In this section, we firstly review topological tools for validating global trajectories in regular dynamical systems. 
Secondly we provide a rigorous numerical computation procedure of maximal existence times of trajectories in dynamical systems with degeneracy, which are often referred to as {\em extinction time}, {\em arrival time} or {\em passing time}.

\subsection{Isolating blocks}
\label{section-block-brief}
Firstly we review topological tools for {\em regular} dynamical systems.
The first one is an {\em isolating block}, which validates invariant sets with rigorous information of vector fields on the boundary.
Here we assume that $f$ is smooth.
\begin{dfn}[Isolating block]\rm
\label{dfn-isolation}
Let $N\subset \mathbb{R}^D$ be a compact set. 
We say $N$ an {\em isolating neighborhood} if $\Inv(N)\subset \Int (N)$ holds, where 
\begin{equation*}
\Inv(N):= \{x\in N \mid \varphi(\mathbb{R},x)\subset N\}
\end{equation*}
for a flow $\varphi: \mathbb{R}\times \mathbb{R}^D\to \mathbb{R}^D$ on $\mathbb{R}^D$.
Next let $x\in \partial N$. We say $x$ an {\em exit} ({\em resp. entrance}) point of $N$, if for every solution $\sigma:[-\delta_1,\delta_2]\to \mathbb{R}^N$ through $x= \sigma(0)$, with $\delta_1\geq 0$ and $\delta_2 > 0$ there are $0\leq \epsilon_1 \leq \delta_1$ and $0 < \epsilon_2 \leq \delta_2$ such that for $0 < t \leq \epsilon_2$,
\begin{equation*}
\sigma(t)\not \in N\ (\text{resp. } \sigma(t)\in \Int(N)),
\end{equation*}
and for $-\epsilon_1 \leq t < 0$,
\begin{equation*}
\sigma(t)\not \in \partial N\ (\text{resp. } \sigma(t)\not \in N)
\end{equation*}
hold. $N^{\exit}$ (resp. $N^{\ent}$) denote the set of all exit (resp. entrance) points of the closed set $N$. We call $N^{\exit}$ and $N^{\ent}$ {\em the exit} and {\em the entrance} of $N$, respectively.
Finally $N$ is called {\em an isolating block} if $\partial N = N^{\exit}\cup N^{\ent}$ holds and $N^{\exit}$ is closed in $\partial N$.
\end{dfn}
There are several algorithms for constructing isolating blocks.
Here we apply the procedure provided first by Zgliczy\'{n}ski and Mischaikow \cite{ZM2001, Mat2} to our problems, which effectively constructs blocks including equilibria.
In \cite{Mat2}, construction of isolating blocks based on predictor-corrector method with respect to parameters for sample (numerical) equilibria is also shown, which is briefly collected in Appendix \ref{section-block}.
Note that these blocks can be directly applied to the successive validations related to (un)stable manifolds, which are discussed in Section \ref{section-stable-manifold}.

\subsection{Covering relations and connecting orbits}

Next we move to the review of covering relations.
We assume that $f$ is smooth as in the previous subsection.
A central target is an {\em $h$-set} defined as follows.
\begin{dfn}[$h$-set, cf. \cite{ZCov, ZG}]\rm
\label{dfn-hset}
An {\em $h$-set} consists of the following set, integers and a map:
\begin{itemize}
\item A compact subset $N\subset \mathbb{R}^D$.
\item Nonnegative integers $u(N)$ and $s(N)$ such that $u(N) + s(N) = n$ with $n\leq D$.
\item A homeomorphism $c_N:\mathbb{R}^n\to \mathbb{R}^{u(N)}\times \mathbb{R}^{s(N)}$ satisfying
\begin{equation*}
c_N(N) = \overline{{\bf B}_{u(N)}}\times \overline{{\bf B}_{s(N)}},
\end{equation*}
where $\overline{{\bf B}_d} = \overline{{\bf B}_d(0,1)}$ denotes the $d$-dimensional open unit ball centered at the origin.
\end{itemize}
We shall say the coordinate given by the image of $c_N$ {\em the diagonal coordinate}.
Finally define the {\em dimension} of an $h$-set $N$ by $\dim N:= n$.
\end{dfn}
We shall write an $h$-set $(N,u(N),s(N),c_N)$ simply by $N$ if no confusion arises. Let
\begin{align*}
N_c &:=  \overline{{\bf B}_{u(N)}} \times  \overline{{\bf B}_{s(N)}},\quad
N_c^- := \partial \overline{{\bf B}_{u(N)}} \times  \overline{{\bf B}_{s(N)}},\quad
N_c^+ := \overline{{\bf B}_{u(N)}} \times  \partial \overline{{\bf B}_{s(N)}},\\
N^- &:= c_N^{-1}(N_c^-), \quad N^+:= c_N^{-1}(N_c^+)
\end{align*}

The following notion describes the topological transversality between two $h$-sets relative to continuous maps.

\begin{dfn}[Covering relations, cf. \cite{ZCov, ZG}]\rm
\label{dfn-covrel}
Let $N, M\subset \mathbb{R}^m$ be $h$-sets with $u(N)+s(N), u(M)+s(M)\leq m$ and $u(N)=u(M) = u$. $f: N \to \mathbb{R}^{\dim M}$ denotes a continuous mapping and $f_c:= c_M\circ f\circ c_N^{-1}: N_c\to \mathbb{R}^{u} \times \mathbb{R}^{s(M)}$. We say $N$ {\em $f$-covers} $M$ ($N\overset{f}{\Longrightarrow}M$) if the following statements hold: 
\begin{enumerate}
\item There exists a continuous homotopy $h:[0,1]\times N_c\to \mathbb{R}^{u}\times \mathbb{R}^{s(M)}$ satisfying
\begin{align*}
&h_0 = f_c,\quad 
h([0,1],N_c^-)\cap M_c = \emptyset,\quad
h([0,1],N_c)\cap M_c^+ = \emptyset,
\end{align*}
where $h_\lambda = h(\lambda, \cdot)$ ($\lambda \in [0,1]$).
\item There exists a continuous mapping $A:\mathbb{R}^{u}\to \mathbb{R}^{u}$ such that
\begin{equation}
\label{cov-degree}
\begin{cases}
h_1(p,q) = (A(p),0), &\\
A(\partial {\bf B}_u (0,1)) \subset \mathbb{R}^u \setminus \overline{{\bf B}_u}(0,1), & \\
\deg(A, \overline{{\bf B}_u}, 0)\not = 0
\end{cases}
\end{equation}
holds for $p\in \overline{{\bf B}_u}(0,1), q\in \overline{{\bf B}_s}(0,1)$.
\end{enumerate}
\end{dfn}

\begin{rem}\rm
In definition of covering relation between $N$ and $M$, the disagreement of $\dim N$ and $\dim M$ is not essential. On the contrary, the equality $u(N) = u(M) = u$ is essential because the mapping degree of $u$-dimensional mapping $A$ should be derived.
\end{rem}



A fundamental result in the theory of covering relations is the following proposition.

\begin{prop}[Theorem 4 in \cite{ZG}]
\label{ZG-periodic}
Let $N_i, i=0,1,\cdots, k$ be $h$-sets such that $u(N_i)= u$ for $i=0,1,\cdots, k$ and let $f_i: N_i \to \mathbb{R}^{\dim (N_{i+1})}$, $i=0,1,\cdots, k-1$, be continuous. 
Assume that, for all $i=0,1,\cdots, k-1$, the covering relation $N_i\overset{f_i}{\Longrightarrow}N_{i+1}$ holds.
Then there is a point $p\in \Int N_0$ such that
\begin{equation*}
f_i \circ f_{i-1}\circ \cdots f_0(p)\in \Int N_i\quad \text{ for all }i=0,\cdots, k-1.
\end{equation*}
\end{prop}

The concept of {\em horizontal and vertical disks} are useful to describe asymptotic trajectories in terms of covering relations.

\begin{dfn}[Horizontal and vertical disk, e.g. \cite{W, ZCov}]\rm
\label{dfn-disks}
Let $N$ be an $h$-set. 
Let $b_s: \overline{{\bf B}_{s(N)}}\to N$ be continuous and let $(b_s)_c = c_N \circ b_s$. We say that $b_s$ is a {\em vertical disk in $N$} if there exists a homotopy $h: [0,1]\times \overline{{\bf B}_{s(N)}}\to N_c$ such that
\begin{align*}
h_0 &= (b_s)_c,\quad h_1(x) = (0,x),\quad \text{ for all }x\in \overline{{\bf B}_{s(N)}},\\
h(t,x) &\in N_c^+,\quad \text{ for all }t\in [0,1]\text{ and }x\in \partial {\bf B}_{s(N)}.
\end{align*}
Let $b_u: \overline{{\bf B}_{u(N)}}\to N$ be continuous and let $(b_u)_c = c_N \circ b_u$. We say that $b_u$ is a {\em horizontal disk in $N$} if there exists a homotopy $h: [0,1]\times \overline{{\bf B}_{u(N)}}\to N_c$ such that
\begin{align*}
h_0 &= (b_u)_c,\quad h_1(x) = (x,0),\quad \text{ for all }x\in \overline{{\bf B}_{u(N)}},\\
h(t,x) &\in N_c^-,\quad \text{ for all }t\in [0,1]\text{ and }x\in \partial {\bf B}_{u(N)}.
\end{align*}
\end{dfn}

Combining these concepts with covering relations, we obtain the following result, 
which is often applied to the existence of homoclinic and heteroclinic orbits.
\begin{prop}[e.g., Theorem 3.9 in \cite{W}]
\label{WZ-heteroclinic}
Let $N_i, i=0,1,\cdots, k$ be $h$-sets such that $u(N_i)= u$ for $i=0,1,\cdots, k$ and let $f_i: N_i \to \mathbb{R}^{\dim (N_{i+1})}$, $i=0,1,\cdots, k-1$, be continuous. Let $v: \overline{{\bf B}_{s(N_k)}}\to N_k$ be a vertical disk in $N_k$. If $N_i\overset{f_i}{\Longrightarrow}N_{i+1}$ holds for $i=0,1,\cdots, k-1$, then there exists $x\in N_0$ such that
\begin{align*}
&(f_i\circ f_{i-1}\circ \cdots \circ f_0)(x) \in N_{i+1},\quad \text{ for }i=0,1,\cdots, k-2,\\
&(f_{k-1}\circ f_{k-2}\circ \cdots \circ f_0)(x)\in v(\overline{{\bf B}_{s(N_k)}}).
\end{align*}
Moreover, if $b: \overline{{\bf B}_u}\to N_0$ is a horizontal disk in $N_0$, then the is a point $\tau \in \overline{{\bf B}_u}$ such that
\begin{align*}
&(f_i\circ f_{i-1}\circ \cdots \circ f_0)(b(\tau)) \in N_{i+1},\quad \text{ for }i=0,1,\cdots, k-2,\\
&(f_{k-1}\circ f_{k-2}\circ \cdots \circ f_0)(b(\tau))\in v(\overline{{\bf B}_{s(N_k)}}).
\end{align*}
\end{prop}

Integration of (regular) differential equations is realized by, say Lohner's method (e.g., \cite{ZLoh}) with computer assistance.

\subsection{Explicit estimates of arrival time with computer assistance}
Here we provide an explicit estimate methodology of arrival times.
Assume that the vector field (\ref{ODE}) has {\em a degeneracy} at a zero of $T(x) = 0$ for some function $T$, and that admit the following time-scale desingularization
\begin{equation}
\label{time-desing}
\frac{d\tau}{dt} = T(x(t))
\end{equation}
to obtain the {\em desingularized} vector field
\begin{equation}
\label{ODE-desing}
\dot x \equiv \frac{dx}{d\tau} = \frac{dx}{dt}\frac{dt}{d\tau} = T(x)^{-1}f(x, \mu).
\end{equation}
The desingularized vector field is supposed to possess the same dynamics structure as (\ref{ODE}) in the sense of {\em orbital equivalence}; namely $T(x) > 0$ at least off the degeneracy.
Let $\varphi^{des}_{\mu}$ be the flow generated by (\ref{ODE-desing}) with fixed $\mu\in \mathbb{R}^k$.
The main assumption in this subsection is summarized as follows.
\begin{ass}
There is a nonnegative smooth function $T:\mathbb{R}^n \to \mathbb{R}$ such that
the vector field (\ref{ODE}) has {\em a degeneracy} at a zero of $T(x) = 0$, and that the vector-valued function $T(x)^{-1}f(x, \mu)$ is smooth everywhere.
We shall call the zero of $T(x)=0$ where the original vector field (\ref{ODE}) degenerates {\em a degeneracy-inducing point}.
The point $x_\ast$ denote the degeneracy-inducing point.
Moreover, the function $T(x)$ is positive in an open subset $U \subset \mathbb{R}^n\setminus \{x_\ast\}$.
\end{ass}

If one validates a time-global orbit for (\ref{ODE-desing}) {\em in $U$}, then we calculate the maximal existence time
\begin{equation*}
t_{\max} = \int_0^\infty \frac{d\tau}{T(x(\tau))}
\end{equation*}
in the original $t$-time scale, which gives information of trajectories in the original problem (\ref{ODE}).
\par
The next problem we have to consider is how we calculate $t_{\max}$, which looks very non-trivial since the integral has to be considered on {\em infinite} set and the {\em whole} trajectory $\{x(\tau)\}_{\tau \geq 0}$ is required. 
The basic idea we apply here is {\em Lyapunov tracing} discussed in \cite{MHY2016, TMSTMO}, namely, computation of $t_{\max}$ of trajectory $\{x(\tau)\}$ in terms of Lyapunov functions around degeneracy-inducing point $x_\ast$\footnote{
In the case of blow-up solutions, the degeneracy-inducing point corresponds to equilibria \lq\lq at infinity".
}.
Solutions with finite-time singularity correspond to trajectories on {\em stable manifolds of hyperbolic equilibria $x_\ast$ for (\ref{ODE-desing})}\footnote
{
In general, such equilibria for (\ref{ODE-desing}) are {\em not} those for (\ref{ODE}).
}.
According to this fact and preceding methodology in \cite{TMSTMO, MT2017}, we validate asymptotic behavior of finite-time singularities by the following steps.

\begin{alg}[Asymptotic behavior around degeneracy-inducing points]
\label{alg-asym}
\begin{enumerate}
\item Validate a degeneracy-inducing point $x_\ast$ as well as an isolating block $N$ for degenerate vector field (\ref{ODE-desing}) such that $x_\ast \in {\rm int}N$. 
\item Validate a trajectory on the stable manifold $W^s(x_\ast; \varphi^{des}_\mu)$ of $x_\ast\in N$.
\item Construct a Lyapunov function on $W^s(x_\ast; \varphi^{des}_\mu)\cap N$.
\item Verify whether $T(x)$ has an identical signature on each component of $(W^s(x_\ast; \varphi^{des}_\mu)\setminus \{x_\ast\})\cap N$.
\item Calculate an enclosure of $t_{\max}$.
\end{enumerate}
\end{alg}
Whole steps for validating asymptotic behavior in Algorithm \ref{alg-asym} is the same as those for validating blow-up solutions \cite{TMSTMO, MT2017} except Step 4.
Step 4 is the additional part for guaranteeing the orbital equivalence of trajectories between for (\ref{ODE}) and for (\ref{ODE-desing})\footnote{
In the case of blow-up solutions, the special choice of $T(x)$ automatically validates the condition in Step 4.
}.

\subsubsection{Logarithmic norms}
\label{section-CZ}
We give several notations used later.

\begin{dfn}[Logarithmic norms, cf. \cite{CZ2015}]\rm
For a squared matrix $A\in \mathbb{R}^{n\times n}$, define the matrix norm $m(A)$ by
\begin{equation*}
m(A) = \sup_{z\in \mathbb{R}^n, \|z\| =1}\|Az\|,
\end{equation*}
which in general depends on the norm $\|\cdot \|$ on $\mathbb{R}^n$.
The {\em logarithmic norm} of $A$ denoted by $l(A)$ is given by
\begin{equation*}
l(A) = \lim_{h\to +0}\frac{\|I+hA\|-1}{h}
\end{equation*}
and the {\em logarithmic minimum} of $A$ is given by
\begin{equation*}
m_l(A) = \lim_{h\to +0}\frac{m(I+hA)-1}{h}.
\end{equation*}
\end{dfn}

We gather several fundamental facts of $l(A)$ and $m_l(A)$ in the following lemma.
\begin{lem}[cf. \cite{CZ2015, CZ2016}]
\label{lem-log}
\begin{enumerate}
\item The limits in the definition of $l(A)$ and $m_l(A)$ exist and we have $m_l(A)=-l(-A)$.
\item For the Euclidean norm, we also have
\begin{align*}
l(A) &= \max\{\lambda \in {\rm Spec}((A+A^T)/2)\},\quad 
m_l(A) = \min\{\lambda \in {\rm Spec}((A+A^T)/2)\}.
\end{align*}
In particular, under the standard Euclidean norm, the following inequality holds:
\begin{equation*}
m_l(A)|x|^2 \leq x^T A x \leq l(A) |x|^2.
\end{equation*}
\end{enumerate}
\end{lem}

This lemma is used for validating (un)stable manifolds of (hyperbolic) equilibria discussed below.

\subsubsection{Stable manifold validations}
\label{section-stable-manifold}
First we review known validation procedures of stable manifolds of hyperbolic equilibria.
Consider the {\em smooth} vector field (\ref{ODE}) or (\ref{ODE-desing}) admitting a (smooth) change of coordinate $P_\mu \equiv c_N^{-1} : \mathbb{R}^{n_u+n_s}\to \mathbb{R}^n,\ z\equiv (a,b)\mapsto x$, such that (\ref{ODE}) is (locally) equivalent to
\begin{equation}
\label{ODE-coord}
\begin{cases}
\dot a = f^a(a,b; \mu), & \\
\dot b = f^b(a,b; \mu) &
\end{cases}
\quad \text{ with }\quad n_u + n_s = n.
\end{equation}
This is of our main interest for practical validations of solution trajectories.

\begin{prop}[Stable and unstable manifold validation, e.g., \cite{CZ2015, ZCov}]
\label{prop-stable-manifold}
Let $x_\ast = x_\ast(\mu)\in \mathbb{R}^n$ be a $\mu\in K$-parameterized family of equilibria for a smooth vector field (\ref{ODE-coord}), where $K\subset \mathbb{R}^k$ is a compact set.
Let $N$ be an $n$-dimensional $h$-set with $c_N\equiv P_\mu^{-1}$ containing $x_\ast$ such that its coordinate representation $c_N(N) = \overline{{\bf B}_{n_u}} \times \overline{{\bf B}_{n_s}}$ is characterized by the coordinate $(a,b)$ in $\mathbb{R}^{n_u+n_s}$ and that it is an isolating block for (\ref{ODE-coord}).
For $M > 1$, let
\begin{align*}
\overrightarrow{\mu_{s}} &= \overrightarrow{\mu_{s}}(N,K) = \sup_{(x,\mu)\in N\times K} \left\{ l\left(\frac{\partial f^b}{\partial b}(x;\mu) \right) + M \left\| \frac{\partial f^b}{\partial (a,\mu)}(x;\mu)\right\|\right\}, \\
\overrightarrow{\xi_{u}} &= \overrightarrow{\xi_{u}}(N,K) = \inf_{(x,\mu)\in N\times K}m_l\left( \frac{\partial f^a}{\partial a}(x;\mu) \right) - M\sup_{(x,\mu)\in N\times K} \left\| \frac{\partial f^a}{\partial (b,\mu)}(x;\mu) \right\|,\\
\overrightarrow{\mu_{ss}} &= \overrightarrow{\mu_{ss}}(N,K) = \sup_{(x,\mu)\in N\times K} \left\{ l\left(\frac{\partial f^b}{\partial b}(x;\mu) \right) + \frac{1}{M} \left\| \frac{\partial f^b}{\partial a}(x;\mu)\right\|\right\}, \\ 
\overrightarrow{\xi_{su}} &= \overrightarrow{\xi_{su}}(N,K) = \inf_{(x,\mu)\in N\times K}m_l\left( \frac{\partial f^a}{\partial a}(x;\mu) \right) - \frac{1}{M}\sup_{(x,\mu)\in N\times K} \left\| \frac{\partial f^a}{\partial b}(x;\mu) \right\|.
\end{align*}
We say that the vector field $(f^a, f^b)^T$ satisfies the {\em $M$-cone condition} in $N\times K$ if
\begin{align}
\label{rate-fss-1}
&\overrightarrow{\mu_{s}} < 0 < \overrightarrow{\xi_{u}},\\
\label{rate-fss-2}
&\overrightarrow{\mu_{ss}} < \overrightarrow{\xi_{u}},\quad \overrightarrow{\mu_{s}} < \overrightarrow{\xi_{su}}.
\end{align}
We assume the $M$-cone condition of $(f^a, f^b)^T$ in $N\times K$.
Then, for each $\mu\in K$, there are Lipschitzian functions $\sigma^u_\mu : \overline{{\bf B}_{n_u}} \to \overline{{\bf B}_{n_s}}$ and $\sigma^s_\mu : \overline{{\bf B}_{n_s}} \to \overline{{\bf B}_{n_u}}$ such that
\begin{equation*}
W^u(x_\ast(\mu))\cap N = \{c_N^{-1}(a,\sigma^u_\mu(a)) \mid a\in \overline{{\bf B}_{n_u}} \}, \quad W^s(x_\ast(\mu))\cap N = \{c_N^{-1}(\sigma^s_\mu(b),b) \mid b\in \overline{{\bf B}_{n_s}} \}.
\end{equation*}
Namely, the stable and unstable manifolds of $x_\ast$ are given by the graphs of Lipschitzian functions in $N$.
In particular, $\sigma^u_\mu$ is a horizontal disk and $\sigma^s_\mu$ is a vertical disk in $N$ for (\ref{ODE-coord}).
Moreover, functions $\sigma^u_\mu$ and $\sigma^s_\mu$ have the Lipschitz constants $1/M$ and are also continuous with respect to $\mu\in K$.
\end{prop}

This proposition gives not only the existence of stable and unstable manifolds of (hyperbolic) equilibria as graphs of Lipschitzian functions but also explicit criteria for their validations in {\em given} neighborhoods of equilibria.
The most remarkable point is that $W^s(x_\ast(\mu))\cap N$ is contained in the {\em stable $M$-cone} centered at $x_\ast(\mu)$:
\begin{equation*}
C^s_M(x_\ast; c_N^{-1}) = \{x = c_N^{-1}(a,b)\in \mathbb{R}^n\mid \|b\| \geq M\|a\|\}\text{ with }x_\ast = c_N^{-1}(0,0).
\end{equation*}
Similarly, $W^u(x_\ast(\mu))\cap N$ is contained in the {\em unstable $M$-cone} centered at $x_\ast(\mu)$:
\begin{equation*}
C^u_M(x_\ast; c_N^{-1}) = \{x = c_N^{-1}(a,b)\in \mathbb{R}^n\mid M\|b\| \leq \|a\|\}\text{ with }x_\ast = c_N^{-1}(0,0).
\end{equation*}

\begin{rem}[Validation of cone conditions in extended $h$-sets]
In practical validations, we verify cone conditions in larger $h$-sets than isolating blocks containing equilibria so that covering relations can be validated as easy as possible.
For example, assume that $N$ is an isolating block with $h$-set structure $c_N(N) = \overline{{\bf B}_{n_u}(0,R_u)}\times \overline{{\bf B}_{n_s}(0,R_s)}$ and radii $R_u, R_s > 0$ such that $N$ contains an equilibrium $x_\ast$.
Then, for given positive numbers $M=M_s$ and $\ell_s$, verify the $M_s$-cone condition in $c_N^{-1}\left( \overline{{\bf B}_{n_u}(0,R_u + \ell_s/M_s)}\times \overline{{\bf B}_{n_s}(0,R_s + \ell_s)} \right)$.
If the cone condition is validated, then the stable manifold $W^s(x_\ast)$ is contained in the union of stable cones
\begin{align*}
\{ x = c_N^{-1}(a,b) &\mid a\in \overline{{\bf B}_{n_u}(0,R_u + \ell_s/M_s)}, b\in \overline{{\bf B}_{n_s}(0,R_s + \ell_s)} \\
&\quad \left. \text{ with } M_s\|a-a_0\| \leq \|b-b_0\| \text{ for some }(a_0,b_0)\in c_N(N) \right\}.
\end{align*}
Similarly, for given positive numbers $M=M_u$ and $\ell_u$, verify the $M_u$-cone condition in $c_N^{-1}\left( \overline{{\bf B}_{n_u}(0,R_u + \ell_u)}\times \overline{{\bf B}_{n_s}(0,R_s + \ell_u/M_u)} \right)$.
If the cone condition is validated, then the unstable manifold $W^u(x_\ast)$ is contained in the union of unstable cones
\begin{align*}
\{ x = c_N^{-1}(a,b) &\mid a\in \overline{{\bf B}_{n_u}(0,R_u + \ell_u)}, b\in \overline{{\bf B}_{n_s}(0,R_s + \ell_u/M_u)} \\
&\quad \left. \text{ with } \|a-a_0\| \geq M_u\|b-b_0\| \text{ for some }(a_0,b_0)\in c_N(N) \right\}.
\end{align*}
The covering relations are then actually verified for such larger $h$-sets.
\end{rem}

\subsubsection{Lyapunov functions}

Next we construct a Lyapunov function on stable manifolds determined in an explicit domain which can be validated by rigorous numerics.

\begin{prop}[Lyapunov functions, cf. \cite{MHY2016}]
\label{prop-Lyapunov}
Let $x_\ast = x_\ast(\mu)\in \mathbb{R}^n$ be a $\mu\in K$-parameterized family of equilibria for a smooth vector field (\ref{ODE-coord}), where $K\subset \mathbb{R}^k$ is a compact set.
Let $N$ be a $n$-dimensional $h$-set containing $x_\ast(\mu)$ for all $\mu\in K$.
Assume that there is a real symmetric matrix $Y$ such that the matrix
\begin{equation}
\label{neg-def-A}
{\bf A}(x) := D\tilde f(x,\mu)^T Y + Y D\tilde f(x,\mu),\quad \tilde f = (f^a, f^b)^T,
\end{equation}
is strictly negative definite for all $(x,\mu)\in N\times K$.
Then, for each $\mu\in K$, $L_\mu(x) := (x-x_\ast(\mu))^T Y (x-x_\ast(\mu))$ is a Lyapunov function on $N$.
That is, $L_\mu$ satisfies the following properties;
\begin{itemize}
\item $\frac{dL_\mu}{d\tau}(x(\tau; \mu))|_{\tau=0} \leq 0$ for any $x=x(0)\in N$;
\item $\frac{dL_\mu}{d\tau}(x(\tau; \mu))|_{\tau=0} = 0$ implies $x=x_\ast$.
\end{itemize}
Moreover, if $x_\ast(\mu)$ is hyperbolic for all $\mu\in K$, then $L_\mu$ is also smooth with respect to $\mu$.
\end{prop}
In \cite{TMSTMO, MT2017}, Lyapunov functions around stable equilibria are applied to obtaining an upper bound of the maximal existence time $t_{\max}$ of divergent solutions.
The estimate relies on the monotonous behavior of trajectories near hyperbolic equilibria, and we expect that Lyapunov functions around equilibria, even for saddle type, provide finite bounds of $t_{\max}$.
In other words, hyperbolic equilibria at infinity would induce finite-time singularities.
In fact, hyperbolic and stable equilibria at infinity induce explicit bounds of $t_{\max}$ (\cite{TMSTMO, MT2017}).
This situation is expected to be also the case of degeneracy-inducing points which are hyperbolic for (\ref{ODE-desing}).
\par 
However, it immediately turns out that direct estimates of $t_{\max}$ in terms of Lyapunov functions fail in case that equilibria are saddles.
Indeed, the Lyapunov function $L$ may attain negative value in $N\setminus \{x_\ast(\mu)\}$, which violates upper bound estimates of $t_{\max}$ discussed in \cite{TMSTMO, MT2017}.
Now notice that our interests for estimates of $t_{\max}$ are not the whole neighborhood of equilibria, but only solutions on stable manifolds of hyperbolic equilibria.
This observation leads to an implementation of Lyapunov functions {\em only on stable manifolds of hyperbolic equilibria}, which are stated in the following proposition and are compatible with the form of $T(x)$.

\begin{prop}[Lyapunov functions on stable manifolds]
\label{prop-Lyapunov-sm}
Let $x_\ast = x_\ast(\mu)\in \mathbb{R}^n$ be an equilibrium for a smooth vector field (\ref{ODE-coord}) parameterized by $\mu\in K\subset\mathbb{R}^k$ with the change of coordinate $P_\mu$ such that $P_\mu(0,0) = x_\ast(\mu)$.
Let $N$ be an $n$-dimensional $h$-set containing $x_\ast(\mu)$ such that its coordinate representation $c_N(N) = \overline{{\bf B}_{n_u}} \times \overline{{\bf B}_{n_s}}$ is characterized by the coordinate $(a,b)$ in $\mathbb{R}^{n_u+n_s}$ and that it is an isolating block for (\ref{ODE-coord}).
Assume that $(f^a,f^b)^T$ satisfies the $M$-cone condition on $N$.
Further assume that
\begin{equation}
\label{rate-stable-sm}
\overrightarrow{\mu_{s, W^s}} := \sup_{z\in N\times K}\left\{ \frac{M^2+1}{M^3} \left\|\frac{\partial f^b}{\partial a}(z)\right\| + \frac{1}{M^2}m\left(\frac{\partial f^b}{\partial b}(z)\right) + l\left(\frac{\partial f^b}{\partial b}(z)\right)\right\} < 0.
\end{equation}
Then
\begin{equation}
\label{Lyapunov-sm}
L_{s}(a,b;\mu) := \|\sigma^s_\mu(b)\|^2 + \|b\|^2 
\end{equation}
is a Lyapunov function for (\ref{ODE-coord}) on $W^s(x_\ast(\mu))\cap N$, where $\sigma^s_\mu$ is the Lipschitz function such that $W^s(x_\ast(\mu))\cap N = \{(\sigma^s_\mu(b),b)\mid b\in \overline{{\bf B}_{n_s}}\}$.
See Proposition \ref{prop-stable-manifold}.
\end{prop}

\begin{proof}
First note that $\|\sigma^s_\mu(b)\|\leq M^{-1}\|b\|$ holds for all $b\in \overline{{\bf B}_{n_s}}$, since $W^s(x_\ast(\mu))$ is contained in the stable $M$-cone $C^s_M(x_\ast; P_\mu)$.
Differentiating $L_s$ with respect to $\tau$ along trajectories $z(\tau;\mu) = (a(\tau), b(\tau), \mu) \in W^s(x_\ast(\mu))$ evaluated at $\tau = 0$, we have
\begin{align}
\notag
\frac{1}{2}\frac{dL_s}{d\tau}(z(\tau;\mu))|_{\tau = 0} &= \left\langle \frac{\partial \sigma^s_\mu}{\partial b}(b) \frac{db}{d\tau}, \sigma^s(b)\right\rangle + \left\langle \frac{db}{d\tau}, b\right\rangle = \left\langle \frac{\partial \sigma^s_\mu}{\partial b}(b) f^b(z), \sigma^s_\mu(b)\right\rangle + \left\langle f^b(z), b\right\rangle \\
\notag
	&= \left\langle \frac{\partial \sigma^s_\mu}{\partial b}(b) \left( \frac{\partial f^b}{\partial a}(z_b)\frac{\partial \sigma^s_\mu}{\partial b}(b_{\xi})b + \frac{\partial f^b}{\partial b}(z_b)b\right), \sigma^s_\mu(b)\right\rangle \\
\notag
	&\quad\quad+ \left\langle \left( \frac{\partial f^b}{\partial a}(z_b)\frac{\partial \sigma^s_\mu}{\partial b}(b_{\xi})b + \frac{\partial f^b}{\partial b}(z_b)b \right), b\right\rangle \\
\notag
	&= \left\langle \frac{\partial \sigma^s_\mu}{\partial b}(b) \left( \frac{\partial f^b}{\partial a}(z_b)\frac{\partial \sigma^s_\mu}{\partial b}(b_{\xi})b + \frac{\partial f^b}{\partial b}(z_b)b\right), \frac{\partial \sigma^s_\mu}{\partial b}(b_{\xi}')b \right\rangle \\
\notag
	&\quad\quad+ \left\langle \left( \frac{\partial f^b}{\partial a}(z_b)\frac{\partial \sigma^s_\mu}{\partial b}(b_{\xi})b + \frac{\partial f^b}{\partial b}(z_b)b \right), b\right\rangle \\
\notag
	&\leq \frac{1}{M^3} \left\|\frac{\partial f^b}{\partial a}\right\| \|b\|^2 + \frac{1}{M^2}m\left(\frac{\partial f^b}{\partial b}\right) \|b\|^2\\
\notag
	&\quad\quad + \frac{1}{M} \left\|\frac{\partial f^b}{\partial a}\right\| \|b\|^2 + l\left(\frac{\partial f^b}{\partial b}\right)\|b\|^2\\
\notag
	&= \left[ \frac{M^2+1}{M^3} \left\|\frac{\partial f^b}{\partial a}\right\| + \frac{1}{M^2}m\left(\frac{\partial f^b}{\partial b}\right) + l\left(\frac{\partial f^b}{\partial b}\right) \right] \|b\|^2\\
\label{Lyapunov-stable-1}
	&\equiv \overrightarrow{\mu_{s, W^s}} \|b\|^2 \leq \overrightarrow{\mu_{s, W^s}} (\lambda \|b\|^2 + (1-\lambda)M^2\|\sigma^s_\mu(b)\|^2 )
\end{align}
holds for all $\lambda \in [0,1]$, where $z_b\in c_N(N)\times K$ and $b_\xi, b_{\xi'}\in \overline{{\bf B}_{n_s}}$.
The last inequality follows from $\overrightarrow{\mu_{s, W^s}} < 0$ and $\|\sigma^s_\mu(b)\|\leq M^{-1}\|b\|$.
Now set $\lambda = M^2/(M^2+1) < 1$ so that $\lambda = (1-\lambda)M^2$, in which case we have
\begin{equation*}
\frac{1}{2}\frac{dL_s}{d\tau}(z(\tau;\mu))|_{\tau = 0} \leq \frac{M^2}{M^2+1}\overrightarrow{\mu_{s, W^s}} L_s(z)
\end{equation*}
and hence $L_s(z)$ is a Lyapunov function on $W^s(x_\ast)\cap N$ for all $\mu\in K$.
\end{proof}

\begin{rem}
By the cone condition, $\overrightarrow{\mu_{s}} < 0$ are already assumed, which indicates that
\begin{equation*}
\frac{M^2+1}{M^3} \left\|\frac{\partial f^b}{\partial a}(z)\right\| + l\left(\frac{\partial f^b}{\partial b}(z)\right) < 0
\end{equation*}
is mostly satisfied for large $M$ and $z\in N\times K$, since $\frac{M^2+1}{M^3} \sim M^{-1}$.
Note that, by the Stable Manifold Theorem, $M$ can be chosen sufficiently large if $z\in W^s(x_\ast(\mu))$ is sufficiently close to $x_\ast(\mu)$.
The assumption $\overrightarrow{\mu_{s,W^s}} < 0$ takes the slope of stable manifolds in $N$ into account.
In fact, if $W^s(x_\ast(\mu))$ is given by $\{a = \sigma^s_\mu(b) \equiv 0\}$, then $M$ can be chosen arbitrarily large, and hence $\overrightarrow{\mu_{s,W^s}} < 0$ is automatically satisfied. 
\end{rem}

As a special case, we consider the following form of vector fields in a vicinity of an equilibrium:
\begin{equation}
\label{near-diagonal}
a' = Aa + \tilde f^a(a,b;\mu),\quad b' = Bb + \tilde f^b(a,b;\mu),
\end{equation}
where $A$ and $B$ are squared matrices such that
\begin{equation*}
{\rm Re}\lambda \geq \lambda_A,\quad \forall \lambda\in {\rm Spec}(A),\quad 
{\rm Re}\mu \leq \mu_B,\quad \forall \mu\in {\rm Spec}(B)
\end{equation*}
hold for some $\lambda_A > 0$ and $\mu_B < 0$.
Then
\begin{align}
\notag
\frac{1}{2}\frac{dL_s}{d\tau}(z(\tau;\mu))|_{\tau = 0} &= \left\langle \frac{\partial \sigma^s_\mu}{\partial b}(b) \frac{db}{d\tau}, \sigma^s_\mu(b)\right\rangle + \left\langle \frac{db}{d\tau}, b\right\rangle\\
\notag
	&= \left\langle \frac{\partial \sigma^s_\mu}{\partial b}(b) (Bb + \tilde f^b(z)), \sigma^s_\mu(b)\right\rangle + \left\langle Bb + \tilde f^b(z), b\right\rangle \\
\notag
	&= \left\langle \frac{\partial \sigma^s}{\partial b}(b) \left( Bb + \frac{\partial \tilde f^b}{\partial a}(z_b)\frac{\partial \sigma^s_\mu}{\partial b}(b_{\xi})b + \frac{\partial \tilde f^b}{\partial b}(z_b)b\right), \sigma^s_\mu(b)\right\rangle \\
\notag
	&\quad\quad+ \left\langle \left( Bb + \frac{\partial \tilde f^b}{\partial a}(z_b)\frac{\partial \sigma^s_\mu}{\partial b}(b_{\xi})b + \frac{\partial \tilde f^b}{\partial b}(z_b)b \right), b\right\rangle \\
\notag
	&= \left\langle \frac{\partial \sigma^s_\mu}{\partial b}(b) \left( Bb+\frac{\partial \tilde f^b}{\partial a}(z_b)\frac{\partial \sigma^s_\mu}{\partial b}(b_{\xi})b + \frac{\partial \tilde f^b}{\partial b}(z_b)b\right), \frac{\partial \sigma^s_\mu}{\partial b}(b_{\xi}')b \right\rangle \\
\notag
	&\quad\quad+ \left\langle \left( Bb+\frac{\partial \tilde f^b}{\partial a}(z_b)\frac{\partial \sigma^s_\mu}{\partial b}(b_{\xi})b + \frac{\partial \tilde f^b}{\partial b}(z_b)b \right), b\right\rangle \\
\notag
	&\leq \frac{1}{M^3} \left\|\frac{\partial \tilde f^b}{\partial a}\right\| \|b\|^2 + \frac{1}{M^2}m\left(B+\frac{\partial \tilde f^b}{\partial b}\right) \|b\|^2\\
\notag
	&\quad\quad + \frac{1}{M} \left\|\frac{\partial \tilde f^b}{\partial a}\right\| \|b\|^2 + \left\{ \mu_B + l\left(\frac{\partial \tilde f^b}{\partial b}\right) \right\} \|b\|^2\\
\notag
	&= \left[ \frac{M^2+1}{M^3} \left\|\frac{\partial \tilde f^b}{\partial a}\right\| + \frac{1}{M^2}m\left(B+\frac{\partial \tilde f^b}{\partial b}\right) +\left\{\mu_B + l\left(\frac{\partial \tilde f^b}{\partial b}\right) \right\} \right] \|b\|^2\\
\notag
	&\leq \left[ \frac{M^2+1}{M^3} \left\|\frac{\partial \tilde f^b}{\partial a}\right\| + \frac{1}{M^2}m\left(\frac{\partial \tilde f^b}{\partial b}\right) + l\left(\frac{\partial \tilde f^b}{\partial b}\right) + \frac{M^2+1}{M^2}\mu_B \right] \|b\|^2\\
\label{Lyapunov-stable-2}
	&\equiv \overrightarrow{\mu_{s, W^s}} \|b\|^2 \leq \overrightarrow{\mu_{s, W^s}} (\lambda \|b\|^2 + (1-\lambda)M^2\|\sigma^s_\mu(b)\|^2 )
\end{align}
holds for all $\lambda \in [0,1]$.
The last inequality follows from $\overrightarrow{\mu_{s, W^s}} < 0$ and $\|\sigma^s(b)\|\leq M^{-1}\|b\|$.
Now set $\lambda = M^2/(M^2+1) < 1$ so that $\lambda = (1-\lambda)M^2$, in which case we have
\begin{equation*}
\frac{1}{2}\frac{dL_s}{d\tau}(z(\tau;\mu))|_{\tau = 0} \leq \frac{M^2}{M^2+1}\overrightarrow{\mu_{s, W^s}} L_s(z).
\end{equation*}

We can construct Lyapunov functions on unstable manifolds in the similar way, which is stated as follows.

\begin{prop}[Lyapunov functions on unstable manifolds]
\label{prop-Lyapunov-um}
Let $x_\ast\in \mathbb{R}^n$ be an equilibrium for a smooth vector field (\ref{ODE-coord}).
Let $N$ be a $n$-dimensional $h$-set containing $x_\ast$ such that its coordinate representation $c_N(N) = \overline{{\bf B}_{n_u}} \times \overline{{\bf B}_{n_s}}$ is characterized by the coordinate $(a,b)$ in $\mathbb{R}^{n_u+n_s}$.
Assume that $(f^a,f^b)^T$ satisfies the $M$-cone condition on $N$.
Further assume that
\begin{equation}
\label{rate-stable-um}
\overrightarrow{\xi_{u, W^u}} := \inf_{z\in N\times K}\left\{ m_l\left(\frac{\partial f^a}{\partial a}\right) - \frac{M^2+1}{M^3} \left\|\frac{\partial f^a}{\partial b}\right\| - \frac{1}{M^2}m\left(\frac{\partial f^a}{\partial a}\right)  \right\} > 0.
\end{equation}
Then
\begin{equation}
\label{Lyapunov-um}
L_u(a,b;\mu) := -\left(\|a\|^2 + \|\sigma^u_\mu(a)\|^2 \right) 
\end{equation}
is a Lyapunov function for (\ref{ODE-coord}) on $W^u(x_\ast(\mu))\cap N$, where $\sigma^u_\mu$ is the Lipschitz function such that $W^u(x_\ast(\mu))\cap N = \{(a, \sigma^u_\mu(a))\mid a\in \overline{{\bf B}_{n_u}}\}$.
See Proposition \ref{prop-stable-manifold}.
\end{prop}

\begin{proof}
First note that $\|\sigma^u_\mu(a)\|\leq M^{-1}\|a\|$ holds for all $a\in \overline{{\bf B}_{n_u}}$, since $W^u(x_\ast(\mu))$ is contained in the unstable $M$-cone
\begin{equation*}
C^u_M(x_\ast; P_\mu) = \{x = P_\mu(a,b)\in \mathbb{R}^n\mid \|a\| \geq M\|b\|\}\text{ with }x_\ast = P_\mu(0,0).
\end{equation*}
Consider the differential of $\tilde L_u\equiv -L_u$ along the solution trajectory $z = z(\tau;\mu)$, which is as follows:
\begin{align}
\notag
\frac{1}{2}\frac{d\tilde L_u}{d \tau}(z(\tilde \tau;\mu))|_{\tilde \tau = 0} &=  \left\langle \frac{da}{d \tau}, a\right\rangle + \left\langle \frac{\partial \sigma^u_\mu}{\partial a}(a) \frac{da}{d\tau}, \sigma^u_\mu(a)\right\rangle\\
\notag
	&= \left\langle f^a(z), a\right\rangle + \left\langle \frac{\partial \sigma^u_\mu}{\partial a}(a) f^a(z), \sigma^u_\mu(a)\right\rangle  \\
\notag
	&= \left\langle \left( \frac{\partial f^a}{\partial a}(z_a)a + \frac{\partial f^a}{\partial b}(z_a)\frac{\partial \sigma^u_\mu}{\partial a}(a_{\xi})a \right), a\right\rangle \\
\notag
	&\quad\quad + \left\langle \frac{\partial \sigma^u_\mu}{\partial a}(a) \left( \frac{\partial f^a}{\partial a}(z_a)a + \frac{\partial f^a}{\partial b}(z_a)\frac{\partial \sigma^u_\mu}{\partial a}(a_{\xi})a \right), \sigma^u_\mu(a)\right\rangle \\
\notag
	&= \left\langle \left( \frac{\partial f^a}{\partial a}(z_a)a + \frac{\partial f^a}{\partial b}(z_a)\frac{\partial \sigma^u_\mu}{\partial a}(a_{\xi})a \right), a\right\rangle \\
\notag
	&\quad\quad + \left\langle \frac{\partial \sigma^u_\mu}{\partial a}(a) \left( \frac{\partial f^a}{\partial a}(z_a)a + \frac{\partial f^a}{\partial b}(z_a)\frac{\partial \sigma^u_\mu}{\partial a}(a_{\xi})a \right), \frac{\partial \sigma^u_\mu}{\partial a}(a_\xi')a\right\rangle \\
\notag
	&\geq  m_l\left(\frac{\partial f^a}{\partial a}\right)\|a\|^2 - \frac{1}{M} \left\|\frac{\partial f^a}{\partial b}\right\| \|a\|^2\\
\notag
	&\quad\quad - \frac{1}{M^2}m\left(\frac{\partial f^a}{\partial a}\right) \|a\|^2 - \frac{1}{M^3} \left\|\frac{\partial f^a}{\partial b}\right\| \|a\|^2\\
\notag
	&= \left[ m_l\left(\frac{\partial f^a}{\partial a}\right) - \frac{M^2+1}{M^3} \left\|\frac{\partial f^a}{\partial b}\right\| - \frac{1}{M^2}m\left(\frac{\partial f^a}{\partial a}\right) \right] \|a\|^2\\
\label{Lyapunov-unstable-1}
	&\equiv \overrightarrow{\xi_{u, W^u}} \|a\|^2 \geq \overrightarrow{\xi_{u, W^u}} (\lambda \|a\|^2 + (1-\lambda)M^2\|\sigma^u_\mu(a)\|^2 )
\end{align}
holds for all $\lambda \in [0,1]$, where $z_a\in c_N(N)\times K$ and $a_\xi, a_{\xi'}\in \overline{{\bf B}_{n_u}}$.
The last inequality follows from $\overrightarrow{\xi_{u, W^u}} > 0$ and $\|\sigma^u_\mu(a)\|\leq M^{-1}\|a\|$.
Now set $\lambda = M^2/(M^2+1) < 1$ so that $\lambda = (1-\lambda)M^2$, in which case we have
\begin{equation*}
\frac{1}{2}\frac{d\tilde L_u}{d\tau}(z(\tilde \tau; \mu))|_{\tau = 0} \geq \frac{M^2}{M^2+1}\overrightarrow{\xi_{u, W^u}} \tilde L_u(z)
\end{equation*}
and hence $\tilde L_u(z)$ is a Lyapunov function on $W^u(x_\ast(\mu))\cap N$ for all $\mu\in K$.
Obviously $L_u$ also satisfies all requirements of Lyapunov functions.
\end{proof}


The corresponding estimate for the desingularized vector field of the form (\ref{near-diagonal}) is as follows:

\begin{align}
\notag
\frac{1}{2}\frac{d\tilde L_u}{d \tau}(z(\tilde \tau;\mu))|_{\tilde \tau = 0} &=  \left\langle \frac{da}{d \tau}, a\right\rangle + \left\langle \frac{\partial \sigma^u_\mu}{\partial a}(a) \frac{da}{d\tau}, \sigma^u_\mu(a)\right\rangle\\
\notag
	&= \left\langle Aa+ \tilde f^a(z), a\right\rangle + \left\langle \frac{\partial \sigma^u}{\partial a}(a) (Aa+\tilde f^a(z)), \sigma^u_\mu(a)\right\rangle  \\
\notag
	&= \left\langle \left( Aa+\frac{\partial \tilde f^a}{\partial a}(z_a)a + \frac{\partial \tilde f^a}{\partial b}(z_a)\frac{\partial \sigma^u_\mu}{\partial a}(a_{\xi})a \right), a\right\rangle \\
\notag
	&\quad\quad + \left\langle \frac{\partial \sigma^u_\mu}{\partial a}(a) \left( Aa+\frac{\partial \tilde f^a}{\partial a}(z_a)a + \frac{\partial \tilde f^a}{\partial b}(z_a)\frac{\partial \sigma^u_\mu}{\partial a}(a_{\xi})a \right), \sigma^u_\mu(a)\right\rangle \\
\notag
	&= \left\langle \left( Aa+\frac{\partial \tilde f^a}{\partial a}(z_a)a + \frac{\partial \tilde f^a}{\partial b}(z_a)\frac{\partial \sigma^u_\mu}{\partial a}(a_{\xi})a \right), a\right\rangle \\
\notag
	&\quad\quad + \left\langle \frac{\partial \sigma^u_\mu}{\partial a}(a) \left(Aa+ \frac{\partial \tilde f^a}{\partial a}(z_a)a + \frac{\partial \tilde f^a}{\partial b}(z_a)\frac{\partial \sigma^u_\mu}{\partial a}(a_{\xi})a \right), \frac{\partial \sigma^u_\mu}{\partial a}(a_\xi')a\right\rangle \\
\notag
	&\geq  \lambda_A \|a\|^2 - m\left(\frac{\partial \tilde f^a}{\partial a}\right)\|a\|^2 - \frac{1}{M} \left\|\frac{\partial \tilde f^a}{\partial b}\right\| \|a\|^2\\
\notag
	&\quad\quad - \frac{1}{M^2}m\left(\frac{\partial \tilde f^a}{\partial a}\right) \|a\|^2 - \frac{1}{M^3} \left\|\frac{\partial \tilde f^a}{\partial b}\right\| \|a\|^2\\
\notag
	&= \left[ \lambda_A - m\left(\frac{\partial \tilde f^a}{\partial a}\right) - \frac{M^2+1}{M^3} \left\|\frac{\partial \tilde f^a}{\partial b}\right\| - \frac{1}{M^2}m\left(A+\frac{\partial \tilde f^a}{\partial a}\right) \right] \|a\|^2\\
\label{Lyapunov-unstable-2}
	&\equiv \overrightarrow{\xi_{u, W^u}} \|a\|^2 \geq \overrightarrow{\xi_{u, W^u}} (\lambda \|a\|^2 + (1-\lambda)M^2\|\sigma^u_\mu(a)\|^2 ).
\end{align}

\subsubsection{Validation procedure of trajectories with finite-time singularities}

Now we are ready to validate trajectories with finite-time singularity.
Consider (\ref{ODE}) admitting the desingularized vector field (\ref{ODE-desing}) under the time-scale desingularization (\ref{time-desing}).

Main target is the trajectory through a point $x_\ast$ under (\ref{ODE}).
We {\em assume} that $x_\ast$ is an equilibrium of (\ref{ODE-desing}), but {\em do not assume} that it is an equilibrium of (\ref{ODE}).
The following algorithm is a topological validation.

\begin{alg}[Details of Algorithm \ref{alg-asym}]
\label{alg-finite}
Consider (\ref{ODE}) which has a degeneracy-inducing point $x_\ast$.
Let $K\subset \mathbb{R}^k$ be a compact set homeomorphic to $\overline{{\bf B}_{\tilde k}}$ with $\dim K \equiv \tilde k \leq k$. 
\begin{enumerate}
\item Validate $x_\ast = x_\ast(\mu)$ as a $\mu\in K\subset \mathbb{R}^k$-parameterized family of  equilibria for the desingularized system (\ref{ODE-desing}) in terms of isolating blocks $N\subset \mathbb{R}^n$ satisfying $M$-cone conditions.
If possible, validate that $x_\ast$ is a hyperbolic equilibrium for (\ref{ODE-desing}).
\item For given initial $h$-set $N_0$ in $\mathbb{R}^n\times \mathbb{R}^k$, construct a sequence of covering relations\footnote{
$h$-sets $N_i$ ($i=0,\cdots, m$) are chosen so that $\pi_k N_i = K$, where $\pi_k : \mathbb{R}^n\times \mathbb{R}^k \to \mathbb{R}^k$ denotes the orthogonal projection onto parameter variables.
}
\begin{equation}
\label{verify-cov}
%
N_0\overset{\Phi^{des}(\tau_0, \cdot)}{\Longrightarrow}N_1\overset{\Phi^{des}(\tau_1, \cdot)}{\Longrightarrow}N_2 \overset{\Phi^{des}(\tau_2, \cdot)}{\Longrightarrow} \cdots \overset{\Phi^{des}(\tau_{m-1}, \cdot)}{\Longrightarrow}N_m \overset{\Phi^{des}(\tau_m, \cdot)}{\Longrightarrow}N\times K
\end{equation}
for some $\tau_0,\cdots, \tau_m > 0$, where $\Phi^{des}$ denotes the parameterized flow $\Phi^{des} : \mathbb{R}\times \mathbb{R}^n\times \mathbb{R}^k \to \mathbb{R}^n\times \mathbb{R}^k$ given as $\Phi^{des}(t,x,\mu) = (\varphi^{des}_\mu(t,x), \mu)$.
This procedure shows the existence of a point $(x_0,\mu_0)\in N_0$ such that $\varphi^{des}_{\mu_0}(\bar \tau, x_0) \in W^s(x_\ast; \varphi^{des}_{\mu_0})\cap N$, where $\bar \tau = \sum_{j=0}^m \tau_j$.
Let $x(\tau)$ be the trajectory satisfying $x(0) = x_0$.
Verify whether $T(x)$ has an identical sign along the trajectory $\{x(\tau)\mid \tau\in [0,\bar \tau]\}$.
\item Check if $T(x)$ has an identical signature for all points in $(W^s(x_\ast; \varphi^{des}_{\mu_0})\cap N) \setminus \{x_\ast\}$.
See also Remark \ref{rem-cov-block}.
\item Compute 
\begin{equation*}
t_{\max} = \int_0^\infty \frac{d\tau}{T(x(\tau))}
\end{equation*}
around the degeneracy-inducing point, where $\tau = 0$ is supposed to correspond to $t=0$.
\end{enumerate}
If the validated trajectory involves the unstable manifold of $x_\ast$, then replace \lq\lq $W^s(x_\ast; \varphi^{des}_{\mu_0})$" in Step 3 by \lq\lq $W^u(x_\ast; \varphi^{des}_{\mu_0})$", and \lq\lq $t_{\max} = \int_0^\infty \frac{d\tau}{T(x(\tau))}$" by
\begin{equation*}
t_{\min} = -\int_{-\infty}^0 \frac{d\tau}{T(x(\tau))},
\end{equation*}
respectively.
\end{alg}
\begin{thm}
\label{thm-existence-rigorous}
Suppose that all operations in Algorithm \ref{alg-finite} are succeeded and that $t_{\max} < \infty$.
Then there is a solution ${\bf x}(t; {\bf x}_0)$ of (\ref{ODE}) with ${\bf x}(0;{\bf x}_0)= {\bf x}_0 = (x_0, \mu_0)\in N_0$ such that ${\bf x}(t_{\max}) = {\bf x}_\ast$.
Namely, the trajectory $\{x(t;x_0(\mu_0))\}_{t\geq 0}$ goes through $x_\ast(\mu_0)$ at $t = t_{\max}(\mu_0)$.
\end{thm}

\begin{proof}
Success of arguments in Algorithm \ref{alg-finite} yields the following consequence from Proposition \ref{WZ-heteroclinic}:
there is a solution $x(\tau; x_0(\mu_0))$ for (\ref{ODE-desing}) such that $x(\tau; x_0)\to x_\ast$.
Since $x(\bar \tau) \in N$, then Steps 2 and 3 in Algorithm \ref{alg-finite} indicate that $T(x)$ has an identical sign along $\{x(\tau; x_0)\mid \tau\in [0,\infty)\}$.
This condition shows that (\ref{ODE}) and (\ref{ODE-desing}) are orbitally equivalent along $\{x(\tau; x_0)\mid \tau\in [0,\infty)\}$, and hence the solution $x(t;x_0)$ in the original $t$-timescale also goes to $x_\ast$ whose arrival time is give by $t=t_{\max}$ by the definition of (\ref{time-desing}).
\end{proof}

Behavior of $x(t;x_0(\mu_0))$ {\em after $t=t_{\max}$} depends on correspondence of the trajectory to that in the original problem.


\begin{rem}
\label{rem-cov-block}
A typical choice of $N_0$ is as follows for validating connecting orbits for (\ref{ODE-desing}).
Let $\tilde N$ an isolating block for $\Phi^{des}$ homeomorphic to $\subset \mathbb{R}^{n_u+n_u}\times \mathbb{R}^k$ such that a parameter family of equilibria $\{\tilde {\bf x}_\ast\} = \{\tilde {\bf x}_\ast(\mu)\}_{\mu\in K}$ is contained and that the cone condition is satisfied.
Recall that $n_u$ is the number of eigenvalues of $f^a$; namely the number of eigenvalues of the Jacobian matrix with positive real part, and for each $\mu\in K$, $n_u$ determines the dimension of $W^u(\tilde {\bf x}_\ast(\mu); \varphi^{des}_\mu)$ in $\tilde N$.
Now we assume that there is a strong deformation retraction $r$ of $\tilde N$ such that the set $W^u(\{\tilde {\bf x}_\ast\}; \Phi^{des}) \cap r(\tilde N)$ is a strong deformation retract of $W^u(\{\tilde {\bf x}_\ast\}; \Phi^{des}) \cap \tilde N$ and that $W^u(\{\tilde {\bf x}_\ast\}; \Phi^{des}) \cap r(\tilde N)$ is homeomorphic to $\overline{{\bf B}_u} \times \overline{{\bf B}_{\tilde u}}$, where $u$ is the expanding dimension in (\ref{verify-cov}) and $\tilde u \geq 0$.
In this case, the standard consequence of covering relations shows the existence of a point in $W^u(\{\tilde {\bf x}_\ast\}; \Phi^{des}) \cap r(\tilde N)$ traveling all $h$-sets via covering relations (\ref{verify-cov}).
So, we choose $N_0 := r(\tilde N)$.
\par
An example of such a retract is {\em the exit} of $\tilde N$ in a certain direction. 
It is well-known that the exit $\tilde N^{\exit}$ of an isolating block $\tilde N$ for $\Phi^{des}$ is a strong deformation retract of $\tilde N$ itself (e.g., \cite{Smo}), where the intersection $W^u(\{\tilde {\bf x}_\ast\}; \Phi^{des})\cap \tilde N^{\exit}$ is also a strong deformation retract of $W^u(\{\tilde {\bf x}_\ast\}; \Phi^{des})\cap \tilde N$ since $W^u(\{\tilde {\bf x}_\ast\}; \Phi^{des})$ is a horizontal disk (Proposition \ref{prop-stable-manifold}).
Moreover, by assumptions of $\tilde N$, the horizontal disk $W^u(\{\tilde {\bf x}_\ast\}; \Phi^{des})$ is homeomophic to $\overline{{\bf B}_{n_u}} \times \overline{{\bf B}_{\dim K}}$. 
Therefore, the specification of the exit induces a strong deformation retract of $\overline{{\bf B}_{n_u}} \times \overline{{\bf B}_{\dim K}}$ to $\{a_0\} \times \overline{{\bf B}_{\dim K}}$, where $a_0\in \partial {\bf B}_{n_u}$\footnote{
In this case, the exit $N_0$ is homeomorphic to $\overline{{\bf B}_{n_s}} \times \overline{{\bf B}_{\dim K}}$, and $W^u(\{\tilde {\bf x}_\ast\}; \Phi^{des}) \cap N_0$ is homeomorphic to $\overline{{\bf B}_{\dim K}}$. 
Moreover, $\dim K = u + \tilde u$ holds.
}.
\end{rem}

\subsubsection{Details 1 : Validation of sign of $T(x)$ along trajectories. Case study in $2$-dimensional dynamical systems}

Step 3 in Algorithm \ref{alg-finite} involves the verification of orbital equivalence between (\ref{ODE}) and (\ref{ODE-desing}) along validated trajectories.
The condition is concerned with $W^s(x_\ast; \varphi^{des}_\mu)$ and hence we need to determine the location of $W^s(x_\ast; \varphi^{des}_\mu)$, at least, near $x_\ast$.
Our stable manifold validation is done by cone conditions as well as a change of coordinates, which is mainly calculated by eigenvectors of the linearized matrix of vector field at $x_\ast$.
The eigenvectors as well as the property of cones completely detect the enclosure of location of stable manifolds.
Cone condition indicates that the set $(W^s(x_\ast;\varphi^{des}_\mu)\cap N)\setminus \{x_\ast\}$ is completely included in the stable $M$-cone $C^s_M(x_\ast; P)\equiv \{x = P(a,b)\in N\mid M\|a\| \leq \|b\|\}$ in (\ref{ODE-coord}) with the coordinate $x = P(a,b)$ (e.g., \cite{CZ2015, Mat}), where $0\in \mathbb{R}^{n_u + n_s}$ denotes the representation of $x_\ast$ in $(a,b)$-coordinate.
Therefore, in $(a,b)$-coordinate, Step 3 in Algorithm \ref{alg-finite} is reduced to the following verifiable criterion:
\begin{itemize}
\item Check if, in $(a,b)$-coordinate, $T(x)\equiv T(a,b)$ has an identical signature in $C^s_M(x_\ast; P)\cap N$.
\end{itemize}
{\em The Stable Manifold Theorem} shows that the stable manifold $W^s(x_\ast; \varphi^{des}_\mu)$ is tangent to the eigenspace generated by eigenvectors  $V_s$ associated with eigenvalues with negative real part at $x_\ast$.
It makes sense for (\ref{ODE-desing}) since it is sufficiently smooth including $x_\ast$.
Therefore, eigenvectors $V_s$ locally determines the direction of the stable manifold $W^s(x_\ast;\varphi^{des}_\mu)\cap N$.
Schematic illustration of the present argument is shown in Figure \ref{fig_positive}.

\begin{figure}[htbp]
  \begin{center}
    \includegraphics[width=5cm]{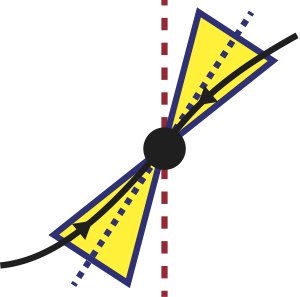}
    \caption{Location of stable manifold of $x_\ast$ for dynamics in $\mathbb{R}^2$}
    \label{fig_positive}
\end{center}
The black ball denotes $x_\ast$, which is supposed to be hyperbolic. 
The black curve denotes $W^s(x_\ast; \varphi^{des}_\mu)$.
The blue dotted line denotes an eigenvector associated with an eigenvalue with negative real part of the linearized matrix of vector field at $x_\ast$, which is tangent to $W^s(x_\ast; \varphi^{des}_\mu)$ at $x_\ast$.
The yellow set denotes a stable cone with the vertex $x_\ast$.
The red dotted line describes the null set $Null(T) \equiv \{T(x) = 0\}$ of $T$ realizing the time-scale desingularization (\ref{time-desing}). 
In both sides $T$ is supposed to have an identical sign.
In this situation, the function $T$ has an identical sign on each component of $W^s(x_\ast;\varphi^{des}_\mu)\setminus \{x_\ast\}$ if the stable cone except the vertex $x_\ast$ is disjoint from $Null(T)$, which is determined by the direction of eigenvector and the slope of stable cone.
\end{figure}

\par
As a case study, we consider (\ref{ODE}) as well as (\ref{ODE-desing}) in $\mathbb{R}^2$, where $x=(x_1,x_2)$.
Now we make the following assumption.
\begin{ass}
\label{ass-case-study}
The origin $x=(0,0)$ is a degeneracy-inducing point of (\ref{ODE}) such that it is a saddle for (\ref{ODE-desing}) and that
\begin{equation*}
T(x) = x_1^{-1}.
\end{equation*}
There is a trajectory $\{x(\tau)\}_{\tau \geq 0}$ for (\ref{ODE-desing}) included in $W^s(0;\varphi^{des}_\mu)$.
Our question here is {\em when $T(x)$ is positive along $\{x(\tau)\}_{\tau \geq 0}$}.
Further we assume that the following objects are validated, possibly with their explicit enclosures:
\begin{itemize}
\item Eigenvector $V_s$ at $x=0$ associated with the negative eigenvalue.
\item $M$-cone condition in $N$, where $N$ is an $h$-set being a neighborhood of $x=0$. 
\item $W^s(0;\varphi^{des}_\mu)$ as a vertical disk in $N$ for (\ref{ODE-desing}).
\item There is a positive $\tau = \bar \tau$ such that $x_1(\tau) > 0$ for all $0\leq \tau \leq \bar \tau$ and that $x(\bar \tau) \in W^s(0;\varphi^{des}_\mu)\cap N \cap \{x_1 > 0\}$.
\end{itemize}
Cone conditions with the above assumption indicates that $x(\tau)\in N$ holds for all $\tau \geq \bar \tau$.
\end{ass}

The stable eigenvector $V_s$ induces a rotation of the stable cone $C^s_M(0; I_2)$ whose the central axis is directed in $(0,\pm 1)^T$ in the $(a,b)$-coordinate\footnote{
In our validation, the stable component is always assumed to be in the second coordinate, as seen in (\ref{ODE-coord}).
}.
More precisely, $V_s = (v_s^1, v_s^2)^T$ is parallel to
\begin{equation}
\label{theta-bound}
\begin{pmatrix}
1 \\ \tan \theta
\end{pmatrix},\quad \text{ where }\theta = \arctan \left(\frac{v_s^2}{v_s^1}\right)\in [\underline{\theta}, \overline{\theta}] \subset [-\pi, \pi)
\end{equation}
obtained through direct calculations or rigorous numerics.
Any points in the stable cone $\{M|a| < |b|\}$ has an angle
\begin{equation}
\label{stable-cone-angle}
\begin{pmatrix}
\frac{1}{\tan \alpha} \\ 1
\end{pmatrix}, \text{ where }\alpha \in \pm \left[\arctan M, \pi - \arctan M\right]\text{ with }\arctan M \in \left(-\frac{\pi}{2}, \frac{\pi}{2}\right)
\end{equation}
so that $\sin \alpha \not = 0$ holds in the range,
equivalently
\begin{equation*}
\begin{pmatrix}
-\tan \beta \\ 1
\end{pmatrix}, \text{ where }\beta \in \pm \left[ \arctan M - \frac{\pi}{2}, \frac{\pi}{2} - \arctan M\right].
\end{equation*}
As a preconditioning operation, we further rotate the stable cone by $-\pi /2$ radian via
\begin{equation}
\label{precondition}
\begin{pmatrix}
\cos \left(-\frac{\pi}{2}\right) & -\sin \left(-\frac{\pi}{2}\right)\\
\sin \left(-\frac{\pi}{2}\right) & \cos \left(-\frac{\pi}{2}\right)
\end{pmatrix}
\begin{pmatrix}
-\tan \beta \\ 1
\end{pmatrix}
=
\begin{pmatrix}
0 & 1\\
-1 & 0
\end{pmatrix}
\begin{pmatrix}
-\tan \beta \\ 1
\end{pmatrix}
=
\begin{pmatrix}
1 \\ \tan \beta
\end{pmatrix}.
\end{equation}

Then the rotation matrix
\begin{equation*}
\begin{pmatrix}
\cos \theta & -\sin \theta\\
\sin \theta & \cos \theta
\end{pmatrix},\quad \theta \in [\underline{\theta}, \overline{\theta}] \text{ with } \left[-\frac{\pi}{2}, \frac{\pi}{2}\right) \cap [\underline{\theta}, \overline{\theta}] \not = \emptyset
\end{equation*}
such that the central direction of the stable cone equals to $V_s$ with $x_1 > 0$ gives the actual rotation range of 
points in the stable cone in the original $(x_1, x_2)$-coordinate which is actually given as
\begin{equation*}
\begin{pmatrix}
\cos \theta & -\sin \theta\\
\sin \theta & \cos \theta
\end{pmatrix}
\begin{pmatrix}
1 \\ \tan \beta
\end{pmatrix}
= 
\begin{pmatrix}
\cos \theta - \sin \theta \tan\beta \\ \sin \theta + \cos \theta \tan\beta
\end{pmatrix}.
\end{equation*}
Therefore, under the property of stable cone at the vertex $x = 0$, the remaining trajectory $\{x(\tau)\}$ satisfies $x_1>0$ if $\cos \theta - \sin \theta \tan\beta > 0$.
The schematic illustration of the present arguments is drawn in Figure \ref{fig-rotation}.
We finally obtain the following result.

\begin{figure}[htbp]\em
\begin{minipage}{0.32\hsize}
\centering
\includegraphics[width=4cm]{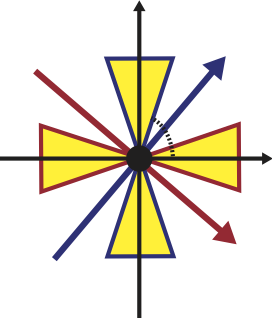}
(a)
\end{minipage}
\begin{minipage}{0.32\hsize}
\centering
\includegraphics[width=4cm]{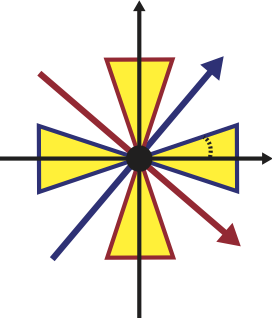}
(b)
\end{minipage}
\begin{minipage}{0.32\hsize}
\centering
\includegraphics[width=4cm]{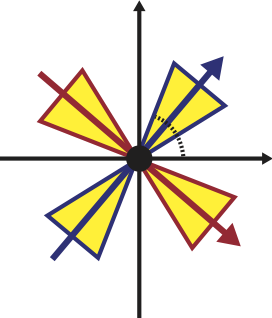}
(c)
\end{minipage}\\
\caption{Validation of $T(x) > 0$ in terms of cones and eigenvectors}
\label{fig-rotation}
In all figures herein, the red and blue arrows denote the unstable and stable eigenvectors $V_u$ and $V_s$ associated with the positive and negative eigenvalues at the origin, respectively, yellow regions denote the unstable (surrounded by red lines) and stable (surrounded by blue lines) cones, respectively, and the dotted curve denotes the angle of corresponding objects. 
(a) : Stable cone is located so that the central axis is equal to $b$-axis; the second coordinate in the diagonalized one.
(b) : Preconditioning rotation of the cone so that the central axis is equal to $a$-axis; the first coordinate in the diagonalized one.
(c) : The original cone location so that the central axis is equal to $V_s$.
\end{figure}

\begin{prop}[Validation of $T(x) > 0$ along $W^s(0;\varphi^{des}_\mu)$ : a case study]
Consider (\ref{ODE}) with Assumption \ref{ass-case-study}.
Assume that $\cos \theta > \sin \theta \tan\beta$ for all $\theta \in [\underline{\theta}, \overline{\theta}]$ given in (\ref{theta-bound}) and for all $\beta \in \pm \left[ \arctan M - \frac{\pi}{2}, \frac{\pi}{2} - \arctan M\right]$.
Then $x_1(\tau) > 0$ for all $\tau \geq \bar \tau$ and $x(\tau) \to 0\in \mathbb{R}^2$ as $\tau \to \infty$.
\end{prop}

\begin{proof}
Our validation indicates that the sign of the first component in $x$-coordinate is identical in each component of the stable cone $C^s_M(0;P)\setminus \{0\}$.
Since $W^s(0;\varphi^{des}_\mu)\cap N$ is contained in $C^s_M(0;P)\cap N$, then the sign of the first component is identical in each component of the stable cone $W^s(0;\varphi^{des}_\mu)\setminus \{0\}$.
By assumption a trajectory $\{x(\tau)\}$ is included in $W^s(0;\varphi^{des}_\mu)$ and $x_1(\bar \tau) > 0$.
Since $W^s(0;\varphi^{des}_\mu)$ is a vertical disk containing $0$, we know that the trajectory $\{x(\tau)\}$ converges to $0$ inside the component of $(W^s(0;\varphi^{des}_\mu)\setminus \{0\})\cap N$ containing $x(\bar \tau)$, which shows that $x_1(\tau) > 0$ for all $\tau \geq \bar \tau$ and the proof is completed.
\end{proof}

If $T(x)$ is a power of $x_1$, the same arguments yield the positivity of $T(x)$.
The condition which $T(x)$ is locally negative everywhere follows from the same arguments with a few modifications.

\begin{rem}
If we want to validate $T(x) > 0$ along $W^u(x_\ast;\varphi^{des}_\mu)\cap N$, replace the following objects during the validation:
\begin{itemize}
\item $V_s$ by $V_u$ : eigenvector at $x=0$ associated with the positive eigenvalue.
\item $C^s_M (0;P)$ by $C^u_M (0;P)$.
\item (\ref{stable-cone-angle}) by
\begin{equation}
\label{unstable-cone-angle}
\begin{pmatrix}
1 \\ \tan \alpha
\end{pmatrix}, \text{ where }\tan \alpha \in \pm \left[\frac{-1}{M}, \frac{1}{M} \right]\text{ with }\arctan \frac{1}{M} \in \left(-\frac{\pi}{2}, \frac{\pi}{2}\right).
\end{equation}
\item We do not need any preconditioning operation like (\ref{precondition}).
\end{itemize}
\end{rem}

\begin{rem}
Rigorous validation methodology for (simple) eigenpairs we have actually applied in this paper to examples in Section \ref{section-example} is based on \cite{Y1980, Mat3}.
\end{rem}

For higher dimensional systems, direction of eigenspaces at $x_\ast$ and slope of cones will determine a sufficient condition to prove $T(x) > 0$ along $W^s(x_\ast;\varphi^{des}_\mu)\cap N$, but we skip the detailed arguments because such a generalization is not used in our validation examples,  and it may depend on the form of $T(x)$ and the direction of trajectories.

\subsubsection{Details 2 : Calculation of enclosure of arrival time. Case study in $2$-dimensional dynamical systems}
The rest is the validation of $t_{\max}$, which depends of the concrete form of $T(x)$ in (\ref{time-desing}).
Here we also treat only $2$-dimensional problems. 
As an example, let $T(x) = x_1^{-m}$ with $m>0$, where $x_1$ is the first component of $x$, and $x_\ast = 0\in \mathbb{R}^2$.
Further we assume that the change of coordinate $P_\mu \equiv c_N^{-1}$ is {\em a nonsingular matrix}\footnote{
In practical validations, $P_\mu$ can be chosen as an affine transformation as long as we apply validation methodology of isolating blocks mentioned in Section \ref{section-block-brief} whose brief summaries are shown in Appendix \ref{section-block}.
}.
\par
If Step 1 in Algorithm \ref{alg-finite} is succeeded, then the stable manifold $W^s(x_\ast; \varphi^{des}_\mu)$ is validated in an isolating block in terms of vertical disk $(\sigma^s_\mu(b), b)$.
Without the loss of generality, we may assume that the initial point $x_0$ is on $W^s(x_\ast; \varphi^{des}_\mu)$ under the time translation.
Then
\begin{align*}
t_{\max} &= \int_0^\infty x_1(\tau;x_0)^{m} d\tau \leq \int_0^\infty |x_1(\tau;x_0)|^{m} d\tau \\
	&\leq (2{\bf p}_1)^{m} \int_0^\infty L_s(x(\tau;x_0))^{m/2} d\tau\\
	&\leq (2{\bf p}_1)^{m}\int_{L_{s,0}}^0 L_s(x(\tau;x_0))^{m/2} \frac{d\tau}{dL_s}dL_s\\
	&\leq -(2{\bf p}_1)^{m}\frac{M^2+1}{2M^2 \overrightarrow{\mu_{s,W^s}}}\int_0^{L_{s,0}} L_s^{m/2-1} dL_s\\
	&= -(2{\bf p}_1)^{m}\frac{M^2+1}{m M^2 \overrightarrow{\mu_{s,W^s}}} L_{s,0}^{m/2} \equiv T_{m,M,x_\ast},
\end{align*}
where $L_{s,0} \geq L_s(x_0)  > 0$ and $\overrightarrow{\mu_{s,W^s}}$ is given in (\ref{rate-stable-sm}).
We have also used the fact
\begin{equation*}
|x_1| = |p_{11}a + p_{12}b| \leq {\bf p}_1 (|a| + |b|) \leq 2{\bf p}_1 (|a|^2 + |b|^2),
\end{equation*}
where $P = (p_{ij})_{i,j=1,2}$ is the matrix and ${\bf p}_1 = \max\{|p_{11}|, |p_{12}|\}$.

Note that $\overrightarrow{\mu_{s,W^s}} < 0$ and hence $T_{m,M,x_\ast}$ is positive. 
The value $T_{m,M,x_\ast}$ gives an upper bound of the arrival time $t_{\max}$ and is computable.
Validation of departure time of $x_\ast$, namely, the time $t > 0$ such that $x(-t) = x_\ast$ and $x(0) = x_0$ can be validated in the similar manner by using the information of $W^u(x_\ast; \varphi^{des}_\mu)$ as horizontal disks and the estimate stated in Proposition \ref{prop-Lyapunov-um}.

\begin{rem}
Obviously, estimates of arrival times can be generalized to the case replacing an equilibrium $x_\ast$ by an invariant set $S$ for desingularized vector field (\ref{ODE-desing}), which induces new type of finite-time singularities involving $S$, such as periodic blow-ups discussed in \cite{Mat}.
\end{rem}

Once we validate cone conditions on stable (resp. unstable) manifolds, we can calculate an upper bound of $L_s$ (resp. an lower bound of $L_u$) in $N$.
Indeed, assume that the stable manifold $W^s(x_\ast; \varphi^{des}_\mu) = \{c_N^{-1}(\sigma^s_\mu(b), b) \mid b\in \overline{B}_s(0,R_s)\}$ of a saddle $x_\ast$ is validated in an $h$-set $N$ with $c_N(N) = \overline{B}_u(0,R_u)\times \overline{B}_s(0,R_s)$.
The unstable $m_u$-cone condition indicates that $W^s(x_\ast; \varphi^{des}_\mu)$ is contained in the cone $C^s_M(x_\ast; P_\mu)\cap N$.
This fact shows that the function $\sigma^s_\mu$ is Lipschitz continuous with Lipschitz constant less than or equal to $M^{-1}$.
Moreover, the origin in $(a,b)$-coordinate is the zero of $L_s(a,b)$, since it corresponds to $x_\ast$ via the change of coordinate.
It thus holds that
\begin{equation*}
L_s(a,b) \leq (1+M^{-2})\|b\|^2 \leq R_s^2(1+M^{-2})\equiv L_{s,0}.
\end{equation*}
The rightmost value is actually computable and, as in the methodology for blow-up time validations \cite{MT2017, TMSTMO}, we can compute the upper bound of $t_{\max}$ as $T_{m,M,x_\ast}$.

\subsection{Remark : Similarity to blow-up solution validations}
We have shown a numerical validation methodology of solutions with finite-time singularities.
Notice that the basic idea is essentially the same as that of blow-up solutions (e.g., \cite{MT2017, TMSTMO}).
Indeed, the qualitative difference is the treatment of points inducing finite-time singularity.
In the case of blow-up solutions, they are located {\em at infinity}\footnote
{
It is justified via compactifications (e.g., \cite{Mat}) of phase spaces.
}.
On the other hand, the latter is typically located in a bounded region of phase space.
According to preceding studies (e.g., \cite{DV1991, MS2001}), the similarity is not so surprising.
But there is an interesting point that such difference and similarity cannot be seen from the form of equations, but can be seen through rigorous computation methodologies.

\section{Validation examples of finite-time singularities}
\label{section-example}
In this section, we demonstrate the applicability of our methodology.
Here we consider finite traveling waves involving finite-time extinction for degenerate differential equations, compacton traveling waves with countable family of composite waves, and trajectories through folded singularities in fast-slow systems.
Both examples show explicit estimates of {\em finite} features of solutions such as support of compactons, or arrival and passing times through singularities, which are new insights of natures in nonlinear systems via rigorous numerics.
All computations were carried out on 
macOS X El Capitan (ver.10.11.6),
CPU 3.1 GHz Intel(R) Core i7 processor and 16GB 1867 MHz DDR3 memory using CAPD library \cite{CAPD} ver. 4.0 to rigorously compute the trajectories of ODEs.
Validation codes are available at \cite{Mat_code}.

\subsection{Finite traveling waves in degenerate diffusion equation}
\label{section-diff}
The first example is the finite-time extinction of the following degenerate diffusion equation:
\begin{equation}
\label{demo1}
u_t = \frac{1}{m+1}(u^{m+1})_{xx} + u^p(1-u)(u-a),\quad m>0,\ p>0\text{ with }m+p=1,
\end{equation}
where $a\in (0,1)$ is a parameter.
Setting $\xi =x-ct$, we have the equation for traveling waves $u(t,x)=\phi(x-ct)$ as follows:
\begin{equation*}
-c\phi' = (\phi^m \phi')' + \phi^p(1-\phi)(\phi-a),\quad {}' = \frac{d}{d\xi}.
\end{equation*}
Introducing the new moving-frame coordinate
\begin{equation*}
\frac{dz}{d\xi} = \phi(\xi)^{-m},
\end{equation*}
we have
\begin{equation*}
-c\phi^{-m}\dot \phi = \phi^{-m} \ddot \phi + \phi^p(1-\phi)(\phi-a),\quad \dot {} = \frac{d}{dz}.
\end{equation*}
equivalently,
\begin{equation}
\label{demo1-main}
-c\dot \phi = \ddot \phi + \phi (1-\phi)(\phi-a)\quad \Leftrightarrow \quad 
\begin{cases}
\dot \phi = \psi, &\\
\dot \psi = -c\psi - \phi(1-\phi)(\phi-a).
\end{cases}
\end{equation}
A preceding work \cite{LPB2004} reports that the system (\ref{demo1}) admits finite time extinction for any solutions with compactly supported initial data $u_0(x)$.
In fact, comparison theorem for (\ref{demo1}) indicates that finite traveling waves $\phi(z)$ with $\lim_{z\to -\infty}\phi(z) = 0$ and $\lim_{z\to +\infty}\phi(z) = 1$, and a positive speed $c>0$, up to translation of moving frame, dominate the compactly supported solution $u(t,x)$ with $u(0,x) = u_0(x)$. 
The key of finite-time extinction is thus the existence of finite traveling waves.

\subsubsection{Finite traveling wave front and finite time extinction}
\label{section-ex-extinction}
Let, for example, $m=3/4$ and $p=1/4$.
Obviously the origin $(\phi, \psi) = (0,0) \equiv p_0$ is an equilibrium for (\ref{demo1-main}).
Moreover, it is hyperbolic with $a > 0$. 
Indeed, the Jacobian matrix $J$ for the linearized matrix at the origin is
\begin{equation*}
J(p_0) = \begin{pmatrix}
0 & 1 \\
a & -c
\end{pmatrix}, \quad {\rm Spec}(J(p_0)) = \left\{\lambda_\pm = \frac{-c\pm \sqrt{c^2+a}}{2}\right\}.
\end{equation*}

\par
The nontrivial maximal existence time in the original $\xi$-scale is intrinsically concerned with estimates of passage time in $\xi$-scale near the origin $p_0$.
Assume that $p_0$ is validated with an isolating block $N$ satisfying $M$-cone conditions.
If we further assume that the solution trajectory $(\phi(\xi), \psi(\xi))$ leaves the exit $N^-$ with the initial time $-z_N$ corresponding to $\xi=0$, for example, the minimum time $\xi_{\min}$ is estimated as follows:

\begin{align*}
\xi_{\min} &\equiv \int_{-\infty}^{-z_N}\phi(\eta)^m d\eta = \int_{z_N}^\infty \phi(\tilde z)^m d\tilde z\quad (\tilde z = -z)\\
& \leq \int_{z_N}^\infty \phi(\tilde z)^m d\tilde z \leq \int_{z_N}^\infty |\phi(\tilde z)|^m d\tilde z \\
	&\leq (2{\bf p}_1)^m \int_{z_N}^\infty \tilde L_u(\phi(\tilde z))^{m/2} d\tilde z\quad (\tilde L_u = -L_u)\\
	&\leq (2{\bf p}_1)^m \int_{L_{u,0}}^0 \tilde L_u(\phi(\tilde z))^{m/2} \frac{d\tilde z}{d\tilde L_u}d\tilde L_u\\
	&\leq (2{\bf p}_1)^m \frac{M^2+1}{2M^2 \overrightarrow{\xi_{u,W^u}}}\int_0^{L_{u,0}} \tilde L_u^{m/2-1} d\tilde L_u\\
	&= (2{\bf p}_1)^m \frac{M^2+1}{m M^2 \overrightarrow{\xi_{u,W^u}}} L_{u,0}^{m/2} \equiv T_{m,M,x_\ast},
\end{align*}
where we have used (\ref{Lyapunov-unstable-2}) for validating $\overrightarrow{\xi_{u,W^u}}$.
Following the validation methodology in Algorithm \ref{alg-finite}, we obtain the following result.
The other computer-assisted results are obtained in the similar manner.

\begin{car}
\label{car_FTW_diff}
Consider the degenerate diffusion equation (\ref{demo1}) with $m=3/4$, $p=1/4$ and $a=0.3$.
Then there is a value $c_\ast \in 0.282842_{72}^{99}$ such that (\ref{demo1}) admits a finite traveling wave solution
\begin{equation*}
u(t,x) = \begin{cases}
\phi(x-c_\ast t) & \text{if $x-c_\ast t \geq 0$}\\
0 & \text{otherwise}
\end{cases}
\end{equation*}
up to translations,
where $\phi(x-c_\ast t)$ is the global solution of the desingularized system (\ref{demo1-main}) satisfying
\begin{equation*}
\lim_{z \to -\infty} \phi(z) = 0,\quad \lim_{z \to +\infty} \phi(z) = 1.
\end{equation*}
Isolating blocks $N_1$ and $N_2$ containing $(\phi_1, \psi_1) = (0,0)$ and $(\phi_2, \psi_2) = (1,0)$, respectively, validating the above global solution are given by $N_i = P_i B_i + \{(\phi_i, \psi_i)\}$, where $P_i = [V_1^i\ V_2^i]$ is the eigenmatrix whose columns are eigenvectors given in 
Table 1, and
\begin{align*}
B_1 &= ([-1.0000000056134599e,+1.0001161922375184]\times 10^{-4})\\
	&\quad \quad \times ([-2.0953502949606345,+2.0000046041789743]\times 10^{-7}),\\
B_2 &= ([-1.0000068275040609,+1.0000038705938853]\times 10^{-6})\\
	&\quad \quad \times ([-1.0000261784559085,+1.0000461805915562e-07]\times 10^{-7}).
\end{align*}
Moreover, we have the following domain estimate in the original frame scale $\xi$:
\begin{equation*}
|\phi^{-1}(0,1/2)| \in 1._{44122673040}^{50547975352}.
\end{equation*}
See Figure \ref{fig-extinction}.
\end{car}

\begin{figure}[htbp]\em
\begin{minipage}{0.32\hsize}
\centering
\includegraphics[width=5cm]{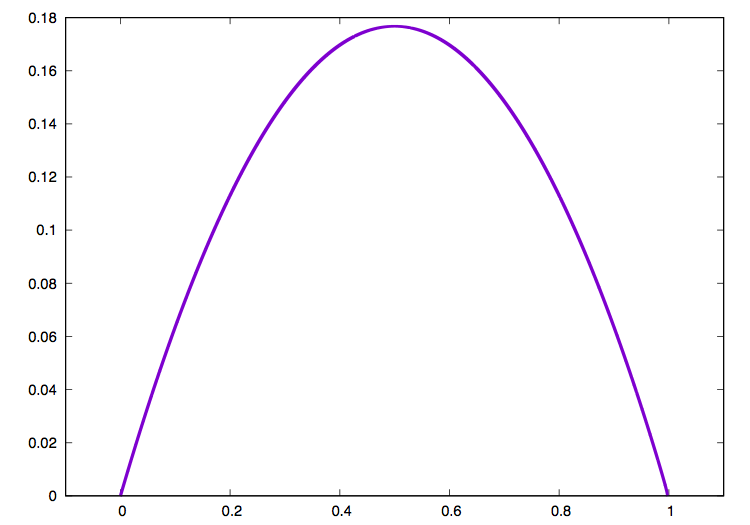}
(a)
\end{minipage}
\begin{minipage}{0.32\hsize}
\centering
\includegraphics[width=5cm]{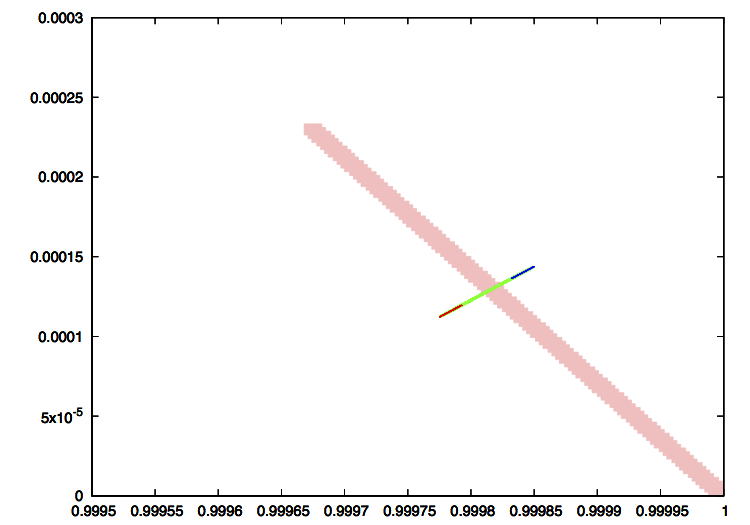}
(b)
\end{minipage}
\begin{minipage}{0.32\hsize}
\centering
\includegraphics[width=5cm]{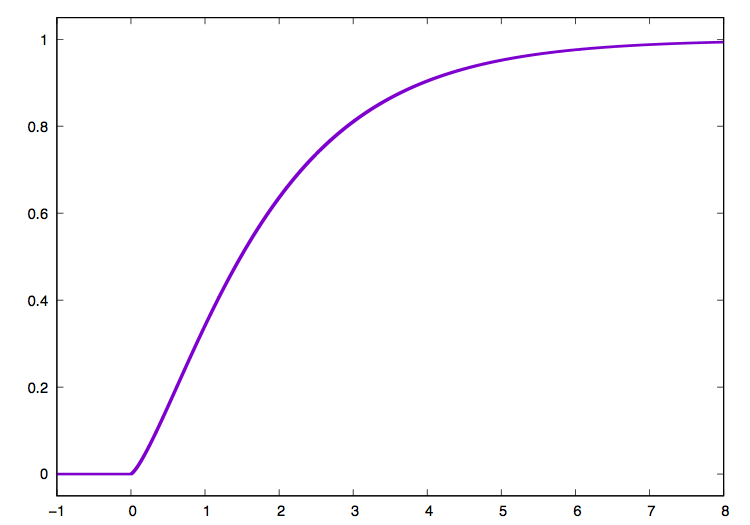}
(c)
\end{minipage}\\
\caption{Finite traveling wave profile for Computer Assisted Result \ref{car_FTW_diff}}
\label{fig-extinction}
(a) : A numerical heteroclinic trajectory connecting $(\phi, \psi) = (0,0)$ and $(\phi, \psi) = (1,0)$ for (\ref{demo1-main}) in $(\phi, \psi)$-phase plane.
(b) : Verification of covering relation $N\overset{\varphi(\bar \tau,\cdot)}{\Longrightarrow}M$ with $\bar \tau=21.528$, where $N$ is the exit of an isolating block around $(0,0)$ times $(c\in )0.282842_{72}^{99}$ and $M$ is an isolating block containing $(1,0)$.
The green region denotes $\varphi(\bar \tau,N)$, the blue region $\varphi(\bar \tau,N)\cap \{c=0.28284272\}$ and the red region $\varphi(\bar \tau,N)\cap \{c=0.28284299\}$.
(c) : Corresponding profile of finite traveling wave for (\ref{demo1}) in $(x,\phi)$-plane.
This wave evolves in positive $x$-direction keeping the profile.
\end{figure}

\begin{table}[ht]
\label{table_FTW_diff}
\centering
\begin{tabular}{ccc}
\hline 
 & At $(\phi_1, \psi_1) = (0,0)$ & At $(\phi_2, \psi_2) = (1,0)$ \\
\hline\\[-2mm]
$M_u$ or $M_s$ for cone condition & $M_u = 5000$ & $M_s = 500$ \\[1mm]
$\ell_u$ or $\ell_s$ for extended cones & $\ell_u = 0.008$ & $\ell_s = 0.00007$ \\[1mm]
Eigenvalue $\lambda_1$ & $\in 0.7071_{6588504430233}^{8897990200269}$ & $\in 0.424_{29327568669606}^{31960764617949}$ \\[1mm]
Eigenvalue $\lambda_2$ & $\in -0.9900_{5997633633080}^{3716718098678}$ & $\in -0.7071_{9000248876762}^{6515941513469}$ \\[1mm]
Eigenvector $V_1$ & $\in \begin{pmatrix}0.8165_{6790481467278}^{8860207657336}\\ 0.577_{39795597826360}^{41798779753506}\end{pmatrix}$ & $\in \begin{pmatrix}0.9206_{5568909148965}^{7769701299829} \\ 0.390_{59404674239356}^{61744908763335}
\end{pmatrix}$\\[2mm]
Eigenvector $V_2$ & $\in \begin{pmatrix}-0.7107_{5049673357216}^{2968702694894} \\ 0.703_{58645755593130}^{60724878877856} \end{pmatrix}$ & $\in \begin{pmatrix} -0.8165_{9045287636589}^{6593971584757} \\ 0.577_{39533855080161}^{42076686823884}  \end{pmatrix}$\\[2mm]
$\tan \theta$ in (\ref{theta-bound}) with $V_1$ &  $\in 0.707_{08549508277319}^{12794905720078}$ & $-$ \\[2mm]
\hline 
\end{tabular}%
\caption{Detailed data for Computer Assisted Result \ref{car_FTW_diff}}
\end{table}

According to \cite{LPB2004}, the original differential equation (\ref{demo1}) admits the comparison principle.
In particular, if $u_0 = u_0(x)$ is the initial data such that $u_0(x) \leq \phi(x)$ for all $x\in \mathbb{R}$, where $\phi$ is the validated traveling wave solution obtained in Computer Assisted Result \ref{car_FTW_diff} with appropriate frame translation, then we have $u(t, x) \leq \phi(x-c_\ast t)$ for all $t\geq 0$ and $x\in \mathbb{R}$.
Moreover, the evolved solution $u(t,x)$ tends to zero everywhere in $x$ for some $t = T < \infty$ (Corollary 3.1 in \cite{LPB2004}).
In other words, a {\em finite time extinction} occurs.

\subsubsection{Compacton traveling wave}
\label{section-ex-compacton}

Next we focus on compactly supported traveling waves called {\em compactons}.
In the present problem, it would correspond to a global trajectory of (\ref{demo1-main}) homoclinic to the origin, which is actually validated with an appropriate choice of parameters.
\begin{car}
\label{car_compacton_diff}
Consider (\ref{demo1}) with $m=3/4$, $p=1/4$ and $a=0.3$.
Then there exists a value $c_\ast\in 1.5\times 10^{-6}\times [-1, 1]$ such that the function $u(t,x)=\phi(x-c_\ast t)$ given by
\begin{equation*}
\phi(x-c_\ast t) \equiv \phi(\xi) = \begin{cases}
\varphi(\xi) & \xi \in (\xi_{\min}, \xi_{\max}),\\
0 & \xi \in \mathbb{R}\setminus (\xi_{\min}, \xi_{\max})
\end{cases}
\end{equation*}
is a compacton traveling wave solution with speed $c_\ast$\footnote{
This result does not tell us the detailed information of $c_\ast$ whether or not $c_\ast = 0$. 
If it is the case, the wave should be said to be a compacton {\em standing wave}.
}, where $\{\phi(\tau)\}$ is a homoclinic orbit of $(\phi, \psi) = (0,0)$ contained in the isolating block $N = PB$, where $P = [V_1\ V_2]$ is the eigenmatrix given in 
Table 2, and
\begin{align*}
B &= ([-1.0000068275040609,+1.0000038705938853]\times 10^{-6})\\
	&\quad \quad \times ([-1.0000261784559085,+1.0000461805915562e-07]\times 10^{-7}).
\end{align*}

Finally, $\xi_{\min}$ and $\xi_{\max}$ can be determined as finite values satisfying
\begin{equation*}
\xi_{\max} - \xi_{\min} \equiv |{\rm supp}u| \in 5._{7503277707}^{8115725019}
\end{equation*}
See Figure \ref{fig-compacton}.
\end{car}

\begin{figure}[htbp]\em
\begin{minipage}{0.32\hsize}
\centering
\includegraphics[width=5cm]{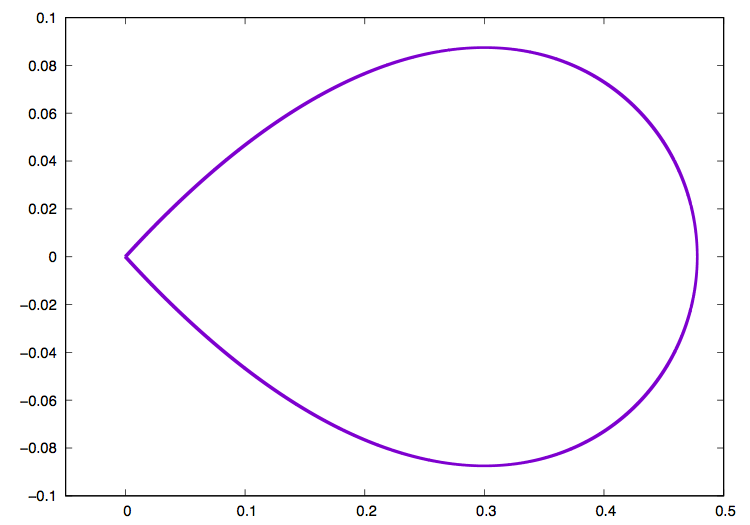}
(a)
\end{minipage}
\begin{minipage}{0.32\hsize}
\centering
\includegraphics[width=5cm]{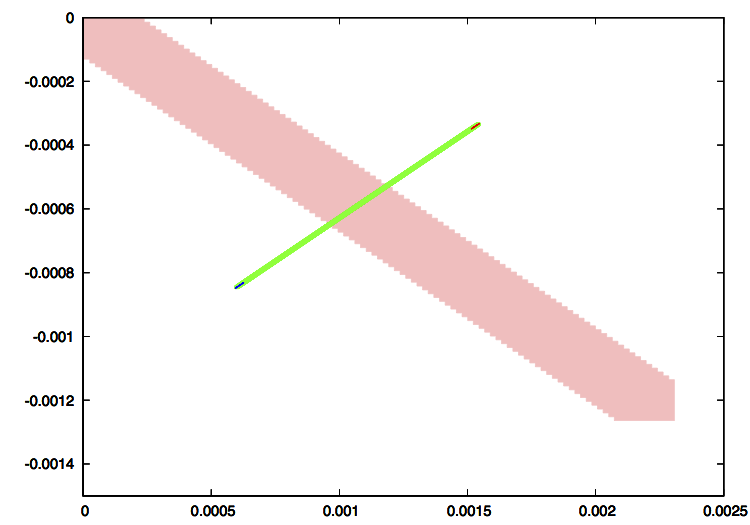}
(b)
\end{minipage}
\begin{minipage}{0.32\hsize}
\centering
\includegraphics[width=5cm]{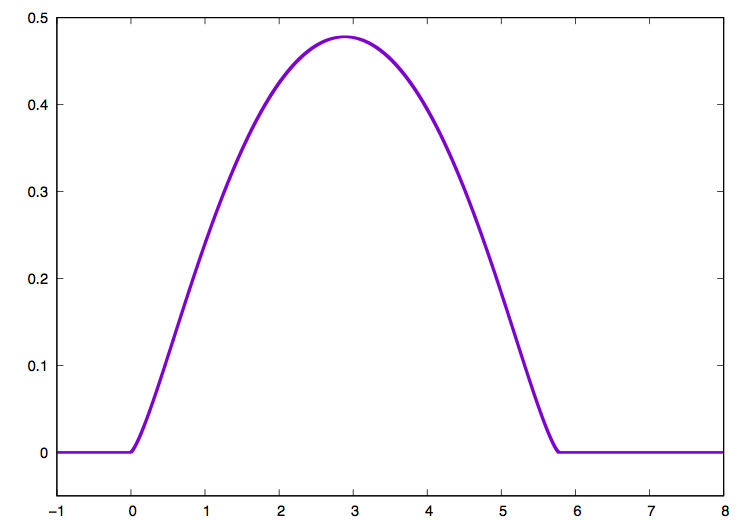}
(c)
\end{minipage}\\
\caption{Finite traveling wave profile for Computer Assisted Result \ref{car_compacton_diff}}
\label{fig-compacton}
(a) : A numerical homoclinic trajectory of $(\phi, \psi) = (0,0)$ for (\ref{demo1-main}) in $(\phi, \psi)$-phase plane.
(b) : Verification of covering relation $N\overset{\varphi(\bar \tau,\cdot)}{\Longrightarrow}M$ with $\bar \tau=29.403$, where $N$ is the exit of an isolating block around the origin times $(c\in )1.5\times 10^{-6}\times [-1, 1]$ and $M$ is an isolating block containing the origin (pink).
The green region denotes $\varphi(\bar \tau,N)$, the blue region $\varphi(\bar \tau,N)\cap \{c=-1.5\times 10^{-6}\}$ and the red region $\varphi(\bar \tau,N)\cap \{c=1.5\times 10^{-6}\}$.
(c) : Corresponding profile of compacton traveling wave for (\ref{demo1}) in $(x,\phi)$-plane.
\end{figure}

\begin{table}[ht]
\centering
\begin{tabular}{cc}
\hline 
 & At $(0,0),\ c\in 5.0\times 10^{-6}\times [-1, 1]$ \\
\hline\\[-2mm]
$(M_u, M_s)$ for cone condition & $(7000,100)$ \\[1mm]
$(\ell_u, \ell_s)$ for extended cones & $(0.0015, 0.001)$ \\[1mm]
Eigenvalue $\lambda_1$ & $\in 0.5477_{7157930167553}^{8308020874578}$  \\[1mm]
Eigenvalue $\lambda_2$ & $\in -0.5477_{8308024148070}^{7157930405451}$ \\[1mm]
Eigenvector $V_1$ & $\in \begin{pmatrix} 0.87714_{184587409483}^{960434933953} \\ 0.4804_{2765940975440}^{3734031068247} \end{pmatrix}$ \\[2mm]
Eigenvector $V_2$ & $\in \begin{pmatrix} -0.87714_{960433700939}^{184587344934} \\ 0.4804_{2765939685023}^{3734034727664} \end{pmatrix}$ \\[2mm]
$\tan \theta$ in (\ref{theta-bound}) with $V_1$ &  $\in 0.5477_{1461678550337}^{3049828892173}$ \\[2mm]
$\tan \theta$ in (\ref{theta-bound}) with $V_2$ &  $\in -0.5477_{3049833104448}^{1461677849109}$ \\[2mm]
\hline 
\end{tabular}%
\label{table_compacton_diff}
\caption{Detailed data for Computer Assisted Result \ref{car_compacton_diff}}
\end{table}

\begin{rem}
In \cite{SZ2013}, the property which numerically validated traveling wave has compact support is proved, which is due to hyperbolicity of the origin in the desingularized system.
There the concrete support has not been estimated explicitly.
On the other hand, our present validation also measures the concrete size of supports.
\end{rem}

\subsubsection{Composite wave 1: dead core expansion}
\label{section-ex-deadcore}

Finite traveling waves whose supports are proper subsets of $\mathbb{R}$, under symmetry in governing equation, easily generate another weak solutions by superposing finite traveling wave profiles, which are called {\em composite waves}.
Generation of composite waves is one of fundamental and non-trivial problems for nonlinear wave systems, such as solitary waves for KdV-type systems, or shocks and rarefactions for Riemann problems of conservation laws (e.g., \cite{SMP1996}).
Unlike these special structured waves, finite traveling waves easily generate composite waves at least in the present system.
With the help of concrete wave validations by rigorous numerics, we can discuss concrete possible types of composite waves, which are available not only to (\ref{demo1}) but to various systems.
\par
In the present case, let $\phi = \phi(\xi)$ be a traveling wave validated in, say, Computer Assisted Result \ref{car_FTW_diff}. 
Note that the wave $\phi$ has the speed $c > 0$.
We then easily know that $\tilde \phi(\xi) := \phi(-\xi)$ is also a finite traveling wave solution of (\ref{demo1}) with the same $a$ and the speed $-c$.
Obviously, for any $w\in \mathbb{R}$, the function $\phi(\xi+w)$ is also a finite traveling wave solution, which is nothing but the translation symmetry of traveling waves.
Let $\phi_w$ be the wave profile $\phi$ defined by
\begin{equation*}
\phi_w(\xi) = \begin{cases}
\phi(\xi) & \xi > w\\
0 & \xi \leq w
\end{cases},
\end{equation*}
in which case we have
\begin{equation*}
\tilde \phi_w(\xi)\equiv \phi_w(-\xi) = \begin{cases}
\phi(-\xi) & \xi < -w\\
0 & \xi \geq -w
\end{cases}.
\end{equation*}
Note that the endpoint $w$ can be freely arranged, which is due to translation symmetry, so that ${\rm supp}\phi_w \cap {\rm supp}\tilde \phi_w = \emptyset$ for all $(t,x)\in (0,\infty)\times \mathbb{R}$.
We thus have the following result.
\begin{cor}
\label{cor-deadcore}
Let $\phi(\xi) = \phi(x-ct)$ be a finite traveling wave solution of (\ref{demo1}) with speed $c > 0$, say that validated in Computer Assisted Result \ref{car_FTW_diff}.
Then, for any $w_1, w_2\in \mathbb{R}$ with $w_1 + w_2 > 0$, the function
\begin{equation*}
\Phi_{w_1, w_2}(t,x) := 
\phi_{w_1}(\xi) + \tilde \phi_{w_2}(\xi)
\end{equation*}
is a weak solution of (\ref{demo1}).
\end{cor}
\begin{proof}
Notice that $\phi_{w_1}$ and $\tilde \phi_{w_2}$ move in the opposite directions to each other so that ${\rm supp}\phi_{w_1}(t, \cdot) \cap {\rm supp}\tilde \phi_{w_2}(t, \cdot) = \emptyset$ holds for all $t\geq 0$.
Therefore all assumptions in Proposition \ref{prop-superpose} are satisfied for $\phi_{w_1}$ and $\tilde \phi_{w_2}$. 
In particular, the sum $\Phi_{w_1, w_2}(t,x)$ is also a weak solution of (\ref{demo1}).
\end{proof}

The solution $\Phi_{w_1, w_2}$ is referred to as a {\em composite wave} since two wave profiles are superposed.
By construction the support $\Phi_{w_1,w_2}$ is monotonously contracted. 
More precisely,
${\rm supp}\Phi_{w_1,w_2}(t_2, \cdot)\subset {\rm supp}\Phi_{w_1,w_2}(t_1, \cdot)$
holds for all $t_2 > t_1 \geq 0$.
Equivalently, the zero region $N_0(\Phi_{w_1,w_2}(t,\cdot)) =\{x\in \mathbb{R}\mid \Phi_{w_1,w_2}(t,x) = 0\}$ expands monotonously; namely, $N_0(\Phi_{w_1,w_2}(t_1, \cdot)) \subset N_0(\Phi_{w_1,w_2}(t_2, \cdot))$ holds for all $t_2 > t_1 \geq 0$.
This behavior represents an expansion of {\em dead core} (e.g., \cite{LPB2004}).
See Figure \ref{fig-further}-(a).

\begin{figure}[htbp]\em
\begin{minipage}{0.32\hsize}
\centering
\includegraphics[width=5cm]{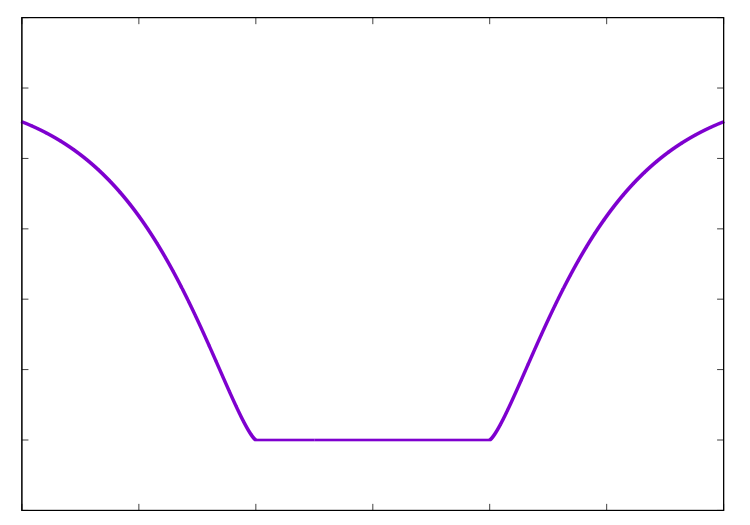}
(a)
\end{minipage}
\begin{minipage}{0.32\hsize}
\centering
\includegraphics[width=5cm]{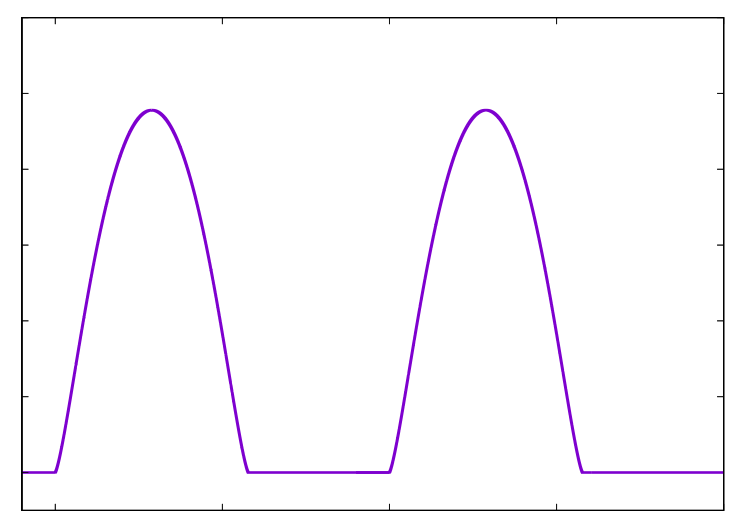}
(b)
\end{minipage}
\begin{minipage}{0.32\hsize}
\centering
\includegraphics[width=5cm]{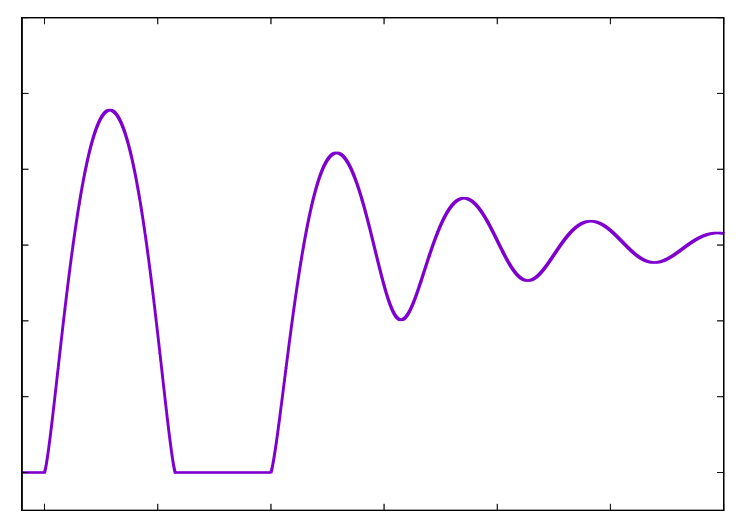}
(c)
\end{minipage}
\caption{Various weak solutions for (\ref{demo1}) in $(x,\phi)$-plot}
\label{fig-further}
(a) : Expanding dead core (Corollary \ref{cor-deadcore}). 
Each wave component evolves so that the null region of $u$; namely dead core, is monotonously expanded.
(b) : Composite finite traveling wave $\phi_{{\bf w}(+,+)}$ consisting of two positive compactons (Corollary \ref{cor_composite1_diff}). 
Wave speed of all wave components is identical by our construction and hence the supports of any two wave components never intersect for all time.
(c) : Composite wave consisting of a compacton and an oscillating finite traveling wave (Corollary \ref{cor_composite2_diff}).
It is not actually a traveling wave solution in the present sense because each wave component evolves with different speeds.
\end{figure}

\subsubsection{Composite wave 2 : infinitely many composite waves generated by a compacton}
\label{section-ex-osc}

Once we validate a compacton traveling wave $\phi^c_{[w_{\min}, w_{\max}]}$ with the speed $c\in \mathbb{R}$ and the support $[w_{\min}, w_{\max}]$ for (\ref{demo1}), we can easily construct countably many types of weak solutions parameterized by the number of nontrivial profiles and their supports.
For example, for a given $N\in \mathbb{N}$, let $\Phi_{{\bf w}_N} = \Phi_{{\bf w}_N}(t,x)$ be the function defined as
\begin{equation}
\label{compacton-family}
\Phi_{{\bf w}_N}(t,x) := \sum_{i=1}^N \phi^c_{[w_{{\min},i}, w_{{\max},i}]}(x-ct),
\end{equation}
where ${\bf w}_N = \{[w_{{\min},i}, w_{{\max},i}]\}_{i=1}^N$ denotes the family of supports of each {\em wave components}; namely the collection of $\phi_{[w_{{\min},i}, w_{{\max},i}]}$'s, at $t=0$ such that
\begin{equation*}
w_{{\min},i} < w_{{\max},i}\quad \text{ and  }\quad w_{{\max},j} < w_{{\min},j+1}
\end{equation*}
hold for $i=1,\cdots, N$ and $j=1,\cdots, N-1$.
Then, since propagation speeds of wave components are identical, their supports {\em never} intersect for all $t\in \mathbb{R}$.
In particular, any two pairs of wave components satisfy all assumptions of Proposition \ref{prop-superpose}.
It immediately follows that $\Phi_{{\bf w}_N}(t,x)$ is a weak solution of (\ref{demo1}).
Since the natural number $N$ can be chosen arbitrarily, we obtain (at least) countably many families of weak solutions $\{\Phi_{{\bf w}_N}\}_{N=1}^\infty$.
Combining this consequence with, say Computer Assisted Result \ref{car_compacton_diff}, we obtain the following result.

\begin{cor}[Existence of countable families of composite waves]
\label{cor_composite1_diff}
Consider (\ref{demo1}) with $m=3/4$, $p=1/4$ and $a=0.3$.
Let $\phi^{c_\ast}_{([w_{{\min}}, w_{{\max}}])}$ be a compacton traveling wave validated in Computer Assisted Result \ref{car_compacton_diff}\footnote{
Note that we do {\em not} discuss the uniqueness of such compacton waves in the present arguments.
Therefore we just fix one of such $c_\ast$. 
If the unique existence of $c_\ast$ is also validated, then $c_\ast$ should be fixed as the value.
}.
Then, for any $N\in \mathbb{N}$, the function $\Phi_{{\bf w}_N}$ of the form (\ref{compacton-family}) such that the width of support $[w_{{\min},i}, w_{{\max},i}]$ for each wave component is bounded by
\begin{equation*}
w_{{\max},i} - w_{{\min},i}\in 5._{7503277707}^{8115725019}
\end{equation*}
is a weak solution of (\ref{demo1}).
In particular, $u(t,x) = \Phi_{{\bf w}_N}(x-ct)$ is a finite traveling wave solution of (\ref{demo1}) for any $N\in \mathbb{N}$ and ${\bf w}_N$.
See Figure \ref{fig-further}-(b).
\end{cor}

\par
Furthermore, the other finite traveling wave profiles with different propagation speeds can be attached to $\Phi_{{\bf w}_N}$ {\em as long as supports of every pairs of wave components are disjoint for all $t\geq 0$}.
Numerical calculations indicate that there is a finite traveling wave solution connecting $u\equiv 0$ and  $u\equiv a=0.3$ for any $c>0$.
For practical values, these trajectories can be validated with rigorous numerics.
Indeed, we have the following validation result as an example.

\begin{car}[Finite traveling wave with oscillation]
\label{car_osc_diff}
Consider (\ref{demo1}) with $m=3/4$, $p=1/4$ and $a=0.3$.
Then for all $c\in [0.0999, 0.1001]$, the function $u(t,x)=\phi(x-ct)$ given by
\begin{equation*}
\phi(x-ct) \equiv \phi(\xi) = \begin{cases}
\varphi(\xi) & \xi > \xi_{\min},\\
0 & \xi \leq \xi_{\min}
\end{cases}
\end{equation*}
is a finite traveling wave solution with speed $c$ such that $\lim_{\xi\to +\infty}\phi(\xi) = a$ holds with oscillation.
The concrete profile is seen in Figure \ref{fig-further}-(c) (the right wave component).
Detailed data for validation is listed in 
Table 3.
\end{car}

\begin{table}[ht]
\centering
\begin{tabular}{ccc}
\hline 
 & At $(a,0)$, $c\in [0.0999, 0.1001]$ \\
\hline\\[-2mm]
$(M_u, M_s)$ for cone condition & $(({\rm empty}), 60)$ \\[1mm]
$\ell_s$ for extended cones & $\ell_s = 0.04$ \\[1mm]
${\rm Re}\lambda$ & $\in -0.050_{278878499779037}^{101269209305632}$ \\[1mm]
${\rm Im}\lambda\ (>0)$ & $\in 0.455_{21318969367819}^{33131247286973}$ \\[1mm]
${\rm Re}V$ & $\in \begin{pmatrix} 0.909_{04294196743196}^{13350922139213} \\ -0.045_{723512838268622}^{531445098151199} \end{pmatrix}$ \\[2mm]
${\rm Im}V $ & $\in \begin{pmatrix} \pm 2.0021795094256335 \times 10^{-16}\times [-1,1] \\ 0.414_{40857864989611}^{26997851057934} \end{pmatrix}$ \\[2mm]
\hline 
\end{tabular}%
\label{table_osc_diff}
\caption{Detailed data for Computer Assisted Result \ref{car_osc_diff}}
\end{table}

Comparing the oscillating finite traveling wave $\phi^{osc, c_2}$ with a compacton traveling wave $\phi^{c_1}$ in Computer Assisted Result \ref{car_compacton_diff}, we have the inequality $c_1 < c_2$.
Using the fact, define $\Phi^{c_1,c_2} = \Phi^{c_1,c_2}(t,x)$ be a function given as
\begin{equation}
\label{composite-family-2}
\Phi^{c_1,c_2}(t,x) := \phi^{c_1}_{[w_{{\min},1}, w_{{\max},1}]}(x-c_1 t) + \phi^{osc, c_2}_{w_{{\min},2}}(x-c_2 t)
\end{equation}
such that $w_{{\max},1} < w_{{\min},2}$.
Then it holds that the supports ${\rm supp} \phi^{c_1}_{[w_{{\min},1}, w_{{\max},1}]}$ and ${\rm supp} \phi^{osc, c_2}_{w_{{\min},2}}$ are disjoint from each other for all $t\geq 0$.
In other words, the function $\Phi^{c_1,c_2}$ is also a weak solution of (\ref{demo1}).

\begin{cor}[Composite wave of compacton and oscillatory wave]
\label{cor_composite2_diff}
Consider (\ref{demo1}) with $m=3/4$, $p=1/4$ and $a=0.3$.
Then the function $u(t,x) = \Phi^{c_1,c_2}(t,x)$ given in (\ref{composite-family-2}) is a weak solution of (\ref{demo1}).
\end{cor}

Note that the new solution obtained in the above corollary is {\em not} a traveling wave solution in the present sense, since each wave component propagates with different speed.

\begin{rem}
The construction of composite waves by compactons and finite traveling waves in the above argument generates a symbolic sequence $\{c_i\}_{i=1}^N$ of wave components up to translations of supports for each wave component.
The sequence has a restriction $c_1 \leq c_2 \leq \cdots \leq c_N$ for generating composite waves so that every two supports are disjoint, which indicates that composite waves should be parameterized monotonously.
\end{rem}

\subsubsection{Composite wave 3 : \lq\lq Chaos" in the family of traveling waves}
\label{section-chaos}
As an application of finite traveling wave validations with computer assistance, we can construct new type of weak solutions, some of which are shown in the previous subsection.
As a corollary, we can construct a family of traveling waves with \lq\lq chaotic" structure if a special pair of compactons is validated.
For example, consider the following degenerate diffusion equation with (artificial) quintic nonlinearity:
\begin{equation}
\label{demo2}
u_t = \frac{1}{m+1}(u^{m+1})_{xx} + u^p(1-u^2)(u^2-a^2),\quad m>0,\ p>0\text{ with }m+p=1,
\end{equation}
where $a\in (0,1)$ is a parameter.
\begin{equation}
\label{demo2-main}
\begin{cases}
\dot \phi = \psi, &\\
\dot \psi = -c\psi - \phi(1-\phi^2)(\phi^2-a^2),
\end{cases}\quad \dot {} = \frac{d}{d\tau}
\end{equation}

At first, we easily know the following symmetric property for solutions of (\ref{demo2-main}), which helps our present argument but is not essential to the \lq\lq chaotic" structure of waves we shall discuss later.
\begin{lem}
\label{lem-symmetry}
Assume that $(\phi(\tau), \psi(\tau))$ is a solution of (\ref{demo2-main}) with the speed $c\in \mathbb{R}$ and the initial data $(\phi_0, \psi_0)$.
Then the function $(-\phi(\tau), -\psi(\tau))$ is a solution of (\ref{demo2-main}) with the same speed $c$ and the initial data $(-\phi_0, -\psi_0)$.
\end{lem}

The same strategy as previous validation examples yields the following numerical validation result.
\begin{car}
\label{car_compacton_diff_chaotic}
Consider (\ref{demo2}) with $m=3/4$, $p=1/4$ and $a=0.3$.
Then there exists a value $c_\ast\in 5.0\times 10^{-7}\times [-1,1]$ such that the function $u(t,x)=\phi(x-c_\ast t)$ given by
\begin{equation*}
\phi(x-c_\ast t) \equiv \phi(\xi) = \begin{cases}
\varphi(\xi) & \xi \in (\xi_{\min}, \xi_{\max}),\\
0 & \xi \in \mathbb{R}\setminus (\xi_{\min}, \xi_{\max})
\end{cases}
\end{equation*}
is a compacton traveling wave solution with speed $c_\ast$, where $\{\phi(\tau)\}$ is a homoclinic orbit\footnote
{
The corresponding covering relation of the form $N\overset{\varphi(\bar \tau,\cdot)}{\Longrightarrow}M$ as in Figure \ref{fig-compacton} holds with $\bar \tau = 44.253$.
} of $(\phi, \psi) = (0,0)$ contained in the isolating block $N = PB$, where $P = [V_1\ V_2]$ is the eigenmatrix given in 
Table 4, and
\begin{align*}
B &= ([-1.0000000013268460,+1.0001776533279528]\times 10^{-5})\\
	&\quad \quad \times ([--1.0177653015013772,+1.0000001326843545]\times 10^{-7}).
\end{align*}
Finally, $\xi_{\min}$ and $\xi_{\max}$ can be determined as finite values satisfying
\begin{equation*}
\xi_{\max} - \xi_{\min} \equiv |{\rm supp}u| \in [6.984582800445637, 7.20846810204].
\end{equation*}
\end{car}

\begin{table}[ht]
\centering
\begin{tabular}{cc}
\hline 
 & At $(0,0),\ c\in 2.0\times 10^{-6}\times [-1,1]$ \\
\hline\\[-2mm]
$(M_u, M_s)$ for cone condition & $(7000,500)$ \\[1mm]
$\ell_u, \ell_s$ for extended cones & $(0.002, 0.001)$ \\[1mm]
Eigenvalue $\lambda_1$ & $\in 0.3000_{2861643465067}^{3138356514092}$  \\[1mm]
Eigenvalue $\lambda_2$ & $\in -0.3000_{3214083434215}^{2785917455782}$ \\[1mm]
Eigenvector $V_1$ & $\in \begin{pmatrix} 0.95792_{122415048908}^{291154895415} \\ 0.28737_{526087354431}^{797983612196} \end{pmatrix}$ \\[2mm]
Eigenvector $V_2$ & $\in \begin{pmatrix} -0.95792_{291154868769}^{122415055381} \\ 0.28737_{526087240495}^{797983793328} \end{pmatrix}$ \\[2mm]
$\tan \theta$ in (\ref{theta-bound}) with $V_1$ &  $\in +0._{29999831657524578}^{30000168342754563}$ \\[2mm]
$\tan \theta$ in (\ref{theta-bound}) with $V_2$ &  $\in -0._{30000168342941619}^{29999831657413978}$ \\[2mm]
\hline 
\end{tabular}%
\label{table_compacton_diff_chaotic}
\caption{Detailed data for Computer Assisted Result \ref{car_compacton_diff_chaotic}}
\end{table}

Now let $\phi$ be the validated compacton traveling wave with speed $c$.
Then Lemma \ref{lem-symmetry} indicates that $-\phi$ is also an compacton traveling wave with speed $c$.
The key point we mention here is that {\em we have two different compacton traveling wave solutions with identical speed}, which shall be written as $\{\phi_+, \phi_-\}$.
Proposition \ref{prop-superpose} indicate that, if ${\rm supp}\phi_+ \cap {\rm supp}\phi_- = \emptyset$ for a certain time $t$, then this property holds for all time and hence the superposed functions
\begin{equation*}
\Phi_{\bf w(+,-)}(t,x) := \phi_{+,[w_{\min, 1}, w_{\max, 1}]}(x-ct) + \phi_{-,[w_{\min, 2}, w_{\max, 2}]}(x-ct)
\end{equation*}
and
\begin{equation*}
\Phi_{\bf w(-,+)}(t,x) := \phi_{-,[w_{\min, 1}, w_{\max, 1}]}(x-ct) + \phi_{+,[w_{\min, 2}, w_{\max, 2}]}(x-ct)
\end{equation*}
are finite traveling wave solutions, where $[w_{\min, i}, w_{\max, i}]$ is the support of compacton with $w_{\max 1} < w_{\min, 2}$.
Similarly, functions defined as $\Phi_{\bf w(+,+)}(t,x)$ and $\Phi_{\bf w(-,-)}(t,x)$ are also finite traveling waves, and more number of compacton components can be superposed in the similar manner.
Consequently, up to translations of supports of compactons, we obtain a family of traveling wave solutions labelled by the following:
\begin{align*}
\{+\},\quad \{-\},\quad \{+, +\},\quad \{+, -\},\quad \{-, +\},\quad \{-, -\},\quad \{+, +, +\},\quad \cdots.
\end{align*}
In particular, we obtain traveling wave solutions consisting of countable number of compactons.
The most significance is that the family contains traveling waves whose profiles reflect the {\em full two-shift $\sigma : \{\pm \}^{\mathbb{Z}} \to \{\pm \}^{\mathbb{Z}}$} under translation of frame coordinates, which possesses the following properties and is well-known as {\em symbolic chaos} (e.g., \cite{HSD2012}):
\begin{itemize}
\item For any $N\in \mathbb{N}$, there are $2^N$ periodic symbolic sequences;
\item The set of all periodic points for $\sigma$ is dense in $\{\pm \}^{\mathbb{Z}}$ with the topology given by, say
\begin{equation*}
d({\bf s}, {\bf t}) = \sum_{j\in \mathbb{Z}}\frac{|s_j - t_j|}{2^{|j|}},\quad {\bf s} = \{s_j\}_{j\in \mathbb{Z}}, {\bf t} = \{t_j\}_{j\in \mathbb{Z}}\in \{\pm \}^{\mathbb{Z}},
\end{equation*}
identifying $+$ and $-$ with $1$ and $0$, respectively;
\item $\sigma$ is topologically transitive on $\{\pm \}^{\mathbb{Z}}$;
\item $\sigma$ has a sensitive dependence on initial conditions.
\end{itemize}

\begin{thm}
The degenerate diffusion equation (\ref{demo2}) admits a family of traveling wave solutions $\mathscr{S}$ in a class of weak solutions in the sense of Definition \ref{dfn-weak} such that
\begin{itemize}
\item $\mathscr{S}$ consists of finite traveling wave solutions of the following form:
\begin{equation*}
\Phi_{{\bf w}_N({\bf s}_N)}(t, x) := \sum_{i=1}^N \phi_{w_i(s_i)} (x-ct),
\end{equation*}
where ${\bf s}_N \equiv \{s_1, s_2, \cdots s_N\}\in \{\pm \}^N$ for each $N\in \mathbb{N}$, as well as 
\begin{equation*}
\Phi_{{\bf w}({\bf s})}(t, x) := \sum_{n\in \mathbb{Z}} \phi_{w_n(s_n)} (x-ct),
\end{equation*}
where ${\bf s} \equiv \{s_n\}_{n\in \mathbb{Z}} \in \{\pm \}^\mathbb{Z}$.
\end{itemize}
In particular, the set $\{\Phi_{{\bf w}({\bf s})}\}_{{\bf s}\in \{\pm \}^\mathbb{Z}}$ contains, up to translation of ${\bf w}({\bf s})$, the symbolic chaos in the following sense.
Set the countable sequence of intervals $\mathcal{I}_{\bf w} = \{{\bf w}({\bf s}) = \{w_{n'}(s_{n'})\}_{n'\in \mathbb{Z}}\mid {\bf s}\in \{\pm \}^\mathbb{Z}\}$, 
so that
\begin{itemize}
\item there is a positive number $\xi_0 > 0$ such that ${\rm supp }(w_{n_2}(s_{n_2})) = \{(n_2-n_1)\xi_0\} + {\rm supp }(w_{n_1}(s_{n_1}))$ for all $n_1,n_2\in \mathbb{Z}$ and ${\bf s}\in \{\pm \}^\mathbb{Z}$
and that the intervals $\{w_{n'}(s_{n'})\}_{n'\in \mathbb{Z}}$ are mutually disjoint for each ${\bf s}\in \{\pm \}^\mathbb{Z}$;
\item the measure of supports of intervals $w_{n'}(s_{n'})$ is identical among $n'\in \mathbb{Z}$;
\item the signature of $i$-th wave component $\phi_{w_i(s_i)}$ is $s_i$.
\end{itemize}
Then $\{\mathcal{I}_{\bf w},\sigma\}$ with 
\begin{equation*}
{\bf w}({\bf s}) \mapsto \sigma({\bf w}({\bf s})) = \tilde {\bf w}({\bf s}) \equiv \{\tilde w_{n'}'(s_{n'})\}_{n'\in \mathbb{Z}},\quad \tilde w_{n'}'(s_{n'}) := w_{n'-1}(s_{n'-1})
\end{equation*}
has a structure of full $2$-shift.
Consequently, the pair $\left( \{ \Phi_{\bf W}\}_{{\bf W}\in \mathcal{I}_{\bf w} }, \tilde \sigma \right)$ with
\begin{equation*}
\tilde \sigma \left( \Phi_{\bf W} \right) = \Phi_{\sigma({\bf W})}
\end{equation*}
has a structure of full $2$-shift with respect to the metric
\begin{equation*}
D(\Phi_{{\bf w}({\bf s})}, \Phi_{{\bf w}({\bf t})}) := \sum_{j\in \mathbb{Z}} 2^{-|j|} \sup_{\xi\in \mathbb{R}} |\phi_{w_j(s_j)} (\xi) - \phi_{w_j(t_j)} (\xi)|.
\end{equation*}
\end{thm}

\begin{proof}
The nontrivial statement is reduced to whether $D$ is a metric on the set $\mathscr{S}_{\mathcal{I}_{\bf w}} \equiv \{ \Phi_{\bf W}\}_{{\bf W}\in \mathcal{I}_{\bf w} }$.
Since each wave component $\phi_{w_n(s_n)}$ is a (continuous) compacton with identical profile up to translation, the property of $\mathcal{I}_{\bf w}$ shows that each distance $\sup_{\xi\in \mathbb{R}} |\phi_{w_j(s_j)} (\xi) - \phi_{w_j(t_j)} (\xi)|$ is either zero or a positive constant $C>0$ for any $n\in \mathbb{Z}$.
In particular, the value $D(\Phi_{{\bf w}({\bf s})}, \Phi_{{\bf w}({\bf t})})$ is determined as a finite and nonnegative value for any pairs in $\mathscr{S}_{\mathcal{I}_{\bf w}}$.
What we have to show are the following three parts:
\begin{itemize}
\item $D(\Phi_{{\bf W}_1}, \Phi_{{\bf W}_2}) = 0$ implies $\Phi_{{\bf W}_1} = \Phi_{{\bf W}_2}$.
\item $D(\Phi_{{\bf W}_1}, \Phi_{{\bf W}_2}) = \Phi_{{\bf W}_2}, \Phi_{{\bf W}_1}$ for any pairs in $\mathscr{S}_{\mathcal{I}_{\bf w}}$.
\item $D(\Phi_{{\bf W}_1}, \Phi_{{\bf W}_3}) \leq D(\Phi_{{\bf W}_1}, \Phi_{{\bf W}_2}) + D(\Phi_{{\bf W}_2}, \Phi_{{\bf W}_3})$ holds for any triples 
$(\Phi_{{\bf W}_1}, \Phi_{{\bf W}_2}, \Phi_{{\bf W}_3})\subset \mathscr{S}_{\mathcal{I}_{\bf w}}$.
\end{itemize}
The second and the third immediately follow from the property of $\sup$ function.
The rest is the first. 
Assume that $D(\Phi_{{\bf W}_1}, \Phi_{{\bf W}_2}) = 0$.
Since each component in the sum is nonnegative, it requires that $\sup_{\xi\in \mathbb{R}} |\phi_{w_j(s_j)} (\xi) - \phi_{w_j(t_j)} (\xi)| = 0$ for all $j$, which holds if and only if $s_j = t_j$ for all $j\in \mathbb{Z}$.
Therefore ${\bf W}_1 = {\bf W}_2$ and hence $\Phi_{{\bf W}_1} = \Phi_{{\bf W}_2}$ holds.
\end{proof}

The above theorem obviously gives a topological conjugacy $H : \{\mathcal{I}_{\bf w},\sigma\}\to \{\mathscr{S}_{\mathcal{I}_{\bf w}},\tilde \sigma\}$, which is defined by $H({\bf W}) = \{\Phi_{\bf W}\}$ as well as $H^{-1}(\{\Phi_{{\bf W}'}\}) = {\bf W}'$.
The shift $\tilde \sigma$ corresponds to translation of wave components and, if $c > 0$, it is realized by the movement of waves via time-$(\xi_0 / c)$ evolution.
The above argument shows that this \lq\lq chaotic" structure is induced by at least two different compacton traveling wave solutions with an identical speed.
Existence of these compacton wave solutions are nontrivial in general and our present study shows a potential of computer assisted analysis of  finite-time singularities for revealing very rich solution structure in degenerate systems.
Similarly, if we validate $p$ different compacton traveling waves for a certain system, then the system will possess a family of traveling waves reflecting the structure of {\em full $p$-shift} $\sigma_p : \{1,\cdots, p\}^{\mathbb{Z}} \to \{1,\cdots, p\}^{\mathbb{Z}}$.

\subsubsection{Technical Details}

In practical computations, we set the following parameters, which are specific for usage of CAPD library \cite{CAPD, ZLoh}.
\begin{itemize}
\item Order of Taylor expansion in Lohner's method : $12$
\item Grid size $dt$ for integration of ODEs : $0.001$
\end{itemize}
As for integration of ODEs, we divide the initial set (the exit of isolating block) into $20$ small pieces in $c$ (parameter)-direction in each validation.
In 
Table 5, execution times for our validated computations in the present environment are listed.

\begin{table}[ht]
\centering
\begin{tabular}{c|c}
\hline 
Computer Assisted Result  & Execution Time \\
\hline\\[-2mm]
\ref{car_FTW_diff}, covering relation & 0.082 sec. \\[1mm]
\ref{car_FTW_diff}, time integration & 1 min. 26.667 sec.  \\[1mm]
\ref{car_compacton_diff}, covering relation & 0.097 sec. \\[1mm]
\ref{car_compacton_diff}, time integration & 19 min. 5.070 sec. \\[1mm]
\ref{car_osc_diff}, covering relation & 0.133 sec. \\[1mm]
\ref{car_compacton_diff_chaotic}, covering relation & 0.126 sec. \\[1mm]
\ref{car_compacton_diff_chaotic}, time integration & 39 min. 56.205 sec. \\[1mm]
\hline 
\end{tabular}%
\label{table_diff_time}
\caption{Execution times for validating results in Section \ref{section-diff}}
\end{table}

\subsection{Singular canards with passage time through folded singularities}
We move to demonstrating the other type of finite-time singularities arising in fast-slow systems (\ref{fast-slow-param}).
In particular, we focus on the following $3$-dimensional autocatalysis system (e.g., \cite{MS2001}):
\begin{align}
\notag
\frac{da}{dt} &= \mu\left(\frac{5}{2}+c\right) - ab^2 - a \\
\label{autocatalysis}
\epsilon \frac{db}{dt} &= ab^2 + a - b \\
\notag
\frac{dc}{dt} &= b-c.
\end{align}
The variables $a,b$ and $c$ represent concentrations of three different chemical species depending on time $t$.
In particular, nonnegative values of $a,b$ and $c$ are considered.
The parameter $\epsilon \geq 0$ is typically assumed to be sufficiently small, but at present we pay attention to the singular limit case, namely $\epsilon = 0$.
$\mu$ is a parameter. 
In \cite{MS2001}, this is considered as a bifurcation parameter, but such a property is out of our present argument.
\par
Our main issue here is dynamics around folded singularities, thus we formally set $\epsilon = 0$ in (\ref{autocatalysis}), which yields the restriction of our considerations to the manifold
\begin{equation*}
S_0 = \{(a,b,c)\mid ab^2 + a - b = 0\},
\end{equation*}
which shall be called the {\em critical manifold}.
In the present case, the variable $a$ is explicitly represented by $b$ on $S_0$:
\begin{equation*}
a = \frac{b}{b^2+1}.
\end{equation*}
Then the (critical) dynamics on $S_0$ is dominated by
\begin{equation*}
\frac{1-b^2}{(1+b^2)^2}\frac{db}{dt} = \mu\left(\frac{5}{2}+c\right) - b, \quad
\frac{dc}{dt} = b-c,
\end{equation*}
equivalently
\begin{equation}
\label{reduced-FS}
(1-b^2) \frac{db}{dt} = (1+b^2)^2\left\{\mu\left(\frac{5}{2}+c\right) - b\right\}, \quad
\frac{dc}{dt} = b-c.
\end{equation}
The system becomes degenerate\footnote{
Actually the system has a degeneracy at $b=\pm 1$, but our focus is nonnegative values of variables due to the realistic chemical setting, and hence we only pay attention to $b=1$.
} at $b=1$.
We then desingularize the system via
\begin{equation}
\label{time-desing-autocatalysis}
\frac{dt}{d\tau} = 1-b^2
\end{equation}
to obtain the following desingularized system:
\begin{equation}
\label{desingularize-FS}
\frac{db}{d\tau} = (1+b^2)^2\left\{\mu\left(\frac{5}{2}+c\right) - b\right\}, \quad
\frac{dc}{d\tau} = (b-c)(1-b^2).
\end{equation}

\begin{rem}
Correspondence to notations in Section \ref{section-folded} is given as follows:
\begin{itemize}
\item $(w,y,z) = (b, a, c)$.
\item $f = ab^2+a-b$, $g_1 = \mu(\frac{5}{2} + c)-ab^2-a$, $g_2=b-c$.
\item $f_w = b^2-1$, $f_y = -(b^2+1)$, $f_z = \mu$.
\item $h^y(w,z) \equiv h(w,z) \equiv h(b,c) = b(b^2+1)^{-1}$.
\item $h^y_w \equiv h^y_b = (1-b^2)(1+b^2)^{-2}$, $h^y_z \equiv h^y_c = 0$.
\end{itemize}
The time-scale desingularization (\ref{time-desing-autocatalysis}) is actually different from $d\tau/dt = h^y_w$ mentioned in Section \ref{section-folded}, but leads to orbital equivalence since $1+b^2 > 0$.
\par
The line $F=\{b=1\}$ on $\{f=0\} = \{a = b(b^2+1)^{-1}\}$ is a generic folded line since we have
\begin{equation*}
f_w = b^2-1 = 0,\quad f_{ww} \equiv f_{bb} = 2\not = 0,\quad D_{y,z}f \equiv (f_a, f_c)^T = (2,0)^T,
\end{equation*}
where all requirements in (\ref{FS-generic}) are satisfied.
\end{rem}

\subsection{Validation of singular canard}
Observe that (\ref{desingularize-FS}) has the following equilibria:
\begin{equation*}
\left(\frac{5\mu}{2(1-\mu)}, \frac{5\mu}{2(1-\mu)}\right),\quad \left(1, \frac{1}{\mu}-\frac{5}{2}\right)\equiv p_0.
\end{equation*}
The former is an equilibrium in the full-scale system (\ref{autocatalysis}) on $S_0$, whereas the latter is {\em not}.
Our main interest here is dynamics on $S_0$ around the latter. 
Since $p_0$ is on the generic folded line $F$ and hence it is a folded singularity.
Vector field around $p_0$ with $\mu = 0.2$ is shown in Figure \ref{fig_canard}.
As an example, rigorous numerics followed by Algorithm \ref{alg-finite} directly yields the following result.

\begin{car}
\label{car_canard}
Consider the desingularized system (\ref{desingularize-FS}) with $\mu = 0.2$.
Let $\varphi_{\mu}^{des}$ be the generated flow with $\mu$. 
Then the following statements hold.
\begin{enumerate}
\item This system has an equilibrium $p_0 = (1, 2.5)$ which is saddle and locally unique.
In other words, $p_0$ is an folded saddle.
$p_0$ is contained in the isolating block $N = PB + \{p_0\}\subset \mathbb{R}^2$ for (\ref{desingularize-FS}), where $P = [V_1\ V_2]$ is the eigenmatrix given in 
Table 6 and
\begin{align*}
B &= ([-1.0000182335639729,+1.0000182335639729]\times 10^{-6})\\
	&\quad \quad \times ([-1.0000026639762002,+1.0000026639762002]\times 10^{-6}).
\end{align*}
\item Let $N_1 =0.6929_{48}^{66}\times \{2.8\}$ and $N_2 = 1.5_{7995}^{8005}\times \{2.3\}$.
Then there are points $p_i\in N_i$, $i=1,2$, such that $p_i \in W^s(p_0; \varphi_{\mu}^{des})$.
\item Let $\{\gamma_i(\tau)\}_{\tau\geq 0}\subset W^s(p_0; \varphi_{\mu}^{des})$, $i=1,2$, be trajectories with $\gamma_i(0) = p_i$.
Let also $b_{\gamma_i}(\tau)$ be the $b$-component of the trajectory $\gamma_i$ at $\tau$.
Then we have the following estimates:
\begin{align*}
t_{\max,1} &\equiv \int_0^\infty \{1-b_{\gamma_1}(\tau)^2\}d\tau \in 0.16_{589950513}^{647064143},\quad
t_{\max,2} \equiv \int_0^\infty \{b_{\gamma_2}(\tau)^2-1\}d\tau \in 0.17_{503780357}^{660511191}.
\end{align*}
\end{enumerate}
\end{car}

\begin{proof}
We only show the concrete estimate for bounds of $t_{\max}$.
First we have the coordinate $(y_a, y_b)$ via $(b,c)^T = (1,2.5)^T + P(y_a, y_b)^T$.
The corresponding Lyapunov function around $p_0$ is given by
$L_s(y_a, y_b) = \|\sigma^s(y_b)\|^2 + \|y_b\|^2$.
There are two cases for validating solutions.
\begin{description}
\item[Case 1]  $b\leq 1$.
\end{description}
In this case we have
\begin{align*}
(1-b^2) = (1-b)(1+b)  \leq 2(1-b) = -2(p_{11}y_a + p_{12}y_b) \leq 4{\bf p}_1L_s,
\end{align*}
where ${\bf p}_1 = \max\{|p_{11}|, |p_{12}|\}$, and hence, assuming $b_{\gamma_1}(0)\in W^s(p_0)\cap N$ with $b_{\gamma_1}(0) < 1$,
\begin{align*}
\int_0^\infty \{1-b(\tau)^2\} d\tau &\leq 4{\bf p}_1 \int_{L_{s,0}}^0 L_s(x(\tau;x_0)) \frac{d\tau}{dL_s}dL_s\\
	&\leq -4{\bf p}_1\frac{M^2+1}{2M^2 \overrightarrow{\mu_{s,W^s}}}\int_0^{L_{s,0}} L_s^{1-1} dL_s\\
	&= -2{\bf p}_1\frac{M^2+1}{M^2 \overrightarrow{\mu_{s,W^s}}} L_{s,0} \equiv T_{M_s, p_0}^1.
\end{align*}
\begin{description}
\item[Case 2]  $b\geq 1$.
\end{description}
Note that, in the present case, the sign of vector field is reversed compared with that of the original vector field.
Therefore \lq\lq the maximal existence time" is calculated in the time-reverse flow
\begin{equation*}
t_{\max,2}\equiv \int_{-\infty}^0 \{1- b(-\tau)^2\}d(-\tau) =  \int_0^{\infty} \{b(\tau)^2-1\}d\tau
\end{equation*}
as long as $b \geq 1$ holds along trajectories.
Now we have
\begin{align*}
&(b^2-1) = (b-1)(b-1 + 2) = (b-1)^2 + 2(b-1),\quad
2(b-1) = 2(p_{11}y_a + p_{12}y_b) \leq 4{\bf p}_1L_s
\end{align*}
to obtain
\begin{equation*}
(b^2-1) \leq 4{\bf p}_1L_s (1+{\bf p}_1L_s).
\end{equation*}
Therefore, assuming $b_{\gamma_2}(0)\in W^s(p_0)\cap N$ with $b_{\gamma_2}(0) > 1$,
\begin{align*}
\int_0^\infty \{b(\tau)^2-1\} d\tau &\leq 4{\bf p}_1 \int_{L_{s,0}}^0 L_s(x(\tau;x_0)) (1+{\bf p}_1 L_s(x(\tau;x_0))) \frac{d\tau}{dL_s}dL_s\\
	&\leq -4{\bf p}_1\frac{M^2+1}{2M^2 \overrightarrow{\mu_{s,W^s}}}\int_0^{L_{s,0}} (1+ {\bf p}_1 L_s) dL_s\\
	&= -2{\bf p}_1\frac{M^2+1}{M^2 \overrightarrow{\mu_{s,W^s}}} L_{s,0} \left( 1+ \frac{1}{2}{\bf p}_1 L_{s,0} \right) \equiv T_{M_s, p_0}^2.
\end{align*}
\end{proof}

Two validated trajectories describe the stable manifold of $p_0$ for desingularized system (\ref{desingularize-FS}), as drawn in Figure \ref{fig_FS}-(a).
Going back to the reduced system (\ref{reduced-FS}), the evolution direction of $\gamma_2$ becomes opposite, namely the trajectory in $t$-timescale runs {\em from $p_0$ to $p_2$}.
Moreover, the fact that $p_0$ not being an equilibrium of (\ref{reduced-FS}) implies that the trajectory $\gamma_1$ with $\gamma_1(0) = p_1$ arrives at $p_0$ at the time $t = t_{\max, 1}$ and passes $p_0$ with nonzero speed (the speed in $b$-component is infinitely fast !).

\begin{cor}[Validation of singular canards]
Let $p_1, p_2$ be points validated in Computer Assisted Result \ref{car_canard}.
Then the union $\Gamma_\tau \equiv \bigcup_{i=1}^2 \{\gamma_i(\tau)\}_{\tau\geq 0}$ in the $t$-timescale, which is denoted by $\Gamma_t$, is a singular canard.
In the $t$-timescale, $\Gamma_t$ starting at $p_1$ at $t=0$ passes the folded singularity $p_0$ within the time $t_{pass} \in 0.16_{589950513}^{647064143}$, and arrives at $p_2$ within the arrival time $t_{arrival} \in 0.34_{09373087}^{307575334}$.
\end{cor}

\begin{proof}
Computer Assisted Result \ref{car_canard} with the following argument shows that the trajectory $\Gamma_t$ is either a singular faux canard or a singular canard.
The estimate of arrival time immediately follows from the validation result.
The rest is to check the stability of the critical manifold $S_0$, which is determined by the sign of $f_b$ on $S_0$.
We immediately have
\begin{equation*}
f_b|_{S_0} = (2ab-1)|_{S_0} = 2\frac{b^2}{b^2+1} - 1 = \frac{b^2-1}{b^2+1},
\end{equation*}
which is negative for $b < 1$ and positive for $b>1$.
This result shows that $\gamma_1$ is on the attracting branch of $S_0$ and $\gamma_2$ is on the repelling branch of $S_0$.
Therefore the trajectory $\Gamma_t$ runs from the attracting branch to the repelling branch, which shows that $\Gamma_t$ is a singular canard. 
\end{proof}

The existence of singular canard is already discussed in \cite{MS2001}.
The present argument shows that our methodology also provides a numerical validation methodology for singular (faux) canards with explicit passage times through folded singularities.

\begin{rem}
In Computer Assisted Result \ref{car_canard}, estimations of $t_{\max,i}$ are applied with a few generalizations. 
Firstly the folded singularity is located not at the origin but at a bounded point, which can be realized by translations.
Secondly, $T(b,c)^{-1} = 1-b^2$ becomes negative for $b > 1$.
In this case, the time direction of the desingularized flow in $b > 1$ is reversed compared with the original flow.
Nevertheless, the same arguments as the case $T(b,c)>0$, namely $b<1$, can be applied to the case $b<1$ as long as the sign of $T$ is identical along validated trajectories.
\end{rem}

\begin{rem}
In the present example, several notions such as location of critical manifolds including folded lines, their attracting and repelling branches and folded singularities can be handled by hands. 
Treatments of singular canards are necessary to validate these notions and properties in advance, possibly with computer assistance, in which case rigorous numerics of higher differentials of vector fields as well as those of their eigenpairs will be necessary to validate requirements argued in Section \ref{section-procedure}.
They will be beyond our main aims and hence we leave such an advanced arguments in future works.
\end{rem}

\begin{table}[ht]
\centering
\begin{tabular}{cc}
\hline 
 & At $p_0$ \\
\hline\\[-2mm]
$M_s$ for cone condition & $200$ \\[1mm]
$\ell_s$ for extended cones & $0.0002$ \\[1mm]
Eigenvalue $\lambda_1$ & $\in 0.5298751_{0279054661}^{1790448638}$  \\[1mm]
Eigenvalue $\lambda_2$ & $\in -4.5302751_{166655018}^{040295286}$ \\[1mm]
Eigenvector $V_1$ & $\in \begin{pmatrix} 0.1739333_{5699235576}^{6098705264} \\ 0.9848589_{6784199634}^{7808890144} \end{pmatrix}$ \\[2mm]
Eigenvector $V_2$ & $\in \begin{pmatrix} 0.83381870_{005810610}^{990889451} \\ -0.552219_{50484416682}^{49703803452} \end{pmatrix}$ \\[2mm]
$\tan \theta$ in (\ref{theta-bound}) with $V_2$ &  $\in -0.6622776_{6876142802}^{5157533058}$ \\[2mm]
\hline 
\end{tabular}%
\label{table_canard}
\caption{Detailed data for Computer Assisted Result \ref{car_canard}}
\end{table}

\begin{figure}[htbp]
  \begin{center}
    \includegraphics[width=7cm]{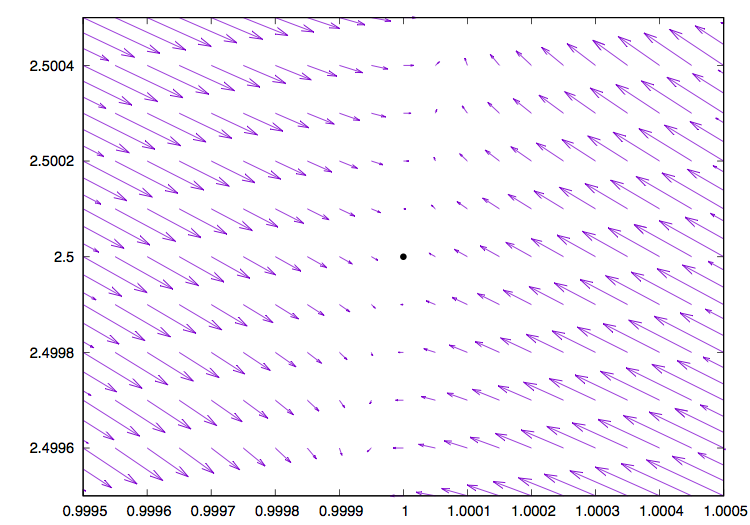}
    \caption{Vector field of desingularized vector field (\ref{desingularize-FS}) near $p_0$ with $\mu = 0.2$}
    \label{fig_canard}
\end{center}
The point $p_0$ is a saddle equilibrium for (\ref{desingularize-FS}); namely an folded saddle for (\ref{reduced-FS}).
Length of each arrow represents the strength of vector field (\ref{desingularize-FS}) at each point.
The corresponding vector field (\ref{reduced-FS}) becomes drastically different from that of (\ref{desingularize-FS}) shown herein.
Firstly, $p_0$ is not an equilibrium of (\ref{reduced-FS}). 
Secondly, amplitude of the $b$-component of vector field is arbitrarily large near $b=1$ (that is why we cannot draw the vector field of (\ref{reduced-FS}) in the similar manner), which implies that a trajectory can run across $p_0$ with extremely fast speed.
Thirdly, for $b > 1$, all arrows are reversed both in $(b,c)$-direction.
\end{figure}

\subsubsection{Technical Details}

In practical computations, we set the following parameters, which are specific for usage of CAPD library \cite{CAPD, ZLoh}.
\begin{itemize}
\item Order of Taylor expansion in Lohner's method : $12$
\item Grid size $dt$ for integration of ODEs : $0.001$
\end{itemize}
As for integration of ODEs, we divide the initial sets ($N_1$ and $N_2$ in Computer Assisted Result \ref{car_canard}) into $20$ small pieces in $b$-direction in each validation.
In 
Table 7, execution times for our validated computations in the present environment are listed.

\begin{table}[ht]
\centering
\begin{tabular}{c|c}
\hline 
Integration & Execution Time \\
\hline\\[-2mm]
From $N_1$, covering relation & 0.237 sec. \\[1mm]
From $N_1$, time integration & 1 min. 16.413 sec.  \\[1mm]
From $N_2$, covering relation & 0.324 sec. \\[1mm]
From $N_2$, time integration & 2 min. 0.395 sec.  \\[1mm]
\hline 
\end{tabular}%
\label{table_canard_time}
\caption{Execution times for validating results in Computer Assisted Result \ref{car_canard}}
\end{table}

\section*{Conclusion}
In this paper, we have proposed a rigorous computation procedure for trajectories of dynamical systems possessing finite-time singularity.
The basic idea consists of detection of points $p_0$ inducing degeneracy ({\em degeneracy-inducing points}), desingularization of vector fields so that it is regular including $p_0$, numerical validation of trajectories on stable manifolds of (hyperbolic) equilibria corresponding to $p_0$ and estimation of arrival (or extinction) time at $p_0$.
The present process is the same as validation of blow-up solutions \cite{TMSTMO, MT2017} except  detection of $p_0$.
We have newly introduced a validation procedure of {\em Lyapunov functions on stable manifolds}, which is a generalized version of \cite{MHY2016}.
This Lyapunov function enables us to validate the maximal existence time of trajectories asymptotic to equilibria {\em even of saddle type}, including the case discussed in \cite{TMSTMO, MT2017}.
The evaluation of arrival time provides a universal aspect and common approach to finite-time singularities with rigorous numerics, including blow-up phenomena (e.g., \cite{TMSTMO, MT2017}), finite traveling waves, compactons, extinction in degenerate diffusion system and canard points.
\par
As an applicability of our present method, we have demonstrated rigorous numerics of several finite traveling waves for degenerate diffusion equation as well as composite weak solutions.
In particular, a compacton traveling wave can induce, up to translation symmetry, infinitely many (finite) traveling wave solutions.
Moreover, we have proved that the family of traveling waves for degenerate diffusion systems can admit chaotic structure if the desingularized system admits more than one homoclinic solutions (namely, compacton traveling waves) with an identical speed.
As a consequence, degenerate differential equations can have very rich solution structures from the viewpoint of dynamical systems, and our present method reveals such structure with standard rigorous computation methodology.
\par
We have also shown an example of singular canard validation for (\ref{fast-slow-param})$_{\epsilon = 0}$ through folded singularities.
It turns out that canard-induced solutions can be also categorized as a kind of finite-time singularity similar to finite-time extinction and blow-ups from the viewpoint of rigorous numerics, since all these solutions can be validated by a methodology with common mechanism (desingularization, global solution for desingularized system and estimate of maximal existence time).
\par
\bigskip
There are plenty of preceding mathematical studies that solutions with finite-time singularities trigger very rich solution structure of systems, although their mathematical or numerical validations remain nontrivial depending on systems (e.g., \cite{DGKKOW2012, K1992, MS2001}). 
Our present methodology will be thus a fundamental to studying solution structures of systems with \lq\lq singular" natures with computer assistance from both quantitative and qualitative viewpoints.

\section*{Acknowledgements}
The author was partially supported by Program for Promoting the reform of national universities (Kyushu University), Ministry of Education, Culture, Sports, Science and Technology (MEXT), Japan, World Premier International Research Center Initiative (WPI), MEXT, Japan, and JSPS Grant-in-Aid for Young Scientists (B) (No. JP17K14235).

\appendix
%
%

\section{Construction of isolating blocks around equilibria via rigorous numerics}
\label{section-block}

This appendix collects a basic procedure of constructing isolating blocks with rigorous numerics so that readers who are not familiar with treatments of isolating blocks can easily access.
Detailed arguments are shown in \cite{ZM2001, Mat2}.

\subsection{A basic form}
\label{section-block-basic}
There is a preceding work for the systematic construction of isolating blocks around equilibria \cite{ZM2001}. 
One will see that such procedures are very suitable for analyzing dynamics around equilibria including degeneracy-inducing points.

\bigskip
Let $K \subset \mathbb{R}^k$ be a compact, connected and simply connected set. 
Consider first the differential equation of the following abstract form: 
\begin{equation}
\label{abstract-layer}
x' \equiv \frac{dx}{dt} = f(x, \mu),\quad f:\mathbb{R}^n\times \mathbb{R}^k\to \mathbb{R}^n.
\end{equation}
For simplicity, assume that $f$ is $C^\infty$. 
Let $\varphi_\mu$ the flow of (\ref{abstract-layer}) with fixed $\mu\in \mathbb{R}^k$.
Our purpose here is to construct an isolating block which contains an equilibrium of (\ref{abstract-layer}).

\bigskip
Let $x_0$ be a numerical equilibrium of (\ref{abstract-layer}) at $\mu_0 \in K$ and rewrite (\ref{abstract-layer}) as a series around $(x_0,\mu_0)$:
\begin{equation}
\label{Taylor}
x' = f_x(x_0,\mu_0)(x-x_0) + \hat f(x,\lambda),
\end{equation}
where $f_x(x_0,\mu_0)$ is the Fr\'{e}chet differential of $f$ with respect to $x$-variable at $(x_0,\mu_0)$. $\hat f(x,\lambda)$ denotes the higher order term of $f$ with $O(|x-x_0|^2 + |\mu-\mu_0|)$. This term may in general contain an additional term arising from the numerical error $f(x_0,\mu_0)\approx 0$.

Here assume that the $n\times n$-matrix $f_x(x_0,\mu_0)$ is nonsingular. Diagonalizing $f_x(x_0,\mu_0)$, which is generically possible, (\ref{Taylor}) is further rewritten by the following perturbed diagonal system around $(x_0,\mu_0)$:
\begin{equation}
\label{ode-y-coord}
y_j' = \lambda_j y_j + \tilde f_j(y,\mu),\quad j=1,\cdots, n.
\end{equation}
Here $\lambda_j\in \mathbb{C}$, $y = (y_1,\cdots, y_n)$ and $\tilde f (y,\mu) = (\tilde f_1(y,\mu), \cdots, \tilde f_n(y,\mu))^T$ are defined by $x = Py + x_0$ and $\tilde f(x,\mu) = P(\hat f(y,\mu))$, where $P=(P_{ij})_{i,j=1,\cdots, n}$ is a nonsingular matrix diagonalizing $f_x(x_0,\mu_0)$ and $\ast^T$ is the transpose.
Let $N\subset \mathbb{R}^n$ be a compact set containing $x_0$.
Assume that each $\tilde f_j(y,\mu)$ has a bound $[\delta_j^-, \delta_j^+]$ in $N\times K$, namely,
\begin{equation*}
\label{error-bounds}
\left\{ \tilde f_j(y,\lambda) \mid x = Py+x_0\in N, \lambda \in K\right\}\subsetneq [\delta_j^-, \delta_j^+].
\end{equation*}
Then $y_j'$ must satisfy 
\begin{equation*}
\lambda_j \left(y_j + \frac{\delta_j^-}{\lambda_j}\right) < y_j' < \lambda_j \left(y_j + \frac{\delta_j^+}{\lambda_j}\right),\quad \forall y \text{ with }x=Py+x_0\in N,\ \forall \lambda \in K.
\end{equation*}
For simplicity we assume that each $\lambda_j$ is real. 
We then obtain the candidate of an isolating block $B$ in $y$-coordinate given by the following:
\begin{align}
\label{block-exit}
B:= \prod_{j=1}^n B_j,\quad B_j = [y_j^-, y_j^+] &:= \left[-\frac{\delta_j^+}{\lambda_j}, -\frac{\delta_j^-}{\lambda_j} \right]\quad \text{ if }\lambda_j > 0,\\
\label{block-entrance}
B_j = [y_j^-, y_j^+] &:= \left[-\frac{\delta_j^-}{\lambda_j}, -\frac{\delta_j^+}{\lambda_j} \right]\quad \text{ if }\lambda_j < 0.
\end{align}

In the case that $\lambda_j$ is complex-valued for some $i$, $f_x(x_0,\lambda_0)$ contains the complex conjugate of $\lambda_j$ as the other eigenvalue. 
Without the loss of generality, we may assume $\lambda_j = \alpha_j + \sqrt{-1}\beta_j$, $\lambda_{j+1} = \bar \lambda_j = \alpha_j - \sqrt{-1}\beta_j$, $\beta_j \not = 0$. 
To be simplified, we further assume that $\lambda_j$ and $\lambda_{j+1}$ are the only complex pair of eigenvalues of $f_x(x_0,\lambda_0)$. 
The general case can be handled in the same manner.
The dynamics for $y_j$ and $z_{j+1}$ is formally written by
\begin{align*}
y_j' &= \lambda_j y_j + \tilde f_j(y,\mu),\\
y_{j+1}' &= \lambda_{j+1} y_{j+1} + \tilde f_{j+1}(y,\mu).
\end{align*}
Now we would like to consider real dynamical systems. To do this we transform the above form into 
\begin{align*}
w_j' &= \alpha_j w_j + \beta_j w_{j+1} + \bar f_j(w,\mu),\\
w_{j+1}' &= -\beta_j w_j + \alpha_j w_{j+1} + \bar f_{j+1}(w,\mu)
\end{align*}
via $Q = \begin{pmatrix} 1 & 1 \\ \sqrt{-1} & -\sqrt{-1} \end{pmatrix}$, $(w_j,w_{j+1})^T=Q(y_j,y_{j+1})^T$ and $(\bar f_j,\bar f_{j+1})^T=Q(\tilde f_j,\tilde f_{j+1})^T$, where $w=(w_1,\cdots, w_n)$ is the new coordinate satisfying $w_i= y_i$ for $i\not = j,j+1$. 
Let $r_j(w,\mu):= \sqrt{\bar f_j(w,\mu)^2 + \bar f_{j+1}(w,\mu)^2}$ and assume that $r_j(w,\mu)$ is bounded by a positive number $\bar r_j$ uniformly on $N\times K$. 
Our aim here is to construct a candidate of isolating block and hence we assume that
{\em the scalar product of the vector field and the coordinate vector
\begin{equation*}
(w_j,w_{j+1}) \cdot (\alpha_j w_j + \beta_j w_{j+1} + \bar f_j(w,\mu), -\beta_j w_j + \alpha_j w_{j+1} + \bar f_{j+1}(w,\mu))
\end{equation*}
has the identical sign and the above function never attain $0$ on $\left\{(w_j,w_{j+1}) \mid \sqrt{w_j^2 + w_{j+1}^2} \leq b_j \right\}$ ($b_j > 0$)}.
With this assumption in mind, we set the candidate of isolating block in $i$-th and $(i+1)$-th coordinate
\begin{equation*}
B_{j,j+1}:= \left\{(w_j,w_{j+1}) \mid \sqrt{w_j^2 + w_{j+1}^2} \leq \frac{\bar r_j}{|\alpha_j|} \right\}.
\end{equation*}
Its boundary becomes exit if $\alpha_j > 0$ and entrance if $\alpha_j < 0$. Finally, replace $B_j\times B_{j+1}$ in the definition of $B$ ((\ref{block-exit}) and (\ref{block-entrance})) by $B_{j,j+1}$.
\par
\bigskip
A series of estimates for error terms involves $N$ and it only makes sense if it is self-consistent, namely, $\{p_0\}+PB\subset N$.
If it is the case, then $B$ is desiring isolating block for (\ref{abstract-layer}). Indeed, if $\lambda_j > 0$, then
\begin{equation*}
y_j'\mid_{y_j = y_j^-} < 0\quad \text{ and }\quad y_j'\mid_{y_j = y_j^+} > 0
\end{equation*}
hold. Namely, the set $\{y\in B \mid y_j = y_j^{\pm}\}$ is contained in the exit. Similarly if $\lambda_j < 0$ then 
\begin{equation*}
y_j'\mid_{y_j = y_j^-} > 0\quad \text{ and }\quad y_j'\mid_{y_j = y_j^+} < 0
\end{equation*}
hold. Namely, the set $\{y\in B \mid y_j = y_j^{\pm}\}$ is contained in the entrance. Obviously $\partial B$ is the union of the closed exit and the entrance, which shows that $B$ is an isolating block for $\varphi_\mu$ for all $\mu\in K$.
Equivalently, the set $B\times K$ is an isolating block for the parameterized flow $\Phi : \mathbb{R}\times \mathbb{R}^n\times \mathbb{R}^k \to \mathbb{R}^n\times \mathbb{R}^k$ given by
$\Phi(t,x,\mu) = (\varphi_\mu(t,x), \mu)$.
\par
Once such an isolating block $B$ is constructed, one obtains an equilibrium in $B$. 
\begin{prop}[cf. \cite{ZM2001}]\rm
\label{prop-existence-fixpt}
Let $B$ be an isolating block constructed as above. Then $B$ contains an equilibrium of (\ref{abstract-layer}) for all $\mu \in K$. 
\end{prop}
This proposition is the consequence of general theory of the {\em Conley index} or the theory of {\em mapping degree}. 
Note that the construction of isolating blocks stated in Proposition \ref{prop-existence-fixpt} around points which are {\em not necessarily equilibria} implies the existence of {\em rigorous} equilibria inside blocks.

%
%
\subsection{The predictor-corrector approach}
\label{section-block-pred-corr}

Here we provide another approach for validating isolating blocks.
In previous subsection, isolating blocks are constructed centered at $\{(\bar x, \mu)\mid \mu\in K\}$, where $K\subset \mathbb{R}^k$ is a (small) compact neighborhood of $\bar \mu \in K$. 
All transformations concerning eigenpairs are done at a point $(\bar x, \bar \mu)$.
On the other hand, we can reselect the center of the candidate of blocks so that blocks can be chosen smaller.
As continuations of equilibria with respect to parameters, {\em the predictor-corrector approach} is one of effective approaches.
We now revisit the construction of fast-saddle type blocks with the predictor-corrector approach.

Let $(\bar x, \bar \mu)$ be a (numerical) equilibrium for (\ref{abstract-layer}), i.e., $f(\bar x, \bar \mu)\approx 0$, such that $f_x(\bar x, \bar \mu)$ is invertible.
Let $K$ be a compact neighborhood of $\bar \mu$ as in the previous case.
The central idea is to choose the center as follows instead of $(\bar x, \bar \mu)$:
\begin{equation}
\label{new-center}
\left(\bar x + \frac{dx}{d\mu}(\bar \mu)(\mu-\bar \mu),\mu\right)\equiv \left(\bar x - f_x(\bar x, \bar \mu)^{-1}f_\mu(\bar x, \bar \mu)(\mu-\bar \mu),\mu\right),
\end{equation}
where $x = x(\mu)$ is the parametrization of $x$ with respect to $\mu$ such that $\bar x = x(\bar \mu)$ and that $f(x(\mu),\mu) = 0$, which is actually realized in a small neighborhood of $\bar \mu$ in $\mathbb{R}^k$ since $f_x(\bar x, \bar \mu)$ is invertible. 
Obviously, the identification in (\ref{new-center}) makes sense thanks to the Implicit Function Theorem.

Around the new center, we define the new affine transformation $T : (y,\nu)\mapsto (x,\mu)$ as
\begin{equation}
(x,\mu) = T(y,\nu) := \left(Py + \bar x - f_x(\bar x, \bar \mu)^{-1}f_\mu(\bar x, \bar \mu)\nu, \nu + \bar \mu\right).
\end{equation}
where $P$ is a nonsingular matrix diagonalizing $f_x(\bar x, \bar \nu)$.
Over the new $(y,\nu)$-coordinate, the fast system $x' = f(x,\mu)$ is transformed into the following:
\begin{align*}
y' &= P^{-1}\left(x' + \overline{f_x}^{-1}\overline{f_\mu}\nu' \right)\\
 	&= P^{-1}f(x,\mu) \\
	&= P^{-1}\left( \overline{f_x}(Py - \overline{f_x}^{-1}\overline{f_\mu} \nu) + \hat f(y,\nu)\right)\\
	&= \Lambda y + P^{-1}\left(-\overline{f_\mu}\nu + \hat f(y,\nu) \right)\\
	&\equiv  \Lambda y + F(y,\nu),
\end{align*}
where $\overline{f_x} = f_x(\bar x, \bar \mu)$ and $\overline{f_\mu}=f_\mu(\bar x, \bar \mu)$,  and $\Lambda = \diag(\lambda_1,\cdots, \lambda_n)$.
The function $\hat f(y, \nu)$ denotes the higher order term of $f$ with $O(|y|^2 + |\nu|)$.
Dividing $y$ into $(a,b)$ corresponding to eigenvalues with positive real parts and negative real parts, respectively, as in (\ref{ode-y-coord}), we can construct a candidate set of isolating blocks as in (\ref{block-exit}) and (\ref{block-entrance}).
\par
Note that the higher order term $\hat f(y,\nu)$ contains the linear term of $w$ as $\overline{f_\mu}\nu$ with small errors in a sufficiently small neighborhood $K$ of $\bar \mu$.
This fact indicates that, in principle, size of blocks becomes smaller than those in Section \ref{section-block-basic}.

\bibliographystyle{plain}
\bibliography{extinction}

\end{document}